\documentclass[a4paper,10pt,twoside]{article}
\usepackage[latin1]{inputenc} 
\usepackage{graphicx}
\usepackage[pdftex,bookmarks]{hyperref}
\usepackage{amsmath}    
\usepackage{amssymb}
\usepackage{amsmath,amsthm}
\usepackage{textcomp}
\usepackage{color}
\usepackage[all]{xy}
\usepackage{slashed}
\usepackage{latexsym}

\usepackage{mathrsfs}

\usepackage[titletoc,toc,title]{appendix}

\usepackage{fancyhdr}
\numberwithin{equation}{section}
\newtheorem{theorem}{Theorem}[section]
\newtheorem{lemma}[theorem]{Lemma}
\newtheorem{proposition}[theorem]{Proposition}
\newtheorem{corollary}[theorem]{Corollary}

\theoremstyle{definition}
\newtheorem{definition}[theorem]{Definition}

\newtheorem{examples}[theorem]{Exemples}
\newtheorem{remark}[theorem]{Remark}
\theoremstyle{plain}

\def\KK{\ensuremath\mathbb{K}}

\def\CC{\ensuremath{\mathbb{C}}}

\def\RR{\ensuremath{\mathbb{R}}}
\def\ZZ{\ensuremath{\mathbb{Z}}}
\date{}
\begin{document}
\pagestyle{plain}
	\title{Adiabatic groupoid and  secondary invariants in K-theory}
	
	\author
	{Vito Felice Zenobi}
	\maketitle

		\begin{abstract}
		In this paper we define new K-theoretic secondary invariants attached to a Lie groupoid $G$.
		The receptacle for these invariants is the K-theory of $C^*_r(G_{ad}^\circ)$ (where $G_{ad}^\circ$ is the adiabatic deformation $G$ restricted to  the interval $[0,1)$). Our construction directly generalises  the cases treated  in \cite{PS,PS2}, in the setting of the Coarse Geometry, to more involved geometrical situations, such as foliations.
		Moreover we tackle the problem of producing a wrong-way functoriality between adiabatic deformation groupoid K-groups associated to transverse maps. This extends the construction of the lower shriek map in \cite{CS}.
		Furthermore we prove a Lie groupoid version of  the Delocalized APS Index Theorem of Piazza and Schick.
		Finally we give a product formula for secondary invariants.
		
		\end{abstract}

Mathematical Subjects: 22A22, 	46L80, 	19K56
\section*{Introduction}
Higher secondary invariants have recently been the subject of a significant number of papers: they were introduced by Higson and Roe in the seminal works \cite{HigRoeI,HigRoeII,HigRoeIII} about mapping surgery to analysis; then they were treated in a more index-theoretic way by Piazza and Schick \cite{PS, PS2} with applications to the Stolz' positive scalar curvature sequence and the surgery exact sequence for smooth manifolds; the author of this paper extended this last construction to the topological setting in \cite{zenobi}; all these works use as principal tool the Coarse C*-algebras introduced by Roe. In \cite{XY, XY2,WXY} these subjects are treated using the localization C*-algebras introduced by Yu. Other interesting works in the same area are  \cite{BR, benameur-roy-ell2,DG1, DG2, DG3,zeidler}.
The approach to secondary invariants in the present work   is given through Lie groupoids.

 Let us start by setting a general framework, see also \cite[Section 3]{Schick}.
We begin by considering an exact sequence of C*-algebras of the following type 
\begin{equation}\label{eq1intro}\xymatrix{0\ar[r]& S\otimes B\ar[r]& E\ar[r]& A\ar[r]& 0}\end{equation}
where $S:= C_0(0,1)$.
We can investigate  the following hierarchy of K-theory classes:
\begin{itemize}
	\item a fundamental class $[D]\in K_*(A)$;
	\item a primary invariant given by the index class $\partial[D]\in K_{*+1}(S\otimes B)$,
where
	$\partial\colon K_*(A)\to K_{*+1}(S\otimes B)$ is the boundary map for the
	long exact sequence in K-theory;
	
	\item assume that the primary invariant is the zero class  and that 
	we know the ``reason'' $w$ why it is zero.   Then we can use that reason to establish a rule for constructing a canonical lift of  $[D]$ in $K_*(E)$, that
	we are going to call a secondary invariant and that we will denote by $\varrho(D,w)$.
	
\end{itemize}

Now let us assume that the exact sequence \eqref{eq1intro}
		has a completely positive section. This implies that $\partial$ is an element in $KK(A,B)$  and one can prove that
		there exists a C*-algebra $A'$ and two morphisms
	\begin{itemize}
		\item $\psi\colon A'\to A$ which induces a KK-equivalence;
		\item$\varphi\colon A'\to B$ which induces the boundary map for the following exact sequence
	\end{itemize} 
	\begin{equation}\label{eq2intro}\xymatrix{0\ar[r]& S\otimes B\ar[r]& \mathcal{C}_\varphi(A',B)\ar[r]& A'\ar[r]& 0}.\end{equation}
	Notice that the long exact sequence in K-theory associated to \eqref{eq2intro} is isomorphic to the one induced by \eqref{eq1intro}.
Here $\mathcal{C}_\varphi(A',B)=\{a\oplus f\in A'\oplus B[0,1)\,|\, f(0)=\varphi(a)\}$ is the mapping cone C*-algebra associated to $\varphi$.
In this context a secondary invariant is a class in $K_0(\mathcal{C}_\varphi(A',B))$  and it is explicitly  represented by
		\begin{itemize}
	\item a projection $p$  over $A'$ defining the fundamental class $[D]$.
		\item a path $q_t$ of projections from $\varphi(p)$  to a degenerated projection over $B$, that concretely gives 
			the reason $w$ why the primary invariant  is zero.
		\end{itemize}		
On the other hand a secondary invariant, as a class in $K_1(\mathcal{C}_\varphi(A',B))$,   is explicitly  represented by
\begin{itemize}
	\item a unitary $u$  over $A'$ defining the fundamental class $[D]$.
	\item a path $v_t$ of unitaries from $\varphi(u)$  to the identity over $B$.
\end{itemize} 
\subsubsection*{The tangent groupoid and $\varrho$-classes}
Let us make all that more concrete in a simple geometric context.
Let $X$ be a closed smooth manifold.
Then consider the pair groupoid $X\times X\rightrightarrows X$. Its smooth convolution algebra $C_c^\infty(X\times X, \Omega^{\frac{1}{2}}(\ker dr\oplus\ker ds))$ of the smooth compactly supported half-densities on $X\times X$ is the *-algebra of the smoothing operators on $L^2(X)$ and $C^*_r(X\times X)$, its reduced C*-algebra, is isomorphic to the algebra of compact operators $\mathcal{K}(L^2(X))$.

The Lie algebroid of $X\times X$ is given by the tangent bundle $TX$, it is a Lie groupoid and, by means of the Fourier transform,  its groupoid
C*-algebra $C^*_r(TX)$ is isomorphic to $C_0(T^*X)$ (notice that  0-order symbols on $X$ are bounded multipliers of this algebra).
By Poincaré duality, see \cite{CS}, we know that  $K_*(C_0(T^*X))$ is isomorphic to $KK_*(C(X),\CC)$, the K-homology of $X$.

So, following the abstract construction given in the previous subsection, we have that $K_*(C^*_r(TX))$ is the receptacle of the fundamental classes and  the analytical index $$\mathrm{Ind}\colon K_*(C_0(T^*X))\to K_*(\KK(L^2(X)))\cong\ZZ$$ gives the primary invariants.
But we would like to have a realization of $\mathrm{Ind}$ as an element of KK-theory or, better, as the boundary map of a semi-split exact sequence as in \eqref{eq2intro}.

Indeed the  following construction of  Connes gives the solution to this problem:
the tangent groupoid of the smooth manifold $X$ is defined as follows
\[
\mathbb{T}X:=TX\times\{0\}\sqcup X\times X\times(0,1]\rightrightarrows X\times[0,1],
\]
 equipped with a suitable smooth structure. It is a deformation groupoid, whose restriction at $0$ is $TX$ and whose restriction at $1$ is $X\times X$.
One can prove that:
\begin{itemize}
	\item the evaluation at $0$, $\mathrm{ev}_0\colon C^*_r(\mathbb{T}X)\to C^*_r(TX)$, induces a KK-equivalence since its kernel is a cone and then K-contractible. If $\sigma$ is an elliptic symbol of order $0$ on $X$, then the symbol $\sigma\times id_{[0,1]}$  on $T^*X\times[0,1]$, the dual Lie algebroid of $\mathbb{T}X$, gives an elliptic pseudodifferential operator on $\mathbb{T}X$, in the sense of \cite{vas}, whose restriction at $1$ is the pseudodifferential operator on $X$ associated to $\sigma$ and whose restriction at $0$ is the Fourier transform of $\sigma$;
	\item the KK-element $\mathrm{Ind}_X:=[\mathrm{ev}_0]^{-1}\otimes_{C^*_r(\mathbb{T}X)}[\mathrm{ev}_1]\in KK(C^*_r(TX),C^*_r(X\times X))$ gives the analytical index $\mathrm{Ind}$, where $\mathrm{ev}_1\colon C^*_r(\mathbb{T}X)\to C^*_r(X\times X)$ is the evaluation at 1, see \cite{MP} for a proof of this fact.
 \end{itemize}
Now let us point out that $\mathcal{C}_{\mathrm{ev}_1}(C^*_r(\mathbb{T}X), C^*_r(X\times X))$, the mapping cone C*-algebra of the evaluation at 1, is isomorphic to $C^*_r(\mathbb{T}^\circ X)$, where $\mathbb{T}^{\circ}X$ is the restriction of $\mathbb{T}X$ to the open interval $[0,1)$.

So we have that the analytical index is the boundary morphism of the long exact sequence of K-groups associated to the exact sequence
\begin{equation}
\label{eq3intro}\xymatrix{0\ar[r]& C^*_r(X\times X)\otimes C_0(0,1)\ar[r]& C^*_r(\mathbb{T}^\circ X)\ar[r]^{\mathrm{ev}_0}& C^*_r(TX)\ar[r]& 0}.
\end{equation}

If $\sigma$ is the principal symbol of an elliptic operator $P$ on $X$ and $P_t$ is a path of elliptic operators on $X$ such that $P_0=P$ and $P_1$ is invertible, then the analytical index of $P$ vanishes. There is a $\varrho$-invariant associated to this situation that we will denote by $\varrho(\sigma, P_t)\in K_*(C^*_r(\mathbb{T}^\circ X))$, see Section \ref{section-rho} for a detailed construction.

We have two typical geometric situations where the analytical index  vanishes.
\begin{itemize}
	\item Let $X= N\sqcup -M$ be the disjoint union of two compact oriented smooth manifolds, that are oriented homotopy equivalent through $f\colon N\to M$, and let $D^{sign}$ be the signature operator of $X$. In \cite{HilSk} the authors proved that there exists a canonical way to produce  a path of operators $D_t$  from $D^{sign}$ to an invertible operator $D_1$; all that (up to passing from the language of K-theory to the language of KK-theory and from the unbounded case to the bounded one) gives a fundamental class, the symbol of the signature operator, and a path to a degenerate cycle, that is a reason why the analytical index of the signature vanishes. As we saw before, all that gives a class $\varrho(f)=(\sigma(D^{sign}),D_t)$ in the  K-theory of $C^*_r(\mathbb{T}^\circ X)$.
	\item Let $X$ be a spin smooth compact manifold, equipped with a Riemannian metric $g$, such that the scalar curvature is positive everywhere.
	Then the Lichnerowicz formula implies that the Dirac operator $\slashed{D}$ associated to the spinor bundle is invertible and  that its analytical index is zero; so, as for the previous case,  
	the Dirac operator itself (no perturbations are needed here) gives a class $\varrho(g)=\varrho(\sigma(\slashed{D}),\slashed{D})$ in the  K-theory of $C^*_r(\mathbb{T}^\circ X)$.
\end{itemize}

\subsubsection*{Wrong-way functoriality}
Once we have constructed such a secondary invariant, we 
would like to study the functoriality of these objects with respect to smooth maps. In other words we would like to push forward classes from $N$ to $M$, through a smooth map $f\colon N\to M $, at the level of the tangent groupoids.
In \cite{CS} the authors construct a lower shriek map
$df_!\in KK^*(C_0(T^*N),C_0(T^*M))$, associated  to any smooth map $f\colon N\to M$ between compact smooth manifolds.
By Poincar\'{e} duality this homomorphism corresponds to the map $[f]\colon K_*(N)\to K_*(M)$ between the K-homology groups of the manifolds.
Thanks to the Poincaré duality and the naturality of the index we have the following equality of morphisms
$ K_*(C_0(T^*N))\to K_*(C^*_r(M\times M))$
\[
df_!\otimes\mathrm{Ind}_{M}=\mathrm{Ind}_{N}\otimes \mu_f
\]
where $\mu_f$ is the Morita equivalence between $C^*_r(N\times N)$ and $C^*_r(M\times M)$.
The problem here concerns the construction of the dotted arrow in the following diagram
\[
\xymatrix{\cdots\ar[r] & K_*\left(C^*_r(N\times N\times(0,1))\right)\ar[r]\ar[d]^{\mu_{f}\otimes id} &
	K_*\left(C^*_r(\mathbb{T}^\circ N))\right)\ar[r]^{[\mathrm{ev}_0]_*}\ar@{.>}[d]^{\psi_!^{ad}} & K_*\left(C_0(T^*N)\right)\ar[r]\ar[d]^{df_!}& \cdots\\
	\cdots\ar[r] &
	K_*\left(C^*_r(M\times M\times(0,1))\right)\ar[r] &
	K_*\left(C^*_r(\mathbb{T}^\circ M))\right)\ar[r]^{[\mathrm{ev}_0]_*} & K_*\left(C_0(T^*M)\right)\ar[r]& \cdots\\
}\]
 so that all the squares commute. This dotted arrow will be implemented by a suitable deformation groupoid, as we will see in Section \ref{section}.

\subsubsection*{Cobordisms}
A second problem is the following: a natural equivalence relation  among  homotopy equivalences and metrics with positive scalar curvature is given by a certain cobordism equivalence.  
The question is if the $\varrho$-classes are well defined on cobordism classes. 
A positive answer is given by the so-called Delocalized Atiyah-Patodi-Singer Index Theorem, firstly stated and proved by Piazza and Schick in \cite{PS,PS2} in the setting of the coarse geometry and in this paper formalized and generalised to the context of Lie groupoids.
In order to do it we need to use the  $b$-groupoid \[\Gamma(W,\partial W)= \mathring{W}\times\mathring{W}\sqcup \partial W\times\partial W\times \RR \rightrightarrows W\] of a manifold $W$ with boundary $\partial W$, see Section \ref{month}.

Let $P$ be an elliptic pseudodifferential operator such that its restriction to the boundary $P_\partial$ is homotopic to an invertible operator, through a path $P_\partial^t$. This implies that one can perturb $P$ to an opertor that has a Fredholm index in the K-theory of $C^*_r(\mathring{W}\times\mathring{W})$.
The deformation groupoid \[\Gamma(W,\partial W)_{ad}^\mathcal{F}\rightrightarrows \mathring{W}\times[0,1]\sqcup\partial W\times[0,1)\] encodes both the index of the perturbation of $P$
and the $\varrho$-class of the boundary so that, through convenient exact sequences, we can compare these two classes. Indeed one can state, roughly, that
the $\varrho$-class associated to the path $P_\partial^t$ is equal to the image of  the Fredholm index of $P$ into the K-theory of the tangent groupoid of $\mathring{W}$.

Hence if we have a spin Riemannian manifold $W$ that is a cobordism between $\partial_0 W$ and $\partial_1 W$ and if $W$ is equipped with a metric $g$ with positive scalar curvature that restricts to $g_0$ and $g_1$ on the boundary components, then $\varrho(g_0)=\varrho(g_1)$ since the Dirac operator of $W$ is invertible and its Fredholm index vanishes. 

Analogously if $W$ is a smooth cobordism between two manifolds $M_0$ and $M_1$ and if there is a homotopy equivalence $F\colon W\to N\times[0,1]$, such that its restrictions to the boundary components, $f_i\colon M_i\to N\times\{i\}$ for $i=0,1$, are homotopy equivalences, then $\varrho(f_0)=\varrho(f_1)$ since the index of the signature operator of $W\sqcup N\times [0,1]$ vanishes.
\subsubsection*{Products}
A last question concerns product formulas.
Let $g$ be a Riemannian metric with positive scalar curvature on a spin smooth compact manifold $Y$ and let $h$ be any Riemannian metric on a spin smooth compact manifold $V$. We know that, up to  multiplication by a scalar factor $\epsilon$ the metric $h$, $g\oplus h$ is a metric with positive scalar curvature on $Y\times V$.
On the other hand, if $f\colon N\to M$ is a homotopy equivalence between smooth manifolds, then so is  $f\times \mathrm{id}\colon N\times W\to M\times W$ for any smooth manifold $W$.
What is the relation between $\varrho(g)$ and $\varrho(g\times h)$? The same question arise for $\varrho(f)$ and $\varrho(f\times\mathrm{id})$.

Let $Z$ denote $Y$ or $N\sqcup -M$ and let $X$ denote $V$ or $W$.
One can define a product 
\[
\boxtimes\colon K_i(C^*_r(\mathbb{T}^\circ Z))\times K_j(X)\to K_{i+j}(\mathbb{T}^\circ (Z\times X))
\]
such that the following formulas holds:
\[
\varrho(g)\boxtimes[\slashed{D}_h]=\varrho(g\oplus h)
\]
where $[\slashed{D}_h]$ is the K-homology class of the Dirac operator on $(X,h)$;
\[
\varrho(f)\boxtimes[D^{sign}_X]=\varrho(f\times\mathrm{id})
\]
where $[D^{sign}_X]$ is the K-homology class of the signature operator of $X$.

After proving that, one can ask when two different $\varrho$-classes on the same manifold remain distinct after making the product with a second manifold. In this specific case the answer is that if the Fredholm index of the K-homology class on $X$ is non-zero, then the product with this K-homology class is rationally injective.

\subsubsection*{Lie groupoids}

So far we have been concerned with a very simple Lie groupoid on a smooth manifold $X$, namely the pair groupoid.
Let us consider the Poincaré groupoid of $X$, that is $\widetilde{X}\times_\Gamma\widetilde{X}\rightrightarrows X$ where $\Gamma$ is the fundamental group of $X$ and $\widetilde{X}$ is the universal covering of $X$.
Then we recover the results obtained by Piazza and Schick in \cite{PS,PS2}, up to consider groupoid C*-algebras instead of Coarse algebras.

A nice feature of the theory we summarily explained above is that, if we take any Lie groupoid $G$ over 
 $X$, \emph{mutatis mutandis} and with little extra work, all the results hold in that generality.
 Where we used the tangent groupoid, we now employ the adiabatic deformation groupoid $G_{ad}$, see Definition \ref{adgroupoid}.
 The wrong-way functoriality generalizes between the adiabatic deformation of a Lie groupoid and the adiabatic deformation of its pull-back, see Subsection \ref{pb}.
Cobordism relations and product formulas are established in this general context.
The main examples are always given by homotopy equivalences
	and positive scalar curvature, in a suitable groupoid fashion.
	A concrete non-trivial example is given by foliations: can we distinguish cobordism classes of foliations homotopically equivalent to a given one? Can we distinguish cobordism classes of foliated metrics that have longitudinally positive scalar curvature? We will see an example in Section \ref{foliated-bundles}.
	
	This paper is devoted to develop the program explained above and  to tackle the technical issues one meets in generalising it  to the context of a general Lie groupoid.
	
	\subsubsection*{Acknowledgements}
	This paper is based on part of my PhD thesis.
	I am very thankful to Georges Skandalis for his enthusiastic  support and for sharing enlightening ideas with me.
	I would like to thank Paolo Piazza too, for his continuous support and for interesting discussions.
It is also a pleasure to thank the two anonymous referees who read this work extremely carefully and contribute to improve significantly the paper.
	\tableofcontents 
\section{Groupoids}
\subsection{Basics}

We refer the reader to \cite{DL} and the bibliography inside it for the notations and a detailed overview of groupoids and index theory.

\begin{definition} Let $G$ and $G^{(0)}$ be two sets.  A groupoid structure on $G$ over $G^{(0)}$ is given by the following morphisms:
	\begin{itemize}
	
		\item Two maps: $r,s: G\rightarrow G^{(0)}$,
		which are respectively the range and  source map.
			\item A map $u:G^{(0)}\rightarrow G$ called the unit map that is a section for both $s$ and $r$. We can identify $G^{(0)}$ with its image
			in $G$. 
		\item An involution: $ i: G\rightarrow G
		$, $  \gamma  \mapsto \gamma^{-1} $ called the inverse
		map. It satisfies: $s\circ i=r$.
		\item A map $ p: G^{(2)}  \rightarrow  G
		$, $ (\gamma_1,\gamma_2)  \mapsto  \gamma_1\cdot \gamma_2 $
		called the product, where the set 
		$$G^{(2)}:=\{(\gamma_1,\gamma_2)\in G\times G \ \vert \
		s(\gamma_1)=r(\gamma_2)\}$$ is the set of composable pairs. Moreover for $(\gamma_1,\gamma_2)\in
		G^{(2)}$ we have $r(\gamma_1\cdot \gamma_2)=r(\gamma_1)$ and $s(\gamma_1\cdot \gamma_2)=s(\gamma_2)$.
	\end{itemize}
	
	The following properties must be fulfilled:
	\begin{itemize}
		\item The product is associative: for any $\gamma_1,\
		\gamma_2,\ \gamma_3$ in $G$ such that $s(\gamma_1)=r(\gamma_2)$ and
		$s(\gamma_2)=r(\gamma_3)$ the following equality
		holds $$(\gamma_1\cdot \gamma_2)\cdot \gamma_3= \gamma_1\cdot
		(\gamma_2\cdot \gamma_3)\ .$$
		\item For any $\gamma$ in $G$: $r(\gamma)\cdot
		\gamma=\gamma\cdot s(\gamma)=\gamma$ and $\gamma\cdot
		\gamma^{-1}=r(\gamma)$.
	\end{itemize}
	
	We denote a groupoid structure on $G$ over $G^{(0)}$ by
	$G\rightrightarrows G^{(0)}$,  where the arrows stand for the source
	and target maps. 
\end{definition}

We will adopt the following notations: $$G_A:=
s^{-1}(A)\ ,\ G^B=r^{-1}(B) \ \mbox{ and } G_A^B=G_A\cap G^B \,$$
in particular if $x\in G^{(0)}$, the  $s$-fiber (respectively 
$r$-fiber) of $G$ over $x$ is $G_x=s^{-1}(x)$ (respectively $G^x=r^{-1}(x)$).
\begin{definition}
	A subset $X$ of $G^{(0)}$ is called $G$-invariant or saturated if for any element $x\in X$ we have that $r(G_x)$, or equivalently $s(G^x)$, is contained in $X$.
\end{definition}
\begin{definition}
	
	We call $G$ a Lie groupoid when $G$ and $G^{(0)}$ are second-countable smooth manifolds
	with $G^{(0)}$ Hausdorff, the structural homomorphisms are smooth and the range and the source maps are submersions.
\end{definition}

\subsection{Groupoid C*-algebras}

We can associate to a Lie groupoid $G$ the *-algebra $C^\infty_c(G,\Omega^{\frac{1}{2}}(\ker ds\oplus\ker dr))$ of the compactly supported sections of the half-densities bundle associated to $\ker ds\oplus\ker dr$, with:

\begin{itemize}
	\item the involution given by $f^*(\gamma)=\overline{f(\gamma^{-1})}$;
	\item and the convolution product given by $f*g(\gamma)=\int_{G_{s(\gamma)}} f(\gamma\eta^{-1})g(\eta)$.
\end{itemize}  

For all $x\in G^{(0)}$ the algebra $C^\infty_c(G,\Omega^{\frac{1}{2}}(\ker ds\oplus\ker dr))$ can be represented on 
$L^2(G_x,\Omega^{\frac{1}{2}}(G_x))$ by 
\[\lambda_x(f)\xi(\gamma)=\int_{G_{x}} f(\gamma\eta^{-1})\xi(\eta), \]
where $f\in C^\infty_c(G,\Omega^{\frac{1}{2}}(\ker ds\oplus\ker dr))$ and $\xi\in L^2(G_x,\Omega^{\frac{1}{2}}(G_x))$.

\begin{definition}
	The reduced C*-algebra of a Lie groupoid G, denoted by $C^*_r(G)$, is the completion of $C^\infty_c(G,\Omega^{\frac{1}{2}}(\ker ds\oplus\ker dr))$ with respect to the norm
	\[
	||f||_r=\sup_{x\in G^{(0)}}||\lambda_x(f)||.
	\]
	
	The full C*-algebra of $G$ is the completion of 
	$C^\infty_c(G,\Omega^{\frac{1}{2}}(\ker ds\oplus\ker dr))$ with respect to all bounded *-representations.
\end{definition}

\begin{remark}\label{fullvsred}
From now on, if $X$ is a $G$-invariant closed subset of $G^{(0)}$ we will call
$e_X\colon C^\infty_c(G)\to C^\infty_c(G_{|X})$ the restriction map to $X$.
That gives an exact sequence of full groupoid C*-algebras
\[
\xymatrix{0\ar[r]& C^*(G_{|G^{(0)}\setminus X})\ar[r]&C^*(G)\ar[r]& C^*(G_{|X})\ar[r]&0},
\]
but in general, for reduced C*-algebras, we have not exactness in the middle:  
the reader can find examples of this phenomenon in \cite{HLS}.
Let us precise that in what follows we will mainly deal with the reduced groupoid C*-algebras, because there are more details to check in the reduced situation. But everything we are going to  prove about the reduced C*-algebras works for the full C*-algebras, too.
\end{remark}

\begin{remark}\label{multiplication}
	Notice that elements of the algebra $C_b(G^{(0)})$ of the bounded continuous functions on $G^{(0)}$ are multipliers of $C^*_r(G)$. Consider $f\in C(G^{(0)})$, then it acts on the left just as the multiplication operator by the continuous function $r^*f$  on $G$ and on the right as the multiplication operator by the continuous function $s^*f$.
\end{remark}
We refer the reader to \cite{Renault} as a classical reference about groupoid C*-algebras.
\subsection{Lie algebroids and the adiabatic groupoid}\label{adiabaticdef}

\begin{definition} A  Lie
	algebroid $\mathfrak{A} =(p:\mathfrak{A}\rightarrow TM,[\ ,\ ]_{\mathfrak{A}})$ on a smooth
	manifold $M$ is a vector bundle $\mathfrak{A} \rightarrow M$
	equipped with a bracket $[\ ,\ ]_{\mathfrak{A}}:\Gamma(\mathfrak{A})\times \Gamma(\mathfrak{A})
	\rightarrow \Gamma(\mathfrak{A})$ on the module of sections of $\mathfrak{A}$, together
	with a homomorphism of vector bundles $p:\mathfrak{A} \rightarrow TM$ from $\mathfrak{A}$ to the
	tangent bundle $TM$ of $M$, called the  anchor map, fulfilling the following conditions:
	\begin{itemize}
		\item the bracket $[\ ,\ ]_{\mathfrak{A}}$ is $\RR$-bilinear, antisymmetric
		and satisfies the Jacobi identity,
		\item $[X,fY]_{\mathfrak{A}}=f[X,Y]_{\mathfrak{A}}+p(X)(f)Y$ for all $X,\ Y \in
		\Gamma(\mathfrak{A})$ and $f$ a smooth function of $M$, 
		\item $p([X,Y]_{\mathfrak{A}})=[p(X),p(Y)]$ for all
		$X,\ Y \in \Gamma(\mathfrak{A})$. 
	\end{itemize}
\end{definition}

Let $G$ be a Lie groupoid.
The  tangent space to $s$-fibers, that is $T_sG := \ker ds$ restricted to the objects of $G$ is
$\bigcup_{x\in G^{(0)}} TG_x$  and it has the structure of  Lie algebroid
on $G^{(0)}$, with the anchor map given by $dr$.  See for instance \cite{McK}.
It is denoted by 
$\mathfrak{A}(G)$ and we call it the Lie algebroid of $G$.
We can also think of it as the normal bundle of the inclusion $G^{(0)}\hookrightarrow G$. 

Let $M_0$ be a smooth submanifold of a smooth manifold $M$ with normal bundle $\mathcal{N}$.
We give now the definition of the so-called \emph{deformation to the normal cone}:
as a set, the deformation to the normal cone is $$D(M_0,M)=\mathcal{N} \times\{0\}\sqcup M\times(0,1]. $$
In order to define its smooth structure, we fix an exponential map, which is a diffeomorphism $\theta$ from
a neighbourhood $V'$ of the zero section $M_0$ in $\mathcal{N}$ to a neighbourhood $V$ of $M_0$ in $M$.
We may cover $D(M_0,M)$ with two open sets $M\times (0,1]$, with the product structure, and $W=\mathcal{N}\times \{0\}\sqcup V\times(0,1]$,
endowed with the smooth structure for which the map 
\begin{equation}\label{adiabatictopology}
\Psi\colon\{(m,\xi,t)\in \mathcal{N}\times[0,1]\,|\,(m,t\xi)\in V'\}\to W
\end{equation}
given by $(m,\xi,t)\mapsto(\theta(m, t\xi),t)$,
for $t\neq0$, and by $(m,\xi,0)\mapsto(m,\xi,0)$, for $t=0$, is a diffeomorphism. One can verify that the transition map on the overlap of these two open sets is smooth, see for instance \cite{HilSk2}.


\begin{definition}\label{adgroupoid}
The adiabatic groupoid $G_{ad}$ is given by $D(G^{(0)},G)$, the deformation to the normal cone of the unit map. As set it is the  following 
\[
\mathfrak{A}(G)\times\{0\}\cup G\times(0,1]\rightrightarrows G^{(0)}\times [0,1],
\]
with the smooth structure given by the 
construction discussed above.
\end{definition}

\begin{definition}
	We will use the notation $G_{ad}^{\circ}$ for the restriction of the adiabatic groupoid to the interval open at $1$, given by
	\[
	\mathfrak{A}(G)\times\{0\}\cup G\times(0,1)\rightrightarrows G^{(0)}\times [0,1).
	\]
\end{definition}

Then we can associate to a Lie groupoid $G$ a short exact sequence of C*-algebras
	\begin{equation}\label{AESmax}
	\xymatrix{0\ar[r] & C^*(G\times(0,1))\ar[r] & C^*(G_{ad}^\circ)\ar[r]^{\mathrm{ev}_0} & C^*(\mathfrak{A}(G))\ar[r]& 0}
	\end{equation}
	that we call the (full) adiabatic extension of $G$.

Since $\mathfrak{A}(G)$ is an amenable groupoid, $C^*(\mathfrak{A}(G))$ is isomorphic to the reduced groupoid C*-algebra $C^*_r(\mathfrak{A}(G))$.
Thanks to the fact that the map from the full C*-algebra of a groupoid to the reduced one is surjective, one can deduce that
the the following  sequence of reduced groupoid C*-algebras
	\begin{equation}\label{AES}
	\xymatrix{0\ar[r] & C_r^*(G\times(0,1))\ar[r] & C_r^*(G_{ad}^\circ)\ar[r]^{\mathrm{ev}_0} & C_r^*(\mathfrak{A}(G))\ar[r]& 0}
	\end{equation}
is exact.

\subsection{Manifolds with boundary and the Monthubert groupoid}\label{month}

For this section we refer the reader to \cite{Mont} and \cite[3.1]{CLM}.
Let $X$ be a manifold with boundary $\partial X$.
We can think of $X$ as a closed subspace of an open manifold $\hat{X}$.
Let $\rho\colon \hat{X}\to \RR$ be a defining function of the boundary, namely a function that is  zero on $\partial X$  and only there, with nowhere vanishing differential on it.
\begin{definition}\label{b-calculus}
	The $b$-groupoid (or equivalently the Monthubert groupoid) of $X$, denoted by $\Gamma(X,\partial X)$, is given as a set by
	\[
	\{(x,y,\alpha)\in X\times X\times \RR\,|\, \rho(x)=e^\alpha\rho(y)\}.
	\]
	For its smooth structure see \cite[Section 3.1]{Mont}.
\end{definition}

Notice that by \cite[Proposition 3.5]{Mont} the Lie groupoid $\Gamma(X,\partial X)$ is amenable and then $C^*(\Gamma(X,\partial X))=C_r^*(\Gamma(X,\partial X))$.

\begin{definition}\label{bcalculus}
	Let $G\rightrightarrows \hat{X}$  be a Lie groupoid and let $\partial X$ be transverse with respect to $G$, this means that  $T_x\partial X+dr(\mathfrak{A}_x(G))=T_x\hat{X}$.
	Define the $b$-groupoid of $G$  with respect to the pair $(X, \partial X)$, denoted by $G(X,\partial X)$, as the following fibered product
	\begin{equation}\label{fibered}
	\xymatrix{G(X,\partial X)\ar[r]\ar[d] & G\ar[d]^{r\times s}\\
		\Gamma(X,\partial X)\ar[r]^{r\times s} & X\times X}
	\end{equation}
	where $\Gamma(X,\partial X)$ is the  $b$-groupoid of $(X,\partial X)$, defined above.
	Then $G(X,\partial X)\rightrightarrows X$ is a longitudinally smooth groupoid.
	The set of arrows is $G_{|\mathring{X}}\cup G_{|\partial X}\times \RR$.
	See \cite[Section 3]{Mont} for a detailed construction.
\end{definition}
\emph{From now on, when we say that $G$ is a Lie groupoid on a manifold $X$ with boundary $\partial X$, we are implicitly saying that  $G$ is the restriction a Lie groupoid over an open manifold $\hat{X}$ which contains $X$ as a closed subspace.}

\begin{remark}\label{boundaryloc}
If we look the  groupoid $G(X,\partial X)$ near the boundary of $X$, we can give a product structure of it as follows.
Let $\underline{n}\in C^\infty(\partial X,T\hat{X}_{|\partial X})$ be a vector field that is normal to $\partial X$, such that $\langle d\rho, \underline{n}\rangle(x)=1$ for any $x\in \partial X$. Moreover, since $\partial X$ is transverse  with respect to $G$, 
 that is $T_x\partial X+dr(\mathfrak{A}_x(G))=T_x\hat{X}$, one can do all the choices so that $\underline{n}(x)$ belongs to the image of the anchor map of $\mathfrak{A}_xG$ for all $x\in \partial X$.

Then, one can chose $\exp\colon V\to X$ an exponential map, where $V$ is a suitable neighbourhood of the zero section in $T\hat{X}_{| \partial X}$ such that one has the following diffeomorphism
\[
\phi\colon (x,t)\mapsto \exp_x(t\underline{n})
\]
from $\partial X\times(-1,1)$ to a neighbourhood  $U$ of $\partial X$ in $\hat{X}$,  that gives an isomorphism of pair groupoids $\Phi\colon \partial X\times \partial X\times (-1,1)\times(-1,1)\to U\times U$.

Since the boundary is transverse with respect to $G$ we can choose $\underline{n}$ and a locally finite  open cover  $\{\mathcal{U}_i\}_{i\in I}$ of $\partial X$ such that $\underline{n}_{|\mathcal{U}_i}$ lifts to a local section $\xi_i$ of $\mathcal{A}G_{|\partial X}$. Then by means of a partition of the unity subordinated to $\{\mathcal{U}_i\}$ one can obtain a section $\xi$ of the Lie algebroid $\mathcal{A}G_{|\partial X}$.
Let 
$\gamma_{x,t}$ be the path in $G$ equal to $\exp_x(t\xi)$ for $t\in(-1,1)$.
Then the map
\[\Psi\colon (\gamma, t, s )\mapsto \gamma_{r(\gamma),t}\cdot\gamma\cdot\gamma_{s(\gamma), s}\]
is an isomorphism  of  Lie groupoids between $G_{|\partial X}\times(-1,1)\times(-1,1)$ and $G_{|U}$. See \cite[Section 2.1]{D-Sk-stability}.

In particular, since $\Psi$ and $\Phi$ are compatible with respect to the source and the target maps, we obtain the following isomorphism of Lie groupoids
\[
\Gamma([0,1),\{0\})\times G_{|\partial X}\cong G(X,\partial X)_{|U\cap X}.
\]
\end{remark}

In this context it is convenient to use a slight variation of the adiabatic groupoid.
\begin{definition}\label{nctangent}
	Let $G(X,\partial X)$ be as in Definition \ref{bcalculus} and denote $X\setminus \partial X$ by $\mathring{X}$:
	\begin{itemize}
		\item let $G(X,\partial X)_{ad}^{\mathcal{F}}$ be the restriction of  $G(X,\partial X)_{ad}$ to the open subset $X_{\mathcal{F}}:=X\times[0,1]\setminus\partial X\times\{1\}= \mathring{X}\times[0,1]\cup \partial X\times [0,1)$ (the superscript $\mathcal{F}$ refers to a condition of being Fredholm that will be clear later). It is the union $(G_{|\mathring{X}})_{ad}\cup (G_{|\partial X})_{ad}^{\circ}\times \RR$.
		\item Let $\mathcal{T}_{nc}G(X,\partial X)$ be the restriction of $G(X,\partial X)_{ad}^{\mathcal{F}}$ to $X_{\partial}:=\mathring{X}\times\{0\}\cup \partial X\times[0,1)$. It is the union $ \mathfrak{A}(G_{|\mathring{X}})\cup(G_{|\partial X})_{ad}^{\circ}\times \RR $.
		\item Finally let $G(X,\partial X)_{ad}^{\circ}$ be the restriction of $G(X,\partial X)_{ad}^{\mathcal{F}}$ to $X\times[0,1)$.
	\end{itemize}
\end{definition}

Now let us give some results about the C*-algebras associated to these groupoids, that will be useful later.

\begin{lemma}\label{conemont}
	The C*-algebra $C^*_r(\Gamma(\RR_+,\{0\}))$ is K-contractible.
\end{lemma}
\begin{proof}
	We have that $\Gamma(\RR_+,\{0\})= \RR^*_+\times\RR^*_+\cup\{0\}\times\{0\}\times\RR\rightrightarrows\RR_+$.
	It is isomorphic to the action groupoid $\RR_+\rtimes\RR_+^*$ (where $\RR_+^*$ acts on $\RR_+$ by multiplication) thanks to the morphism
	$\phi\colon\RR^*_+\times\RR^*_+\cup\{0\}\times\{0\}\times\RR \to \RR_+\rtimes\RR^*_+$ such that
	\begin{itemize}
		\item $(y_1,y_2)\mapsto (y_2,\frac{y_1}{y_2})$ if $y_1,y_2\neq0$;
		\item $(0,0,\lambda)\mapsto (0,e^\lambda)$.
	\end{itemize}
	Hence  $C^*_r(\Gamma(\RR_+,\{0\}))\simeq C^*_r(\RR_+\rtimes\RR^*_+)\simeq C_0(\RR_+)\rtimes\RR^*_+$ and, by the Connes-Thom isomorphism, $K_*\left(C_0(\RR_+)\rtimes\RR^*_+\right)\simeq K_{*-1}\left(C_0(\RR_+)\right)=0$.
\end{proof}

\begin{remark}\label{boundary-mont}
	By \cite[Proposition 3.5]{Mont}, $\Gamma(\RR_+,\{0\})$ is amenable. Then we have the following short exact sequence 
	\[
	\xymatrix{0\ar[r]& C^*_r(\RR^*_+\times\RR^*_+)\ar[r]& C^*_r(\Gamma(\RR_+,\{0\}))\ar[r] & C^*_r(\RR)\ar[r]&0},
	\] 
it is semi-split, the associate boundary map in KK-theory is an isomorphism by Lemma \ref{conemont} and it is given by the suspension isomorphism.          
\end{remark}

As we noticed in Remark \ref{fullvsred}, the restriction to a saturated closed subset of $G^{(0)}$ gives an exact sequence of full C*-algebras. But in general this fact is not true for the reduced groupoid C*-algebras.
We are going to prove that this is the case in the situations that we will encounter later.
Moreover we are going to prove that in those situations we have the completely positive lifting property. By
\cite[Theorem 1.1]{SkExact}, this implies that the boundary map associated to these exact sequences is given by the Kasparov product with an element of KK-theory.

\begin{lemma}\label{cone}
	Let $X$ be any smooth manifold and let $H\rightrightarrows X$ be a Lie groupoid and
	consider the groupoid $G=H\times (-1,1)\times(-1,1)\rightrightarrows X\times(-1,1)$. The $b$-groupoid  associated to the restriction of $G$ to $X\times[0,1)$ , denoted by $G(X\times[0,1),X\times\{0\})$, is given by $H\times\Gamma([0,1),\{0\})$.
	Then we have the following semi-split exact sequence of reduced C*-algebras 
	
	\[
	\xymatrix{0\ar[r]& C^*_r\left(H\times(0,1)\times(0,1)\right)\ar[r]& C^*_r\left(G(X\times[0,1),X\times\{0\})\right)\ar[r] & C^*_r\left(H\times\RR\right) \ar[r] & 0}
	\]
	and the boundary map of the long exact sequence of KK-groups associated to it
	is an isomorphism.
\end{lemma}

\begin{proof}
	By Remark \ref{boundaryloc}, the Monthubert groupoid associated to $G$ is given by $H\times\Gamma([0,1),\{0\})$,
	 we have that $C^*_r(H\times \Gamma([0,1),\{0\}))\cong C^*_r(H)\otimes C^*_r(\Gamma([0,1),\{0\}))$.
	In particular the amenability of $\Gamma([0,1),\{0\})$ implies that the following sequence
	\[
	\xymatrix@=0.9em{0\ar[r] & C^*_r(H)\otimes C^*_r((0,1)\times(0,1))\ar[r]& C^*_r(H)\otimes C^*_r(\Gamma([0,1),\{0\}))\ar[r]^(.6){\mathrm{ev}_{\partial}}& C^*_r(H)\otimes C^*_r(\RR)\ar[r]& 0
	}
	\]
is exact and semi-split.	Lemma \ref{conemont} implies that $K_*(C^*_r\left(G(X\times[0,1),X\times\{0\})\right))=0$, hence the result follows. 
\end{proof}


\begin{lemma}\label{montred}Let $G$ be a Lie groupoid over a manifold $X$ with boundary $\partial X$.
	The restriction to the boundary induces the following
	exact sequence of reduced C*-algebras
	\begin{equation}\label{montexact}
	\xymatrix{0\ar[r] & C_r^*(G_{|\mathring{X}})\ar[r] & C_r^*(G(X,\partial X))\ar[r]^{\mathrm{ev}_{\partial X}} & C_r^*(G_{|\partial X}\times\RR)\ar[r]& 0}.
	\end{equation}
Moreover this exact sequence is semisplit.
\end{lemma}

\begin{proof}
By Remark \ref{boundaryloc}, there is an open neighbourhood $U$ of $\partial X$ such that  the restriction of $G(X,\partial X)$ to $U$ is isomorphic to $G_{|\partial X}\times \Gamma([0,1),\{0\})$ and, since $C^*_r(\Gamma([0,1),\{0\}))$ is nuclear, it follows that $C^*_r(G_{|\partial X}\times \Gamma([0,1),\{0\}))\cong C^*_r(G_{|\partial X})\otimes C^*_r(\Gamma([0,1),\{0\}))$.

We have the following commutative diagram
\[
\xymatrix@=0.9em{0\ar[r] & C^*_r(G_{|\partial X})\otimes C^*_r((0,1)\times(0,1))\ar[r]\ar[d]& C^*_r(G_{|\partial X})\otimes C^*_r(\Gamma([0,1),\{0\}))\ar[r]^(.6){\mathrm{ev}_{\partial}}\ar[d] & C^*_r(G_{|\partial X})\otimes C^*_r(\RR)\ar[r] \ar[d]^\cong& 0\\
0\ar[r] & C_r^*(G_{|\mathring{X}})\ar[r] & C_r^*(G(X,\partial X))\ar[r]^{\mathrm{ev}_{\partial X}} & C_r^*(G_{|\partial X}\times\RR)\ar[r]& 0
	}
\]
where the top row is exact thanks to the amenability of $\Gamma([0,1),\{0\})$ and the vertical arrows are inclusions of algebras.

Let $\{\alpha,1-\alpha\}$ be a partition of unity associated to the open cover $\{\partial X\times[0,1), \mathring{X}\}$ of $X$.
Let $\xi\in C_r^*(G(X,\partial X))$ be such that $\mathrm{ev}_{\partial X}(\xi)=0$.
Observe that 
\begin{itemize}
	\item$\alpha \xi\alpha$ belongs to $C^*_r(G_{|\partial X})\otimes C^*_r(\Gamma([0,1),\{0\}))$; 
	\item $\xi-\alpha \xi\alpha=(1-\alpha)\xi\alpha+\xi(1-\alpha)$ belongs to the ideal  $C_r^*(G_{|\mathring{X}})$, because $(1-\alpha)$ is supported on the saturated submanifold $\mathring{X}$.
\end{itemize}
Here we let the functions $\alpha$ and $1-\alpha$ act as in Remark \ref{multiplication}.
 Since the top row is exact, $\alpha \xi\alpha$ belongs to $C^*_r(G_{|\partial X})\otimes C^*_r((0,1)\times(0,1))$.
Consequently we have that $\xi\in C_r^*(G_{|\mathring{X}})$.
This proves the exactness of \eqref{montexact}.

Finally if $s$  is a completely positive section for $\mathrm{ev}_\partial$, then $\eta\mapsto \alpha s(\eta)\alpha$ is a completely positive section for $\mathrm{ev}_{\partial X}$.
\end{proof}

\begin{remark}\label{b-boundary}
As a consequence of Remark \ref{boundary-mont}, we have that the boundary morphism associated to the exact sequence \eqref{montexact} is the composition of the suspension isomorphism $K_*(C^*_r(G_{|\partial X})\otimes C^*_r(\RR))\to K_{*+1}(C^*_r(G_{|\partial X})\otimes C^*_r((0,1)\times(0,1)))$ and the morphism induced by the inclusion
$C^*_r(G_{|\partial X})\otimes C^*_r((0,1)\times(0,1))\to C_r^*(G_{|\mathring{X}})$.

The same is true at the level of Lie algebroids. Indeed the Lie algebroid  of $G_{|\partial X}\times \RR$ is  $\mathfrak{A}G_{|\partial X}\times \RR$ and the Lie algebroid of $G_{|\partial X}\times (0,1)\times(0,1)$ is given by $\mathfrak{A}G_{|\partial X}\times T(0,1)$ that is included into $\mathfrak{A}G$. Then if we consider the exact sequence
\[
\xymatrix{0\to C^*_r(\mathfrak{A}G_{|\partial X}\times T(0,1))\ar[r]& C^*_r(\mathfrak{A}G(X,\partial X))\ar[r]& C^*_r(\mathfrak{A}G_{|\partial X}\times \RR)\ar[r]&0},\]
 the boundary morphism is given as above using the suspension isomorphism.
 
 Now we want to understand the boundary map associated to the exact sequence of C*-algebras
 \begin{equation}\label{montad}
 \xymatrix@=0.9em{0\to C^*_r((G_{|\mathring{X}})_{ad}^\circ)\ar[r]& C^*_r(G(X,\partial X)_{ad}^{\circ})\ar[r]& C^*_r((G_{|\partial X}\times \RR)_{ad}^\circ)\ar[r]&0}.
 \end{equation}
 
 First observe that the boundary map associated to 
 	\[
 	\xymatrix{0\ar[r]& C^*_r((0,1)\times(0,1)_{ad})\ar[r]& C^*_r(\Gamma([0,1),\{0\})_{ad})\ar[r] & C^*_r(\RR_{ad})\ar[r]&0},
 	\] 
 restricts to the suspension isomorphism at any $t$ of the adiabatic deformation  and that it is an isomorphism by the Five Lemma. We will denote it by $\beta_{ad}$.
 
 Then notice that the map
\begin{equation}\label{Delta}
 \Delta_\partial\colon C^*_r((G_{|\partial X})_{ad}^\circ\times \RR_{ad})\to C^*_r( (G_{|\partial X}\times \RR)_{ad}^\circ)
 \end{equation}
 given by the restriction from $\partial X\times[0,1)\times[0,1]\to \partial X\times\{t=s, (t,s)\in [0,1)\times[0,1]\}$, induces KK-equivalence. Indeed $
 C^*_r( (G_{|\partial X}\times \RR)_{ad}^0)\cong C^*_r( (G_{|\partial X})_{ad}^\circ\times \RR)$ and, by means of this isomorphism, $\Delta_\partial$ is homotopic to $\mathrm{id}\otimes \mathrm{ev}_0\colon C^*_r((G_{|\partial X})_{ad}^\circ\times \RR_{ad})\to C^*_r( (G_{|\partial X})_{ad}^\circ\times \RR)$. But since $\mathrm{ev}_0\colon C^*_r(\RR_{ad})\to C^*_r(\RR)$ is a homotopy of C*-algebras we obtain a KK-equivalence.
 
Consider the following commutative diagram
 \[
 \xymatrix@=0.7em{0\ar[r]& C^*_r((G_{|\partial X})_{ad}^{\circ}\times ((0,1)\times(0,1))_{ad})
 	\ar[r]\ar[dd]^\Delta& C^*_r((G_{|\partial X})_{ad}^{\circ}\times \Gamma([0,1),\{0\})_{ad})\ar[r]\ar[dd]^{\Delta_{ad}}& C^*_r((G_{|\partial X})_{ad}^\circ\times \RR_{ad})\ar[r]\ar[dd]^{\Delta_\partial}&0\\
 	& & & &\\
 	0\ar[r]& C^*_r((G_{|\partial X}\times (0,1)\times(0,1))_{ad}^\circ)\ar[r]& C^*_r((G_{|\partial X}\times \Gamma([0,1),\{0\})_{ad}^{\circ})\ar[r]& C^*_r((G_{|\partial X}\times \RR)_{ad}^\circ)\ar[r]&0
 	} 
 \]
 where  all the vertical arrows are defined analogously to \eqref{Delta}.
The boundary morphism associated to the first row is given by the KK-element $\mathrm{id}\otimes \beta_{ad}$. 
 It follows that the boundary map of the second row is given by  $\partial_b'\in KK^1(C^*_r((G_{|\partial X})_{ad}^\circ)\otimes C^*_r(\RR), C^*_r((G_{|\partial X}\times (0,1)\times(0,1))_{ad}^\circ))$, which is the element defined by the Kasparov product
 \[
 [\mathrm{id}\otimes\mathrm{ev}_0]^{-1}\otimes(\mathrm{id}\otimes \beta_{ad})\otimes [\Delta].
 \]
 
 Finally, one obtains that the boundary map associated to the exact sequence \eqref{montad} is given by  the Kasparov product 
 \begin{equation}\label{boundaryadmont}
 \partial_b:=\partial_b'\otimes[\iota_b]\in KK^1(C^*_r((G_{|\partial X})_{ad}^\circ), C^*_r((G_{|\mathring{X}})_{ad}^{0}))
 \end{equation}
 where  $\iota_b\colon C^*_r((G_{|\partial X}\times (0,1)\times(0,1))_{ad}^\circ)\to C^*_r((G_{|\mathring{X}})_{ad}^{\circ})$ is the natural inclusion. 
   
\end{remark}

We need the following technical result (\cite[Lemma 2.2]{CunSk}).
\begin{lemma}\label{cplift}
	Let $I_1$ and $I_2$ be two ideals in a separable C*-algebra $A$. Let $I$ be the intersection of $I_1$ and $I_2$. 
	If the quotient maps $q_i\colon A\to A/I_i$, for $i=1,2$, have  completely positive sections, then the quotient map  $\pi\colon A\to A/I$ has a  completely positive section.
\end{lemma}

\begin{lemma}\label{bkk}
	The restriction morphism \begin{equation}\mathrm{ev}_{X_{\partial}}\colon C^*_r\left(G(X,\partial X)_{ad}^{\mathcal{F}}\right)\to C^*_r\left(\mathcal{T}_{nc}G(X,\partial X)\right)\end{equation} induces a KK-equivalence.
\end{lemma}

\begin{proof}
	
	First we show that the following sequence 
		\begin{equation}\label{admont}
		\xymatrix{0\ar[r] & C_r^*(G_{|\mathring{X}}\times(0,1])\ar[r] & C^*_r\left(G(X,\partial X)_{ad}^{\mathcal{F}}\right)\ar[r]^{\mathrm{ev}_{X_\partial}} & C^*_r\left(\mathcal{T}_{nc}G(X,\partial X)\right)\ar[r]& 0}.
		\end{equation}
		is exact and semi-split. Indeed we have to check exactness only in the middle.
		
If $\xi \in C^*_r\left(G(X,\partial X)_{ad}^{\mathcal{F}}\right)$ is such that $\mathrm{ev}_{X_\partial}(\xi)=0$, then in particular
$\mathrm{ev}_0(\xi)=0$ in $C^*_r(\mathfrak{A}G(X,\partial X))$.
But we know that the sequence associated to $\mathrm{ev}_0$ is exact. Then $\xi$ belongs to $C^*_r(G(X,\partial X)\times(0,1])$.
Moreover, by hypotesis, the restriction to the boundary of $\xi$, as element of $C^*_r(G(X,\partial X)\times(0,1])$, is zero. We can use  Lemma \ref{montred} to prove that $\xi$ belongs to $C_r^*(G_{|\mathring{X}}\times(0,1])$ and then that \eqref{admont} is exact.

Let $A$ denote the C*-algebra $C^*_r\left(G(X,\partial X)_{ad}^{\mathcal{F}}\right)$.
To prove the fact that \eqref{admont} is semisplit we observe that
the ideal $I=\ker \mathrm{ev}_{X_{\partial}}$ is the intersection of the two ideals $I_1=\ker \mathrm{ev}_{\partial X\times[0,1)}$ and $I_2=\ker \mathrm{ev}_{0}$. By Lemma \ref{cplift} $ \mathrm{ev}_{X_{\partial}}$ has a completely positive section.

\end{proof}

\subsection{Index as deformation}\label{inddef}

Let $G$ be a smooth deformation groupoid, namely a Lie groupoid of the following kind:
$$G= G_1 \times \{0\} \cup G_2\times ]0,1] 
\rightrightarrows G^{(0)}=M\times [0,1].$$ 
One can consider the saturated 
open subset $M\times ]0,1]$ of $G^{(0)}$. Using the isomorphisms 
\[C^*(G\vert_{M\times ]0,1]}) \simeq C^*(G_2)\otimes C_0(]0,1])\] and 
$C^*(G\vert_{M\times\{0\}})\simeq C^*(G_1)$,  we obtain the following 
exact sequence of $C^*$-algebras: 
\begin{equation}\label{des}
\xymatrix{0\ar[r]& C^*(G_2)\otimes C_0(]0,1])\ar[r]^-{i}&  C^*(G)\ar[r]^{\mathrm{ev}_0}  & C^*(G_1)\ar[r] &0 }
\end{equation}
where $i$ is the inclusion map and $\mathrm{ev}_0$ is the evaluation map at 
$0$.

We assume now the exact sequence admits a completely positive section. Since 
the $C^*$-algebra $C^*(G_2)\otimes C_0(]0,1])$ is contractible, 
the long exact sequence in KK-theory shows that the group homomorphism 
$KK\left(A,C^*(G)\right) \rightarrow KK\left(A, C^*(G_1)\right)$, given by the Kasparov product with the element $[\mathrm{ev}_0]$,
is an isomorphism for each separable $C^*$-algebra $A$. 

In particular with $A=C^*(G)$
we get that $[\mathrm{ev}_0]$ is invertible in 
KK-theory: there is an element $[\mathrm{ev}_0]^{-1}$ in 
$KK\left(C^*(G_1), C^*(G)\right)$ such that $[\mathrm{ev}_0] {\otimes} [\mathrm{ev}_0]^{-1}=1_{C^*(G)}$ 
and $[\mathrm{ev}_0]^{-1} {\otimes} [\mathrm{ev}_0]=1_{C^*(G_1)}$. 

Let $\mathrm{ev}_1:C^*(G) \rightarrow C^*(G_2)$ be the 
evaluation map at $1$ and $[\mathrm{ev}_1]$ the corresponding element of 
$KK\left(C^*(G),C^*(G_2)\right)$. 
\begin{definition}\label{KKdeformation}
	The KK-element associated to the 
	deformation groupoid $G$ is defined by: 
	$$\partial_G=[\mathrm{ev}_0]^{-1} {\otimes} [\mathrm{ev}_1]\in  KK\left(C^*(G_1),C^*(G_2)\right) \ . $$
\end{definition}

\begin{remark}
	If the sequence \eqref{des} is still exact and semi-split when we consider the reduced groupoid C*-algebras instead of the  full ones, all that we said is still true in the reduced setting. This happens for instance when $G_1$ is amenable and, as we are going to see, the adiabatic deformation is an example of this case.
\end{remark}

Let $G\rightrightarrows X$ be a Lie groupoid and consider its adiabatic deformation $G_{ad}\rightrightarrows X\times[0,1]$.
Recall that it is of the form 
\[\mathfrak{A}(G)\times\{0\}\sqcup G\times(0,1]\] and that $C^*(\mathfrak{A}(G))\cong C_0(\mathfrak{A}^*(G))$ and then, for this case, we can consider the reduced groupoid C*-algebras. This is a particular case of a smooth deformation groupoid. Therefore we can associate to it a KK-element as in definition \ref{KKdeformation}.

\begin{definition}\label{adiabaticindex}
	We will denote by 
	\[\mathrm{Ind_G}\in KK\left(C^*_r(\mathfrak{A}(G)),C^*_r(G)\right)\] 
	the KK-element $\partial_{G_{ad}} $  and we will call the \emph{adiabatic $G$-index} the homomorphism \[K_*\left(C^*_r(\mathfrak{A}(G))\right)\to K_*\left(C_r^*(G)\right),\] given by the Kasparov product with $\mathrm{Ind}_G$.
\end{definition}

This homomorphism corresponds, up to Bott periodicity, to the boundary map associated to the exact sequence \eqref{AES}. Indeed,  one can replace $C^*_r(\mathfrak{A}G)$ with $C^*_r(G_{ad})$ through the KK-equivalence $\mathrm{ev}_0\colon C^*_r(G_{ad})\to C^*_r(\mathfrak{A}G)$. Noticing that $C^*_r(G^\circ_{ad})$ is isomorphic to $ \mathcal{C}_{\mathrm{ev}_1}$, the mapping cone C*-algebra of $\mathrm{ev}_1$, \eqref{AES} corresponds to the mapping exact sequence associated to $\mathrm{ev}_1$. The boundary map of a mapping cone exact sequence associated to a *-homomorphism is just the composition  of the same *-homomorphism and Bott periodicity.

\begin{remark}
In \cite{MP} the authors prove that the adiabatic $G$-index and the classical analytic index given by the pseudodifferential extension coincide.
\end{remark}

\section{Adiabatic groupoid and wrong-way functoriality}
\label{section}

\subsection{The pull-back of a groupoid}\label{pb}
Here we recall the pull-back construction for Lie groupoids.
Let $G\rightrightarrows X$ be a Lie groupoid and let $\varphi\colon Y\to X$ be a transverse map with respect to $G$.
This means that $d\varphi(T_yY)+ q(\mathfrak{A}_{\varphi(y)}(G))=T_{\varphi(y)}X$, where $q\colon\mathfrak{A}(G) \to TX$ is the anchor map of the Lie algebroid.

\begin{definition}\label{pull}
	From the previous data we can define the following spaces:
	\begin{itemize}
		\item $G_\varphi:=\{(\gamma,y)\in G\times Y\,|\,\varphi(y)=s(\gamma)\}$;
		\item $G^\varphi:=\{(y,\gamma)\in Y\times G\,|\,\varphi(y)=r(\gamma)\}$;
		\item $G^\varphi_\varphi:=\{(y_1,\gamma,y_2)\in Y\times G\times Y\,|\, \varphi(y_1)=r(\gamma)\,,\,\varphi(y_2)=s(\gamma)\}$.
	\end{itemize}
\end{definition}

	Contrary to $G_\varphi$ and $G^\varphi$, $G^\varphi_\varphi$ is a groupoid over $Y$.
	The source and the target map for $G^\varphi_\varphi$ are given by $s(y_1,\gamma,y_2)=y_2$ and $r(y_1,\gamma,y_2)=y_1$ respectively.
	Moreover $(y_1,\gamma,y_2)^{-1}=(y_2,\gamma^{-1},y_1)$ and $(y_1,\gamma,y_2)\cdot(y_2,\gamma',y_3)=(y_1,\gamma\cdot\gamma',y_3)$.
	
	Since $r,s\colon G\to X$ are submersions, $G_\varphi$ and $G^\varphi$ are submanifolds of $G\times Y$ and $Y\times G$ respectively.
	We are going to prove that $G_\varphi^\varphi$ is a smooth manifold.
	The space $G_\varphi$  in Definition \ref{pull} is given by the following pull-back
	\begin{equation}\label{k}
	\xymatrix{G_\varphi\ar[r]^{\tilde{\varphi}}\ar[d]^p & G\ar[d]^s \\
		Y\ar[r]^\varphi & X}
	\end{equation}
	and one can see that $p$ is a surjective submersion, because $s$ is so.   
	\begin{lemma}\label{pbt}
		Let $\widetilde{\varphi}$ be the map in diagram \eqref{k}.
The map $k=r\circ\tilde{\varphi}$ is a smooth submersion.
\end{lemma}
\begin{proof}
	Let $(\gamma_0,y_0)\in G_\varphi$ be such that $\gamma_0$ is a unit of $G$. We define the following inclusions
	\begin{itemize}
		\item $i\colon G_{s(\gamma_0)}\to G_\varphi$,
		$i\colon\gamma\mapsto(\gamma,y_0)$ and put $\delta=k\circ i$;
		\item $j\colon Y\to G_\varphi$ is such that
		$j\colon y\mapsto (id_{\varphi(y)},y)$ and put $\varepsilon=k\circ j$.
	\end{itemize}
	Notice that $s(\gamma_0)=\alpha(\gamma_0)=\beta(y_0)=\psi(y_0)$ and then, 
	by trasversality, it turns out that
	\begin{equation*}
	\begin{split}
	dk_{(\gamma_0,y_0)}(di(T_{\gamma_0}G_{s(\gamma_0)})+dj(T_{y_0}Y))&=d\delta(T_{\gamma_0}G_{s(\gamma_0)})+d\varepsilon(T_{y_0}Y)=\\
	&=q(\mathfrak{A}_{\psi(y_0)}(G))+d\varphi(T_{y_0}Y)=T_{\varphi(y_0)}X,
	\end{split}
	\end{equation*}
 that amount to prove the surjectivity of $d_{(\gamma_0,y_0)}k$.
	
	Now let us consider $(\gamma_1,y_1)\in G_\varphi$, where $\gamma_1$ is not necessarily a unit.
	Construct the following pull-back
	\[
	\xymatrix{G^{(2)}_\varphi\ar[r]^{p_1}\ar[d]^{p_2} & G\ar[d]^s \\
		G_\varphi\ar[r]^k & X}
	\]
	where $G^{(2)}_\varphi=\{(\gamma,\gamma',y)\in G\times G_\varphi\,|\, s(\gamma)=r(\gamma') \}$, $p_1\colon(\gamma,\gamma',y) \mapsto \gamma $ and $p_2\colon(\gamma,\gamma',y) \mapsto (\gamma',y)$.
	We have that $p_2$ is a submersion, because $s$ is so. Moreover, at the point $z=(\gamma_1,\varphi(y_1),y_1)$, $d_zp_1$ is onto, since $d_wk$ is onto at $w=(\varphi(y_1),y_1)$.
	
	Let $(m,id)\colon G^{(2)}_\varphi\to U_\varphi$ be the map such that $(m,id)\colon(\gamma,\gamma',y)\mapsto(\gamma\gamma',y)$. Then we have that $r\circ p_1=k\circ(m,id)$. But, at $(\gamma_1,\psi(y_1),y_1)$, $r\circ p_2$ is a submersion, hence so is $k$ at $(m,id)(\gamma_1,\varphi(y_1),y_1)=(\gamma_1,y_1)$.
\end{proof}	

	We have proved that the map $k\colon(\gamma,y)\mapsto r(\gamma)$ is a submersion because of transversality. Then by the following pull-back diagram
	\[
	\xymatrix{G^\varphi_\varphi\ar[r]\ar[d] & G_\varphi\ar[d]^k \\
		Y\ar[r]^\varphi & X}
	\]
	it follows that $G_\varphi^\varphi$ is a smooth manifold. Moreover $G^\varphi_\varphi\rightrightarrows Y$ is a Lie groupoid that we will call the pull-back groupoid of $G$ by $\varphi$.

One can easily show that \[\mathfrak{A}(G^\varphi_\varphi)\simeq\{(\xi,\eta)\in TY\times \mathfrak{A}(G)\,|\,d\varphi(\xi)=q(\eta)\},\]
where $q$ is the anchor map of $\mathfrak{A}(G)$. 
On the other hand the anchor map of $\mathfrak{A}(\,G^\varphi_\varphi)$ is the projection on $TY$.
Now we are going to prove that homotopic transverse maps induce isomorphic pull-back groupoids.
\begin{lemma}\label{isopull}
 Let $\Phi\colon Y \times[0,1]\to X$ be such that
\begin{enumerate}
\item 	$\varphi_t:=\Phi_{|Y\times\{t\}}\colon Y\to X$ is transverse with respect to $G$ for all $t\in[0,1]$;
\item for all fixed $y_0\in Y$ the set $\{\Phi(y_0,t)\,,\, t\in[0,1]\}$ is contained in an orbit of $G$.
	\end{enumerate}
 Then there is an isomorphism $\alpha(\varphi_t)\colon G^{\varphi_0}_{\varphi_0}\to G^{\varphi_1}_{\varphi_1}$.
\end{lemma}

\begin{proof}
By (1) $G_\Phi^\Phi$ is a Lie groupoid and its Lie algebroid is given by
\[
\mathfrak{A}G_\Phi^\Phi=\{(U,V)\in T(Y\times[0,1])\times\mathfrak{A}G\,|\, d\Phi(U)=dr(V) \}.
\]
Let $\partial$  be the vector field that differentiates along the $t$-direction. 
By (2), for all fixed $y_0\in Y$, the integral curves of $\partial$, namely $\{(y_0,t)\,,\,t\in [0,1]\}$, are contained in the orbits of $G_\Phi^\Phi$. Since $Y$ is paracompact, there exists a locally finite cover $\{U_j\}_{j\in J}$ of $Y\times[0,1]$ such that for each  $j\in J$ we have a section $\xi_j$ of $\mathfrak{A}G_\Phi^\Phi$ restricted to $U_j$ that lifts the vector field $\partial_{|U_j}$.
Let $\{\beta_j\}$ be a partition of the unity associated to $\{U_j\}$, then 
$\xi:=\sum_j \beta_j\xi_j$ is a section of $\mathfrak{A}G_\Phi^\Phi$ such that $dr(\xi)=\partial$ everywhere on $Y$.
Now since we have that the curve $r(\exp_y(t\xi))$ gives the vector flow of $\partial$, it follows that
$\exp_y(t\xi)$ has the form
$\left((y,t),\gamma_t(y),(y,0)\right)$, where $s(\gamma_t(y))=\varphi_0(y)$ and  $r(\gamma_t(y))=\varphi_t(y)$.

Finally we can write the desired isomorphism between $G^{\varphi_0}_{\varphi_0}$ and  $G^{\varphi_1}_{\varphi_1}$ in the following way
\begin{equation}\label{isom}
\alpha(\varphi_t)\colon(y_1,\gamma, y_2)\mapsto (y_1,\gamma_1(y_1)\cdot\gamma\cdot\gamma_1(y_2)^{-1},y_2).
\end{equation}

\end{proof}

\begin{lemma}\label{concatenation}
	Let $\varphi_t\colon Y\to X$ and $\psi_t\colon Y\to X$ be as in Lemma \ref{isopull} and such that $\varphi_1=\psi_0$.
	Denote the concatenation of the paths $\varphi_t$ and $\psi_t$ by $\psi_t*\varphi_t$.
	Then
	\[
	\alpha(\psi_t)\circ\alpha(\varphi_t)= \alpha(\psi_t*\varphi_t).
	\]
\end{lemma}
\begin{proof}
	It is clear from the construction of $	\alpha(\psi_t)$ and $\alpha(\varphi_t)$.
\end{proof}

There is a canonical way to construct a natural  $C^*_r(G^\varphi_\varphi)$-$C^*_r(G)$-bimodule associated to the pull-back procedure.
Let us consider the groupoid
\[
G^{\varphi\cup id_X}_{\varphi\cup id_X}\rightrightarrows Y\sqcup X
\]
and let us denote it with $L$.
We have that $L^X_X=G$, $L_Y^Y=G^\varphi_\varphi$,
$L^Y_X=G^\varphi$ and $L^X_Y=G_{\varphi}$.

\begin{remark}
	Because of this decomposition of $G^{\varphi\cup id_X}_{\varphi\cup id_X}$ we will keep the source and target notations for $G^\varphi$ and $G_\varphi$ too, though they are not groupoids.
\end{remark}
Let $p_Y$ be the projection given by the restriction to $L^Y_Y$ and let $p_X$ be the projection given by the restriction to $L^X_X$, they are in the multiplier algebra of $C^*_r(L)$.
Then $E_{\varphi}=p_YC^*_r(L)p_X=C^*_r(G^\varphi)$ is the $C^*_r(G^\varphi_\varphi)$-$C^*_r(G)$-bimodule we were searching for (here $C^*_r(G^\varphi)$ is not a C*-algebra).

\begin{definition}\label{morita}Denote by $\mu_{\varphi}$ the class of $E_\varphi$ in  $KK\left(C^*_r(G^\varphi_\varphi),C^*_r(G)\right)$.
		The $C^*_r(G^\varphi_\varphi)$-valued inner product on $E$ is given by $\langle x,y\rangle_{C^*_r(G^\varphi_\varphi)}=xy^*$ and the $C^*_r(G)$-valued one is given by $\langle x,y\rangle_{C^*_r(G)}=x^*y$.
		It is clear that, if the image of $\varphi$ intersects with all the orbits of $G$, then $E_\varphi$ is full with respect to the  $C^*_r(G)$-valued inner product. Then $\mu_\varphi$ is a Morita equivalence, whose inverse $\mu_{\varphi}^{-1}$ is given by the bimodule $F=p_XC^*_r(L)p_Y$.
		
\end{definition}

\begin{proposition}	\label{prodMorita}	
	If   $\varphi\colon Y\to X$ is transverse with respect to $G\rightrightarrows X$  and $\psi\colon Z\to Y$ is transverse with respect to $G^\varphi_\varphi\rightrightarrows Y$, then
	\[E_{\varphi\circ\psi}=E_\psi\otimes_{C^*_r(G^\varphi_\varphi)}E_\varphi.\]
\end{proposition}
\begin{proof}
	Let  $H\rightrightarrows Z\sqcup Y\sqcup X$ be the pull-back groupoid of $G\rightrightarrows X$ by the map $\varphi\circ\psi\sqcup \varphi\sqcup \mathrm{id}_X\colon Z\sqcup Y\sqcup X \to X$.
	We can see $C^*_r(H)$
	as $3\times 3$ matrices of the following sort
	\[
	\begin{pmatrix}
	C^*_r\left(G^{\varphi\circ\psi}_{\varphi\circ\psi}\right) &  C^*_r\left(G^{\varphi\circ\psi}_\varphi\right)& C^*_r\left(G^{\varphi\circ\psi}\right)\\ 
	C^*_r\left(G^\varphi_{\varphi\circ\psi}\right) &  C^*_r\left(G_\varphi^\varphi\right)& C^*_r\left(G^\varphi\right)\\ 
	C^*_r\left(G_{\varphi\circ\psi}\right) &  C^*_r\left(G_{\varphi\circ\psi}^\varphi\right)& C^*_r\left(G\right)
	\end{pmatrix}
	\] 
	and that $E_{\varphi\circ\psi}=C^*_r\left(G^{\varphi\circ\psi}\right)$, $E_\psi=C^*_r\left((G^{\varphi}_\varphi)^\psi\right)\cong C^*_r\left(G^{\varphi\circ\psi}_\varphi\right)$ and $E_\varphi=C^*_r\left(G^\varphi\right)$ sit concretely in $C^*_r(H)$.
	
	First consider the map $\Xi\colon E_\psi\otimes_{C^*_r(G^\varphi_\varphi)}E_\varphi\to C^*_r\left(G^{\varphi\circ\psi}_\varphi\right)*C^*_r\left(G^\varphi\right)\subset C^*_r\left(G^{\varphi\circ\psi}\right)=E_{\varphi\circ\psi}$ given by
	$f\otimes g\to f*g$, with $f\in C^*_r\left(G^{\varphi\circ\psi}_\varphi\right)$ and $g\in C^*_r\left(G^\varphi\right)$. Let us prove that $\Xi$ is an isometry: 
	\begin{equation*}
	\begin{split}
	&\langle f\otimes g, f\otimes g\rangle_{\tiny{E_\psi\otimes_{C^*_r(G^\varphi_\varphi)}E_\varphi}}=\\&=\langle g,\langle f,f\rangle_{\tiny{E_\psi}}g\rangle_{\tiny{E_\varphi}}=\\
	&=\langle g,(f^**f)*g\rangle_{\tiny{E_\varphi}}=\\
	&=g^**(f^**f)*g=\\&=(g*f)^**(f*g)= \\
	&=\langle f* g, f*g\rangle_{E_{\varphi\circ\psi}}\\
	&=\langle \Xi(f\otimes g), \Xi(f\otimes g)\rangle_{E_{\varphi\circ\psi}}.
	\end{split}
	\end{equation*}
	
Now we have to prove that 
$C^*_r\left(G^{\varphi\circ\psi}_\varphi\right)*C^*_r\left(G^\varphi\right)= C^*_r\left(G^{\varphi\circ\psi}\right)$, namely that $\Xi$ is surjective.
Since $G^{\varphi\circ \psi}_{\varphi}$ is a left principal $G^{\varphi\circ \psi}_{\varphi\circ \psi}$-bundle over $Z$, by \cite[Proposition 2.10]{MRW} one has that the left action of $C^*_r(G^{\varphi\circ \psi}_{\varphi\circ \psi})$ on $E_\psi$ gives a surjective morphism of C*-algebras $C^*_r(G^{\varphi\circ \psi}_{\varphi\circ \psi})\to \mathcal{K}(E_\psi)$. In other words for any $\xi\in C^*_r(G^{\varphi\circ \psi}_{\varphi\circ \psi})$ there exist $\zeta_1,\zeta_2\in C^*_r(G^{\varphi\circ \psi}_{\varphi})$ such that $\xi=\zeta_1*\zeta_2^*$.

The same is true if we consider the action of $C^*_r(G^{\varphi\circ \psi}_{\varphi\circ\psi})$ on $C^*_r(G^{\varphi\circ \psi})$. Now since the obvious representation of $\mathcal{K}(E)$ on $E$ is non-degenerate for any Hilbert module $E$, it follows that 
$C^*_r(G^{\varphi\circ \psi})=C^*_r(G^{\varphi\circ \psi}_{\varphi\circ \psi})*C^*_r(G^{\varphi\circ \psi})$, where the right member of the equality is equal to $\mathcal{K}(E_{\varphi\circ\psi})\cdot E_{\varphi\circ\psi}$.

Then for any $\eta\in C^*_r(G^{\varphi\circ \psi})$ we have $\xi\in C^*_r(G^{\varphi\circ \psi}_{\varphi\circ \psi})$ and $\eta'\in C^*_r(G^{\varphi\circ \psi})$ such that $\eta=\xi*\eta'$. But, as we saw before, there exist $\zeta_1,\zeta_2\in C^*_r(G^{\varphi\circ \psi}_{\varphi})$ such that $\xi=\zeta_1*\zeta_2^*$.
Then
\begin{equation*}
\eta=\xi*\eta'=(\zeta_1*\zeta_2^*)*\eta'= \zeta_1*(\zeta_2^**\eta')
\end{equation*}
	with $\zeta_1\in C^*_r(G^{\varphi\circ \psi}_{\varphi})$ and
	$\zeta_2^**\eta'\in C^*(G^\varphi)$. This proves that $C^*_r\left(G^{\varphi\circ\psi}_\varphi\right)*C^*_r\left(G^\varphi\right)=C^*_r\left(G^{\varphi\circ\psi}\right) $ and then that 
		\[E_{\varphi\circ\psi}=E_\psi\otimes_{C^*_r(G^\varphi_\varphi)}E_\varphi.\]
\end{proof}

\begin{proposition}\label{moritarestriction}
	If   $\varphi\colon Y\to X$ is transverse with respect to $G\rightrightarrows X$ and $Z$ is a saturated and locally closed submanifold  of $X$ and $\varphi'=\varphi_{|\varphi^{-1}(Z)}$, then the following equality holds
	\[
	[e_{\varphi^{-1}(Z)}]\otimes \mu_{\varphi'}=\mu_\varphi\otimes [e_Z]
	\]
	in $ KK\left(C^*_r\left((G^\varphi_\varphi)\right),C^*_r\left(G_{|Z}\right)\right)$.
\end{proposition}
\begin{proof}
	Let $\iota\colon Z\to X$ and $\iota'\colon \varphi^{-1}(Z)\to Y$ be the obvious inclusions.
	Notice that $G_{Z}=G_\iota^\iota$
	and that $(G^\varphi_\varphi)_{|\varphi^{-1}(Z)})=(G^\varphi_\varphi)_{\iota'}^{\iota'}= G_{\varphi\circ\iota'}^{\varphi\circ\iota'}$.
	Moreover $[e_Z]$ is induced by the bimodule $C^*(G_\iota)$ and $[e_{\varphi^{-1}(Z)}]$ is induced by the bimodule $C^*_r((G^\varphi_\varphi)_{\iota'})$.
	 
	 Now, bearing in mind the same procedures we used in the proof of Proposition \ref{moritarestriction}, we have  the following isomorphisms
	 \begin{equation*}
	 \begin{split}
	 &C^*_r((G^\varphi_\varphi)_{\iota'})\otimes_{C^*_r(G_{\varphi\circ\iota'}^{\varphi\circ\iota'})}C^*_r(G^{\varphi'})\cong\\
	 &\cong C^*_r(G^\varphi_{\varphi\circ\iota'})\otimes_{C^*_r(G_{\varphi\circ\iota'}^{\varphi\circ\iota'})}C^*_r((G_\iota^\iota)^{\varphi'})\cong\\
	 &\cong C^*_r(G^\varphi_{\iota\circ\varphi'})\otimes_{C^*_r(G_{\iota\circ\varphi'}^{\iota\circ\varphi'})}C_r^*(G_\iota^{\iota\circ\varphi'})\cong\\
	 &\cong C^*_r(G^\varphi_\iota)=C^*_r(G^\varphi)\otimes_{C^*_r(G)}C^*_r(G_\iota)
	 \end{split}
	 \end{equation*}
	 which imply the equality we had to prove.
	
\end{proof}

\begin{remark}
It is worth recalling that everything was proved in this section for the reduced groupoid C*-algebras is a fortiori true for the full groupoid C*-algebras. 
\end{remark}

\subsection{Wrong-way functoriality for submersions}
Let $G\rightrightarrows G^{(0)}$  be a Lie groupoid. Put $X=G^{(0)}$ and let $\varphi\colon Y \to X$ be a  smooth map transverse with respect to $G$.
Consider the adiabatic groupoid of $G$
\[
G_{ad}=\mathfrak{A}G\times\{0\}\sqcup G\times(0,1] \rightrightarrows X\times[0,1].
\]
Let us recall that the Lie algebroid $\mathfrak{A}(G_{ad})$ is, as a vector bundle, isomorphic to $\mathfrak{A}(G)\times[0,1]$ and 	the anchor map of the adiabatic algebroid   is given by $q_{ad}\colon(\eta,t)\mapsto(t\cdot q(\eta),t)$, see for instance \cite[Example 7, Section 4]{NWX}.

We will need to do the pull-back along $\bar{\varphi}:=\varphi\times \mathrm{id}_{[0,1]}$ of  the adiabatic groupoid, but the fact that $\varphi$ is transverse  with respect to $G$ does not imply that $\bar{\varphi}$ is transverse with respect to $G_{ad}$. Indeed, at $0$ the anchor map of $\mathfrak{A}(G)\times[0,1]$ is zero. Thus $\bar{\varphi}$ is transverse if and only if $\varphi$ is a submersion. So let $\varphi$ be a submersion, then the pull-back groupoid of $G_{ad}$ along $\bar{\varphi}$, given by 
\[
(G_{ad})^{\bar{\varphi}}_{\bar{\varphi}}= \{\left((y_1,t),\gamma,(y_2,t)\right)\in (Y\times[0,1])\times G_{ad}\times (Y\times[0,1])\,|\,r(\gamma)=\bar{\varphi}(y_1,t),s(\gamma)=\bar{\varphi}(y_2,t)\},
\]
 is a smooth manifold.

\begin{lemma}\label{iso}
	The Lie algebroid of  $(G_{ad})^{\bar{\varphi}}_{\bar{\varphi}}$, as vector bundle, is non-canonically isomorphic to $\pi^*(\ker(d\varphi)\oplus \varphi^*\mathfrak{A}(G))$, where $\pi\colon Y\times[0,1]\to Y$ is the projection.
\end{lemma}
\begin{proof}
We have that
	\[
	\mathfrak{A}\left((G_{ad})^{\bar{\varphi}}_{\bar{\varphi}}\right)=\left\{\left( (\xi,t),(\eta,t)\right)\in (TY\times[0,1])\times \mathfrak{A}(G_{ad})\,|\,d\varphi(\xi)=t\cdot q(\eta)\right\}.\]

	We deduce that, for $t\neq0$, the fiber of $\mathfrak{A}\left((G_{ad})^{\bar{\varphi}}_{\bar{\varphi}}\right)$ on 
	$(x,t)$ is given by the following pull back
	\[\xymatrix{
		T_yY\times_{T_{\varphi(y)}X}\mathfrak{A}_{\varphi(y)}(G)\ar[r]\ar[d] & \mathfrak{A}_{\varphi(y)}(G)\ar[d]^{t\cdot q}\\
		T_yY\ar[r]_{d\varphi} & T_{\varphi(y)}X}
	\]
	and 
	for $t=0$ 	the fiber on $(y,0)$ is $\ker(d\varphi)_y\oplus \mathfrak{A}_{\varphi(y)}(G)$. 
	
	Since any vector bundle on $Y\times[0,1]$ is  isomorphic to a vector bundle that is constant in the $[0,1]$-direction, it is enough to know its restriction at $0$ which is exactly $\ker(d\varphi)\oplus \varphi^*\mathfrak{A}(G)$. 
\end{proof}

We are going to construct an element 
$\varphi_!^{ad}\in KK\left(C^*_r((G^\circ_{ad})^{\bar{\varphi}}_{\bar{\varphi}}),C^*_r(G^\circ_{ad})\right)$ associated to $\varphi$ by means of a deformation groupoid, as in the Section \ref{inddef}. It is given by  the following double adiabatic deformation \[\left((G_{ad})^{\bar{\varphi}}_{\bar{\varphi}}\right)_{ad}\rightrightarrows Y\times[0,1]\times[0,1].\]
Let us
give an explicit picture of  this groupoid.
We fix the variables of the two deformations:
\begin{itemize}
	\item in the horizontal direction of the square we have the parameter $t$ of the first adiabatic deformation;
	\item in the vertical direction of the square we have the parameter $u $ of the second adiabatic deformation, performed after the pull-back construction.
\end{itemize}
Then we obtain a groupoid, let us call it $H$, with objects set $Y\times[0,1]_t\times[0,1]_u$, such that
\begin{itemize}
	\item $H$ restricted to $\{u=c\}$, for any $c\in (0,1]$, is equal to $(G_{ad})^{\bar{\varphi}}_{\bar{\varphi}}$, the pull-back of the adiabatic deformation;
		\item $H$ restricted to $\{t=c'\}$,  for any $c'\in (0,1]$, is equal to 
		$\left(G^\varphi_\varphi\right)_{ad}$, the adiabatic deformation  groupoid of the pull-back;
	\item $H$ restricted to the $t$-axis, i.e. to $\{u=0\}$, is the Lie algebroid of $(G_{ad})^{\bar{\varphi}}_{\bar{\varphi}}$, that we have calculated above;

	\item $H$ restricted to $\{t=0\}$ is equal to the
	adiabatic deformation of the Lie algebroid $(\mathfrak{A}(G))^\varphi_{\varphi}$.
	
\end{itemize}

\begin{definition}\label{Lalg}
Let $\mathcal{L_\varphi}$ denote the reduced groupoid C*-algebra $C^*_r\left(H\right)$. We will drop the subscript $\varphi$ when the context does not create ambiguity.

For subset $K\in[0,1]\times[0,1]$ the set $Y\times  K$ is $H$-invariant and we denote by $\mathcal{L}_K$ the C*-algebra of $H_{|Y\times K}$. If $K$ is open then $\mathcal{L}_K$ is an ideal. 
In the particular case where $K=[0,1]\times[0,1]\setminus\{(1,1)\}$ let us denote $\mathcal{L}_K$ by $\mathcal{L}^\circ$. 

\end{definition}

\begin{lemma}\label{4} 
	The evaluation morphisms at $t=1$, $e^t_1\colon\mathcal{L}\to\mathcal{L}_{\{t=1\}}$ and $e^t_1\colon\mathcal{L}^\circ\to\mathcal{L}^\circ_{\{t=1\}}$ induce  KK-equivalences.
\end{lemma}
\begin{proof}
	Let us prove it for $\mathcal{L}^\circ$, the proof for $\mathcal{L}$ is similar.
	Consider the following exact sequence
	\[
	\xymatrix{0\ar[r] & \mathcal{L}^\circ_{\{t\neq1\}}\ar[r] & \mathcal{L}^\circ\ar[r] & \mathcal{L}^\circ_{\{t=1\}}\ar[r] & 0 },
	\]
	the evaluation at $t=1$ has a completely positive section: since $\mathcal{L}^\circ_{(0,1]}$ contains as ideal $ \mathcal{L}^\circ_{\{t=1\}}\otimes C_0(0,1]$, the map $\xi\mapsto t\cdot\xi$ does the job. Hence this exact sequence is semi-split and it is sufficient to prove 
	the K-contractibility of $\mathcal{L}^\circ_{\{t\neq1\}}$.
	
	But let us point out that the evaluation map at $u=0$
	from $\mathcal{L}^\circ_{\{t\neq1\}}$ to $\mathcal{L}^\circ_{\{u=0\,,\,t\neq1\}}$ is a KK-equivalence: this follows from the KK-equivalence between $C_r^*(\mathcal{G}_{ad})$ and $C^*_r(\mathfrak{A}(\mathcal{G}))$ in
	the particular case of $\mathcal{G}=(G_{ad}^\circ)^{\bar{\varphi}}_{\bar{\varphi}}$.
	Hence we have to prove the K-contractibility of $\mathcal{L}^\circ_{\{u=0\,,\,t\neq1\}}$. Since $\mathcal{L}^\circ_{\{u=0\,,\,t\neq1\}}$ is the Lie algebroid of $(G_{ad}^\circ)^{\bar{\varphi}}_{\bar{\varphi}}$, 
	by Lemma \ref{iso}, $\mathcal{L}^\circ_{\{u=0\,,\,t\neq1\}}$ is non-canonically isomorphic to $C^*_r(\mathfrak{A}\left(G^\varphi_\varphi\right))\otimes C[0,1)$, that is clearly K-contractible.
\end{proof}

Let $[e^t_1]\in KK(\mathcal {L}^\circ,\mathcal{L}^\circ_{\{t=1\}})$ denote the KK-equivalence stated in Lemma \ref{4} and let $e_1^u\colon \mathcal{L}^\circ\to \mathcal{L}^\circ_{\{u=1\}}$ be the evaluation map at $u=1$.
\begin{definition}\label{!}
	Let $G\rightrightarrows X$ be a Lie groupoid and
	let $\varphi\colon Y\to X$ be a smooth submersion between smooth manifolds.
	Hence we can define the lower shriek map  $\varphi_!^{ad}$ as the element
	\[
[e^t_1]^{-1}\otimes_{\mathcal{L}^\circ}[e_1^u]\otimes_{\mathcal{L}^\circ_{\{u=1\}}}\mu_{\bar{\varphi}} \in KK(C^*_r((G^\varphi_\varphi)_{ad}^\circ),C^*_r(G_{ad}^\circ)),
	\]
	where $\mu_{\bar{\varphi}}$ is as in Definition \ref{morita}.
\end{definition} 

\begin{remark}
	Observe that, although $G$ does not appear in the notation, $\phi_!^{ad}$ does depend on $G$. We do this choice for not making heavier the notation and usually $G$ is understood from the context. By the way if ambiguities occur, this dependence will be made explicit.
\end{remark}

\begin{lemma}\label{path!}
	Let $\varphi_t\colon Y\to X$ be as in Lemma \ref{isopull}.
	If $\varphi_t\colon Y\to X$ is a submersion for all $t\in[0,1]$,    $(\varphi_0)_!^{ad}$ corresponds with $(\varphi_1)_!^{ad}$ through the isomorphism of adiabatic groupoids induced by \eqref{isom}.
\end{lemma}

\begin{proof}
Since $\varphi_t\colon Y\to X$ is a submersion for all $t\in[0,1]$, it follows that
$\bar{\varphi}_t\colon Y\times[0,1]\to X\times[0,1]$ is  transverse with respect to $G_{ad}$ for all $t\in[0,1]$.
Then applying Lemma \ref{isopull}, we have an isomorphism $\alpha(\bar{\varphi}_t)$ between $(G_{ad})_{\bar{\varphi}_0}^{\bar{\varphi}_0}$ and $(G_{ad})_{\bar{\varphi}_1}^{\bar{\varphi}_1}$ inducing an isomorphism between
$\mathcal{L}_{\bar{\varphi}_0}^\circ$ and $\mathcal{L}_{\bar{\varphi}_1}^\circ$.
Then clearly we have that 
\[
(\varphi_0)_!^{ad}=[\alpha(\varphi_t)]\otimes (\varphi_1)_!^{ad}.
\]

\end{proof}

\begin{lemma}\label{functiso}
	Let $\varphi_t\colon Y\to X$ be as in Lemma \ref{isopull} and let $\psi\colon Z\to Y$ be a submersion.
	 Then we have the following equality
	 \begin{equation}\label{interw}
	 \psi_!^{ad}\otimes[\alpha(\varphi_t)]=[\alpha(\varphi_t\circ \psi)]\otimes\psi_!^{ad}\in KK(C^*_r((G^{\varphi_0\circ\psi}_{\varphi_0\circ\psi})_{ad}^\circ),C^*_r((G^{\varphi_1}_{\varphi_1})_{ad}^\circ))
	 \end{equation}
	 where  $\psi_!^{ad}$  is an element of  $KK(C^*_r((G^{\varphi_1\circ\psi}_{\varphi_1\circ\psi})_{ad}^\circ),C^*_r((G^{\varphi_1}_{\varphi_1})_{ad}^\circ)) $ on the right side  and it is an element of  $ KK(C^*_r((G^{\varphi_0\circ\psi}_{\varphi_0\circ\psi})_{ad}^\circ),C^*_r((G^{\varphi_0}_{\varphi_0})_{ad}^\circ)) $ on the left side.
\end{lemma}
\begin{proof}
It is enough to observe that 
the isomorphism $\alpha(\varphi_t)\colon (G^{\varphi_0}_{\varphi_0})_{ad}\to (G^{\varphi_1}_{\varphi_1})_{ad} $ induces an isomorphism
$ ((G^{\varphi_0}_{\varphi_0})_{ad})_\psi^\psi)_{ad}\to ((G^{\varphi_1}_{\varphi_1})_{ad})_\psi^\psi)_{ad} $. This isomorphism restricts to  an isomorphism $ ((G^{\varphi_0}_{\varphi_0})_\psi^\psi)_{ad}\to ((G^{\varphi_1}_{\varphi_1})_\psi^\psi)_{ad}$ that is exactly
$\alpha(\varphi_t\circ \psi)\colon(G^{\varphi_0\circ\psi}_{\varphi_0\circ\psi})_{ad}\to (G^{\varphi_1\circ\psi}_{\varphi_1\circ\psi})_{ad}$.

Now, by the definition of $\psi_!^{ad}$, the equality \eqref{interw} is clear.
\end{proof}

The following result is the Thom isomorphism in the context of the adiabatic deformation groupoid.

\begin{proposition}\label{thom}
	Let $p\colon E\to Y$ be a vector bundle and let $G\rightrightarrows Y$  be a Lie groupoid. Then $p_!^{ad}\in KK(C^*_r((G^p_p)^\circ_{ad}),C^*_r(G^\circ_{ad}))$ is a KK-equivalence.
\end{proposition}
\begin{proof}
	Since $p$ is surjective, by definition \ref{morita}, $\mu_{\bar{p}}$ is a Morita equivalence between $C^*_r\left((G^\circ_{ad})^{\bar{p}}_{\bar{p}}\right)$ and $C^*_r(G^\circ_{ad})$.
	Because of the definition of $p_!^{ad}$, it is sufficient to show that the evaluation at $u=1$ gives a KK-equivalence in $KK(\mathcal{L}^\circ,\mathcal{L}^\circ_{\{u=1\}})$.
	But this is equivalent to prove that the kernel of the evaluation at $u=1$,
	$\mathcal{L}^\circ_{\{u\neq1\}}$, is KK-contractible. This turns to be equivalent to KK-contractibility of $\mathcal{L}^\circ_{\{t=0,u\neq1\}}$ because of the following exact sequence
	\[
	\xymatrix{0\ar[r] & \mathcal{L}^\circ_{\{t\neq0,u\neq1\}}\ar[r] & \mathcal{L}^\circ_{\{u\neq1\}}\ar[r] & \mathcal{L}^\circ_{\{t=0,u\neq1\}}\ar[r] & 0 },
	\]
	and the KK-contractibility of $\mathcal{L}^\circ_{\{t\neq0,u\neq1\}}\simeq (G^p_p)_{ad}^\circ\otimes C_0(0,1]$.
		But $\mathcal{L}^\circ_{\{t=0,u\neq1\}}$ is 
	the adiabatic deformation of the pull-back groupoid
		\[
		(\mathfrak{A}(G))^p_p=\{(\eta_1,\xi,\eta_2)\in E\times\mathfrak{A}(G)\times E\,|,\ p(\eta_1)=r(\xi), p(\eta_2)=s(\xi)\}
		\]
	that in turn is exactly the definition of the Whitney sum of vector bundles $E\oplus E\oplus \mathfrak{A}(G)$.
	But the adiabatic deformation of a Lie groupoid given by a vector bundle is isomorphic to the pull-back of the vector bundle itself on the interval $[0,1)$.

	Then $\mathcal{L}^\circ_{\{t=0,u\neq1\}}$ is isomorphic to \[\mathcal{L}^\circ_{\{t=0,u\neq1\}}\simeq C_0(E\oplus E\oplus \mathfrak{A}(G))\otimes C_0[0,1)\]
	which is KK-contractible.
	
\end{proof}

Now we want to check that the construction of the lower shriek element behaves well with respect to the composition of submersions.
\begin{proposition}\label{compo}
	Let $G\rightrightarrows Z$ be a Lie groupoid.
	Let $f\colon Y\to X$ and $g\colon X\to Z$ be two smooth submersions between smooth manifolds. Then we have that
	\[
	(g\circ f)^{ad}_{!}=f_!^{ad}\otimes g^{ad}_!\in KK(C^*_r((G^{g\circ f}_{g\circ f})_{ad}^\circ),C^*_r(G_{ad}^\circ)).
	\]
\end{proposition}

\begin{proof}

	Consider the Lie groupoid $K$ given by \[\left(\left(\left((G_{ad})^{\bar{g}}_{\bar{g}}\right)_{ad}\right)^{\bar{f}}_{\bar{f}}\right)_{ad}\rightrightarrows Y\times[0,1]\times[0,1]\times[0,1],\]
	where we set $t,u,v$ as the parameters respectively of the  first, the second and the third  adiabatic deformation in the construction of the groupoid.
	For sake of clarity let us set some notations:
	\begin{itemize}
		\item $C^\circ:=([0,1]\times[0,1]\times[0,1])\setminus\{(1,1,1)\}$ is the cube without the point $(1,1,1)$;
		\item $S^\circ:=([0,1]\times[0,1])\setminus\{(1,1)\}$ is the square without the point $(1,1)$;
		\item $H=Y\times C^\circ$;
		\item $T=\{t=1\}\subset H$, $U=\{u=1\}\subset H$ and $V=\{v=1\}\subset H$ are the right, the posterior and the top faces of the cube, respectively;
		
	\end{itemize}
	
	Restricting $K$ to the  faces $T$, $U$ and $V$ and  to their shared edges, we recognize the following groupoids:
	\begin{itemize}
		\item $K_T$ is equal to $\left( \left(\left(G^g_g\right)_{ad}\right)^{\bar{f}}_{\bar{f}} \right)_{ad}\,$ restricted to $Y\times S^\circ$ ,
		$\,K_U$ is equal to $\left( \left(\left( G_{ad}\right)^{\bar{g}}_{\bar{g}}\right)^{\bar{f}}_{\bar{f}} \right)_{ad}\,$ restricted to $Y\times S^\circ$,
		$\,K_V$ is equal to $\left( \left(\left( G_{ad}\right)^{\bar{g}}_{\bar{g}}\right)_{ad} \right)^{\bar{f}}_{\bar{f}}\,$ restricted to $Y\times S^\circ$;
		\item $K_{T\cap U}=\left(\left( G_g^g\right)^f_f\right)^\circ_{ad}\,$, 	$\,K_{U\cap V}=\left( \left( G^\circ_{ad}\right)_{\bar{g}}^{\bar{g}}\right)^{\bar{f}}_{\bar{f}}\,$, $\,K_{T\cap V}=\left( \left( G_g^g\right)^\circ_{ad}\right)^{\bar{f}}_{\bar{f}}\,$
		.
	\end{itemize}
		Using Lemma \ref{4}, we get the following KK-equivalences:
	$e^{V}_{V\cap T}$, $e^{T}_{T\cap U}$, $e^{U}_{T\cap U}$, $e^H_{T}$ and $e^{H}_U$.
		We have that $f_!^{ad}$ is constructed through the groupoid $K_T$ and it is
	equal to \[(e^T_{T\cap U})^{-1}\otimes e^{T}_{T\cap V}\otimes\mu_f.\]
		On the other hand $g_!^{ad}$ is constructed by means of the restriction of the Lie groupoid $L=\left((G_{ad})^{\bar{g}}_{\bar{g}}\right)_{ad}$ to $X\times S^\circ$ and it is equal to
	\[(e_M)^{-1}\otimes e_N\otimes\mu_g,\]
	where $M=X\times\{1\}\times[0,1)$ and $N=X\times[0,1)\times\{1\}$, inside $X\times S^\circ$, and $L_M=\left(G^g_g\right)^\circ_{ad}$, $L_N=(G^\circ_{ad})^{\bar{g}}_{\bar{g}}$. Thus we have that
	\[
	f_!^{ad}\otimes g_!^{ad}=(e^T_{T\cap U})^{-1}\otimes e^{T}_{T\cap V}\otimes\mu_f\otimes(e_M)^{-1}\otimes e_N\otimes\mu_g.
	\]
	Applying Proposition \ref{moritarestriction}, we get the following equality 
	\[
	\mu_f\otimes(e_M)^{-1}\otimes e_N\otimes\mu_g=(e^V_{T\cap V})^{-1}\otimes e^V_{V\cap U}\otimes\mu_{f'}\otimes\mu_g
	\]
	and, using Proposition \ref{prodMorita}, the term on the right side is equal to
	\[
(e^V_{T\cap V})^{-1}\otimes e^V_{V\cap U}\otimes\mu_{g\circ f}.
	\]
	Then
	\begin{equation*}
	\begin{split}
	f_!^{ad}\otimes g_!^{ad}&=(e^T_{T\cap U})^{-1}\otimes e^{T}_{T\cap V}\otimes(e^V_{T\cap V})^{-1}\otimes e^V_{V\cap U}\otimes\mu_{g\circ f}=\\
	&=(e^T_{T\cap U})^{-1}\otimes (e^{H}_{T})^{-1}\otimes e^H_V\otimes e^V_{V\cap U}\otimes\mu_{g\circ f}=\\
	&=(e^H_{T\cap U})^{-1}\otimes e^H_{V\cap U}\otimes\mu_{g\circ f}=\\
	&=(e^U_{T\cap U})^{-1}\otimes(e^H_U)^{-1}\otimes e^H_U\otimes e^U_{V\cap U}\otimes\mu_{g\circ f}=\\
	&=(e^U_{T\cap U})^{-1}\otimes e^U_{V\cap U}\otimes\mu_{g\circ f}=\\
	&=(g\circ f)_!^{ad}.
	\end{split}
	\end{equation*}

\end{proof}

\subsection{Wrong-way functoriality for transverse maps}

Let $G\rightrightarrows X$ be a Lie groupoid and let $\varphi\colon Y\to X$ be a smooth map that is transverse with respect to $G$.
Since $\mathfrak{A}G$ is defined as the normal bundle of $X$ in $G$, there exists an open neighbourhood  $V\subset \mathfrak{A}G$ of the zero section that is isomorphic to a tubular neighbourhood $U$ of $X$ in $G$. Let us call  $\theta\colon V\to U$ this isomorphism. Observe that, if $v\in V$,  $p(v)=s(\theta(v))$ and in particular that $s(\theta(tv))$ is constant with respect to $t$, for $tv\in V $.
Let $W\subset \varphi^*\mathfrak{A}G$ be the pull-back of $V$ by $\varphi$.
 Consider the following commutative square
\begin{equation}\label{eqtrans}
\xymatrix{
W\ar[r]^{\widetilde{\varphi}}\ar[d]^p &U\ar[d]^s\\
	Y\ar[r]^\varphi & X
	}
\end{equation}
where 
$\widetilde{\varphi}$ is defined as the composition of $	\varphi^*\mathfrak{A}G\to\mathfrak{A}G$  and $\theta$.

Notice the following facts:
\begin{itemize}
	\item $p$ is a bundle projection, then by Lemma \ref{thom} $p_!^{ad}$ is a KK-equivalence;
	\item as in Lemma \ref{pbt} one proves that $r\circ \widetilde{\varphi}$ is a submersion  and let us fix a canonical homotopy $\{\psi_t\colon\xi\mapsto r\circ\widetilde{\varphi}(t\xi)\}_{t\in[0,1]}$ of transverse maps from
	$\varphi\circ p= s\circ\widetilde{\varphi}$ to $r\circ\widetilde{\varphi}$. This specific homotopy is used from now on in the paper.
	Let $\alpha(\psi_t)$ be the associated isomorphism between $G_{\varphi\circ p}^{\varphi\circ p}$ and $G_{r\circ\widetilde{\varphi}}^{r\circ\widetilde{\varphi}}$, as in \eqref{isom}.
\end{itemize}

\begin{definition}\label{!trans}
	Let $\varphi\colon Y\to X$ be as above. Define the lower shriek map associated to $\varphi$ as the element
	\[
	\varphi_!^{ad}:=(p_!^{ad})^{-1}\otimes[\alpha(\psi_t)]
\otimes(r\circ \widetilde{\varphi})_!^{ad}\in  KK\left(C^*_r((G_\varphi^\varphi)_{ad}^\circ),C_r^*(G_{ad}^\circ) \right)	\]
\end{definition}

\begin{lemma}
	The element $\varphi_!^{ad}$ does not depend on the choice of the tubular neighbourhood $U$ of $X$ in $G$.
\end{lemma}
\begin{proof}
	Let $W$ and $U$ be as in \eqref{eqtrans}. Let $U'$ be another  tubular neighbourhood of $X$ in $G$ such that $U'\subset U$ and construct $W'$ in the same way as $W$. Observe that we have the following commutative diagram
\[	\xymatrix{W'\ar[r]^{\widetilde{\varphi}'}\ar[d]^{i}& U'\ar[d]^j\\
		W\ar[r]^{\widetilde{\varphi}}\ar[d]^p &U\ar[d]^s\\
		Y\ar[r]^\varphi & X}
\]	
and put $p'=p\circ i$, $s'=s\circ j$ and $r'=r\circ j$. Notice that $i$ and $j$ are submersive embeddings.
Let $	\varphi_!^{ad}$ and $	(\varphi_!^{ad})'$ be the lower shriek maps constructed using $U$ and $U'$ respectively. Finally let $\psi'_t$ be defined by $\xi\mapsto r'\circ\widetilde{\varphi}'(t\xi)$ for $t\in[0,1]$, notice that $\psi_t'=\psi_t\circ i$.
Consider the following computation
\begin{equation*}
\begin{split}
(p')_!^{ad}\otimes 	(\varphi_!^{ad})'&=[\alpha(\psi_t')]\otimes(r'\circ \widetilde{\varphi}')_!^{ad}=\\
 &=[\alpha(\psi_t')]\otimes(r\circ j\circ \widetilde{\varphi}')_!^{ad}=\\
 &=[\alpha(\psi_t')]\otimes(r\circ\widetilde{\varphi}\circ i)_!^{ad}=\\
 &=[\alpha(\psi_t')]\otimes i_!^{ad}\otimes (r\circ\widetilde{\varphi})_!^{ad}=\\
 &=i_!^{ad}\otimes [\alpha(\psi_t)]\otimes (r\circ\widetilde{\varphi})_!^{ad}
\end{split}
\end{equation*}
where we used Proposition \ref{compo}  and Lemma \ref{functiso} in the third and the fourth equality respectively.

Since $p'= i\circ p$, we have that
\[
i_!^{ad}\otimes p_!^{ad}\otimes (\varphi_!^{ad})'=i_!^{ad}\otimes [\alpha(\psi_t)]\otimes (r\circ\widetilde{\varphi})_!^{ad}.
\]
If we prove that $i_!^{ad}$ is a KK-equivalence, then since  $p_!^{ad}$ is a KK-equivalence we will obtain 
the equality  $ (\varphi_!^{ad})'=\varphi_!^{ad}$.
To prove this fact, observe that there exists an $\epsilon>0$ such that $k\colon\xi\mapsto \epsilon \xi$ is a submersive embedding from $W$ to $W'$ and such that $k\circ i$ and $i\circ k$ are homotopic to the identity. Let $h_t$ be the homotopy from $i\circ k$ to $\mathrm{id}_{W}$. By Lemma \ref{path!} we have the following equalities
\[
\mathrm{Id}=(\mathrm{id}_{W})_!^{ad}=[\alpha(h_t)]\otimes k_!^{ad}\otimes i_!^{ad}
\]
which give a left inverse for $i_!^{ad}$.
Using  the homotopy from $k\circ i $ to $\mathrm{id}_{W'}$ we obtain the analogous equality which gives a right inverse for $i_!^{ad}$ and proves that it is a KK-equivalence.

Finally let $U$ and $U'$ be two arbitrary tubular neighbourhood  of $X$ in $G$, using the inclusions of $U\cap U'$ into $U$ and $U'$  and the computations above, it is clear that the defining $\varphi_!^{ad}$  using $U$ or $U'$ gives the same class in $KK\left(C^*_r((G_\varphi^\varphi)_{ad}^\circ),C_r^*(G_{ad}^\circ) \right)$.
 \end{proof}

\begin{remark}
	If $\varphi\colon Y\to X$ is a submersion, then
	the elements in definitions \ref{!} and \ref{!trans} coincide.
	Indeed we have that $\{\psi_t\colon\xi\mapsto r\circ\widetilde{\varphi}(t\xi)\}_{t\in[0,1]}$ is a path of submersions, then by Lemma \ref{path!} $[\alpha(\psi_t)]
	\otimes(r\circ \widetilde{\varphi})_!^{ad}=(s\circ \widetilde{\varphi})_!^{ad}$.
	This element is equal $(\varphi\circ p)_!^{ad}$, that
	 by Proposition \ref{compo}, is equal to $p_!^{ad}\otimes\varphi_!^{ad}$.
	This proves that $(p_!^{ad})^{-1}\otimes[\alpha(\psi_t)]
	\otimes(r\circ \widetilde{\varphi})_!^{ad}=\varphi_!^{ad}$, where the element on the left side is the one defined in \ref{!}.
\end{remark}

Now we can verify that our definition behaves well with respect to the composition of transverse maps.

\begin{proposition}
	Let $G\rightrightarrows Z$ be a Lie groupoid.
	Let $f\colon X\to Y$ and $g\colon Y\to Z$ be two smooth maps between smooth manifolds such that $g$ is transverse w.r.t. $G$ and f is transverse w.r.t. $G^g_g$. Then we have that
	\[
	(g\circ f)^{ad}_{!}=f_!^{ad}\otimes g^{ad}_!\in KK\left(C^*_r\left((G^{g\circ f}_{g\circ f})_{ad}^\circ\right),C^*_r\left(G_{ad}^\circ\right)\right).
	\]
\end{proposition}
\begin{proof}
	Since the lower shriek maps for general transverse maps are defined through zig-zags, the idea of the proof is comparing two zig-zags with same starting and ending objects (the zig-zag associated to the composition of $f$ and $g$ and the concatenation of the zig-zags associated to $f$ and $g$),  by means of a third zig-zag that is directly equivalent to them, individually taken.
	
		Let $E$ and $F$ denote $(g\circ f)^*\mathfrak{A}G$ and $g^*\mathfrak{A}G$ respectively. 
		
		Consider the following diagram
\begin{equation}
		\xymatrix{ 
			E\times_Y F\ar[r]^{P}\ar[d]^Q& F\ar[d]^{q}\ar[r]^{l}& Z \\
			E\ar[d]^{p}\ar[urr]^{k}\ar[r]^{h} & Y\ar[ur]_g & \\
			X\ar[ur]_f &  &		\\	
			}
\end{equation}
where $P$ and $Q$ are the obvious projections, $h=r\circ\widetilde{f}$, $k=r\circ(\widetilde{g\circ f})$ and $l=r\circ \widetilde{g}$. 
It is commutative up to homotopy: let $h_t$ be the homotopy between $f\circ p$ and $h$, let $k_t$ be the homotopy between $(g\circ f)\circ p$ and $k$ and let $l_t$ be the homotopy between $g\circ q$ and $l$.
Moreover all the vertical arrows are vector bundle projections, therefore the lower shriek maps associated to them induce KK-equivalences.

Recall that $f_!^{ad}=(p_!^{ad})^{-1}\otimes[\alpha(h_t)]
	\otimes h_!^{ad}$ and $g_!^{ad}=(q_!^{ad})^{-1}\otimes[\alpha(l_t)]
	\otimes l_!^{ad}$ are as in Definition \ref{!trans}. Moreover let $(g\circ f)_!^{ad}$ be given by $(p_!^{ad})^{-1}\otimes[\alpha(k_t)]
	\otimes k_!^{ad}$.
Thanks to the following calculations	
\begin{equation}
\begin{split}
&(p_!^{ad})^{-1}\otimes[\alpha(h_t)]
\otimes h_!^{ad}\otimes(q_!^{ad})^{-1}\otimes[\alpha(l_t)]
\otimes l_!^{ad}=\\
&(p_!^{ad})^{-1}\otimes[\alpha(h_t)]
\otimes (Q_!^{ad})^{-1}\otimes P_!^{ad}\otimes[\alpha(l_t)]
\otimes l_!^{ad}=\\
&(p_!^{ad})^{-1}\otimes (Q_!^{ad})^{-1}\otimes[\alpha(h_t\circ Q)]
\otimes[\alpha(l_t\circ P)]\otimes P_!^{ad}
\otimes l_!^{ad}=\\
&(p_!^{ad})^{-1}\otimes (Q_!^{ad})^{-1}\otimes[\alpha(h_t\circ Q)]
\otimes[\alpha(l_t\circ P)]\otimes (l\circ P)_!^{ad}=\\
&(p_!^{ad})^{-1}\otimes (Q_!^{ad})^{-1}\otimes[\alpha((g\circ h_t\circ Q)*(l_t\circ P))]
\otimes (l\circ P)_!^{ad}=\\
&(p_!^{ad})^{-1}\otimes[\alpha(k_t)]\otimes (Q_!^{ad})^{-1}
\otimes (l\circ P)_!^{ad}=\\
&(p_!^{ad})^{-1}\otimes[\alpha(k_t)]\otimes k_!^{ad}=\\
&
 (g\circ f)_!^{ad}\\
\end{split}
\end{equation}
we obtain the desired equality.
We used:
in the second line the fact that $h\circ Q= q\circ P$ and that $q$ and $Q$ are vector bundle projections;
in the third line \eqref{interw};
in the fourth line  Proposition \ref{compo};
in the fifth line Lemma \ref{concatenation};
in the sixth line  \eqref{interw} one more time;
 finally in the seventh line  Proposition \ref{compo}.

\end{proof}

Now we are going to state another important property of this construction, which is its functoriality with respect to the restriction to the boundary of a manifold with boundary.
Before that let us notice the following facts:
let $\psi \colon Y\to X$ be transverse with respect to $G$ and let $X_1$ be a closed and saturated submanifold of $X$. Then, since $X_1$ is saturated, it follows that $dr(\mathfrak{A}_xG)\subset T_xX_1$ and then that $\psi$ and the inclusion of $X_1$ into $X$ are transverse. This implies that $Y_1:=\psi^{-1}(X_1)$ is a submanifold of $Y$;
 the transversality of $\psi$  with respect to $G$ means that $d\psi (T_yY)+dr(\mathfrak{A}_{\psi(y)}G)=T_{\psi(y)}X$. Consequently, if we consider the intersection with $T_{\psi(y)}X_1$ for $y\in Y_1$, we obtain  $d\psi (T_yY_1)+dr(\mathfrak{A}_{\psi(y)}G)=T_{\psi(y)}X_1$, 
	namely that the restriction of $\psi$ to $Y_1$ is still transverse with respect to $G$.

Let $X$ be a smooth manifold with 
 boundary $\partial X$ and let   $G\rightrightarrows X$ a Lie groupoid transverse to the boundary.
Consider a smooth function $(f,\partial f)\colon(Y,\partial Y )\to (X, \partial X)$ between manifolds with boundary transverse with respect to $G$.

Observe that
the b-groupoids
$(G(X,\partial X))_f^f$ and $G^f_f(Y,\partial Y)$ are isomorphic.
Indeed both of them restrict to $G^f_f$ over $\mathring{Y}$ and it remains to notice that $(G_{|\partial X}\times \RR)^{\partial f}_{\partial f}$ is isomorphic to $(G_{|\partial X})^{\partial f}_{\partial f}\times\RR$ just by definition.

\begin{proposition}\label{functoriality}
In the situation described above we have that the following diagram 
\begin{equation}
\xymatrix@=0.8em{
\cdots\ar[r]&K_*\left(C^*_r(((G_{|\mathring{X}})_f^f)_{ad}^\circ)\right)\ar[r]\ar[dd]^{(f_{|\mathring{Y}})_!^{ad}}&K_*\left(C^*_r(((G(X,\partial X))_f^f)_{ad}^\circ)\right)\ar[r]\ar[dd]^{f_!^{ad}}&K_*\left(C^*_r(((G_{|\partial X})^{\partial f}_{\partial f}\times\RR)_{ad}^\circ)\right)\ar[r]\ar[dd]^{ (f_{|\partial Y})_!^{ad}}&\cdots\\
& & & & \\
\cdots\ar[r]&K_*\left(C^*_r(((G_{|\mathring{X}})_{ad}^\circ)\right)\ar[r]&K_*\left(C^*_r(G(X,\partial X)_{ad}^\circ)\right)\ar[r]&K_*\left(C^*_r((G_{|\partial X}\times\RR)_{ad}^\circ)\right)\ar[r]&\cdots
	}
\end{equation}
commutes.
\end{proposition}
 \begin{proof}
It is sufficient to prove this when $f$ is a submersion, because for general transverse maps  the shriek element is the Kasparov product of the inverse of a submersion, a submersion and a Morita equivalence.
 	
 	Let us prove the commutativity of the second square:
 	we have to prove the equality of 
 	\[
 	[e^t_1]^{-1}\otimes[e_1^u]\otimes\mu_\psi\otimes[\mathrm{ev}_{\partial X}]
 	\]
 	and
 	\[
 	[\mathrm{ev}_{\partial Y}]\otimes[\bar{e}^t_1]^{-1}\otimes[\bar{e}^u_1]\otimes\mu_{\psi''}
 	\]
 	in $KK\left(C^*_r(((G(X,\partial X))_f^f)_{ad}^\circ),C^*_r((G_{|\partial X}\times\RR)_{ad}^\circ)\right)$.

 	Here $\bar{e}$ are evaluation maps for the groupoids restricted to $\partial Y$, whereas we use $e$ for evaluation maps for groupoids over $Y$.
 	Noticing that $[\mathrm{ev}_{\partial Y}]\otimes[\bar{e}^t_1]^{-1}\otimes[\bar{e}^u_1]=[e^t_1]^{-1}\otimes[e^u_1]\otimes[\mathrm{ev}_{\partial Y}]$ and applying Proposition \ref{moritarestriction}, we obtain the commutativity of the second square.
 	For the commutativity of the first one we use a similar argument.
 	
 	Now we are going to prove the commutativity of the square which involves the boundary morphisms of the exact sequences.
First we can restrict our attention to the collar neighbourhoods of the boundaries where our groupoids are isomorphic to $(G_{|\partial X}\times \Gamma([0,1),\{0\}))_{ad}^\circ$ and $((G_{|\partial X})_{\partial f}^{\partial f}\times \Gamma([0,1),\{0\}))_{ad}^\circ$.
 	
 Then recall from Remark \ref{b-boundary} that the boundary morphism of those exact sequences are  given by 
 \[
 \partial_b= [\mathrm{id}\otimes\mathrm{ev}_0]^{-1}\otimes(\mathrm{id}\otimes \beta_{ad})\otimes [\Delta].
 \]
 
Now observe that $(f_{|\partial Y})_!^{ad}=\partial f_!^{ad}\otimes \mathrm{id}$
\[
(\partial f_!^{ad}\otimes \mathrm{id})\otimes[\mathrm{id}\otimes\mathrm{ev}_0]=[\mathrm{id}\otimes\mathrm{ev}_0]\otimes (\partial f_!^{ad}\otimes \mathrm{id}) 
\]
in 
$KK\left(C^*_r(
((G_{|\partial X})^{\partial f}_{\partial f})_{ad}^\circ)\otimes C^*_r(\RR),C^*_r((G_{|\partial X})_{ad}^\circ)\otimes C^*_r(\RR_{ad})\right)$.

Moreover \[
(\partial f_!^{ad}\otimes \mathrm{id})\otimes[\mathrm{id}\otimes\beta_{ad}]=[\mathrm{id}\otimes\beta_{ad}]\otimes (\partial f_!^{ad}\otimes \mathrm{id}) 
\]
in 
$KK^1\left(C^*_r(
((G_{|\partial X})^{\partial f}_{\partial f})_{ad}^\circ)\otimes C^*_r(\RR),C^*_r((G_{|\partial X})_{ad}^\circ)\otimes C^*_r((0,1)\times(0,1)_{ad})\right)$.

Finally since $\Delta$ is a restriction with respect to the parameters of the adiabatic deformations, it is clear from the nature of $\Delta$ and the shriek construction that 

 \[
 (\partial f_!^{ad}\otimes \mathrm{id})\otimes[\Delta]=[\Delta]\otimes(f_{|\partial Y\times(0,1)})_!^{ad}
 \]
 in 
 $KK\left(C^*_r(
 ((G_{|\partial X})^{\partial f}_{\partial f})_{ad}^\circ)\otimes C^*_r((0,1)\times(0,1)_{ad}^\circ),C^*_r((G_{|\partial X}\times (0,1)\times(0,1))_{ad}^\circ)\right)$. And this complete the proof.
 \end{proof}

Finally we have the main result of this section.
\begin{theorem}\label{th-comm}
	If $\psi\colon Y\to X$ is transverse with respect to $G\rightrightarrows X$,  then the following diagram
si commutative:
\begin{equation}\label{shriek}
\xymatrix{\cdots\ar[r] & K_*\left(C^*_r(G^\psi_\psi\times (0,1))\right)\ar[r]\ar[d]^{\mu_{\psi}} &
	K_*\left(C^*_r((G^\psi_\psi)^\circ_{ad})\right)\ar[r]^{[\mathrm{ev}_0]_*}\ar[d]^{\psi_!^{ad}} & K_*\left(C^*_r(\mathfrak{A}(G^\psi_\psi))\right)\ar[r]\ar[d]^{d\psi_!}& \cdots\\
	\cdots\ar[r] &
	K_*\left(C^*_r(G\times (0,1))\right)\ar[r] &
	K_*\left(C^*_r(G^\circ_{ad})\right)\ar[r]^{[\mathrm{ev}_0]_*} & K_*\left(C^*_r(\mathfrak{A}(G))\right)\ar[r]& \cdots\\
}
\end{equation} 
where $\mu_{\psi}$ is the KK-element given in Definition \ref{morita} and $d\psi_!\in KK\big(C^*_r(\mathfrak{A}(G^\psi_\psi)),C^*_r(\mathfrak{A}(G))\big)$ is the KK-class obtained in the obvious way, as for $\psi_!^{ad}$, but restricting the process to the Lie algebroids. Furthermore, the commutativity of the diagram still holds in the KK-theory framework.
\end{theorem}

\begin{remark}\label{rmkshriek}
	Notice that one also has a wrong-way functoriality for the adiabatic deformation up to $t=1$ included.
	It is given by the same construction and it enjoys the same properties. Moreover there is a commutative diagram analogous to \eqref{shriek}
	for the exact sequence
	\[
	\xymatrix{ 0\ar[r] & C^*_r(G\times(0,1])\ar[r] & C^*_r(G_{ad})\ar[r]^{\mathrm{ev}_0} & C^*_r(\mathfrak{A}(G))\ar[r] & 0}.
	\]
	
\end{remark}

\section{Lie groupoids and secondary invariants}

 \subsection{Pseudodifferential operators on Lie groupoids}
In this section we are going to recall the definition of a pseudodifferential operator on a Lie groupoid. For more details the reader is referred to 
\cite{NWX, vas}.   
Consider the
following data:
\begin{itemize}
	\item a smooth embedding  $\theta\colon U\to \mathfrak{A}^*G$,  where $ U$ is a tubular neighbourhood of $G^{(0)}$
	in $G$, such that
	$\theta(G^{(0)})=G^{(0)}$, $(d\theta)|_{G^{0}}=\hbox{Id}$ and $\theta(\gamma)\in\mathfrak{A}^*_{s(\gamma)}G$
	for all $\gamma\in U$;
	\item a smooth compactly supported map $\phi:G\to \RR_{+}$ such that
	$\phi^{-1}(1)=G^{(0)}$;
	\item a polyhomogeneous symbol $a$ on $ \mathfrak{A}^*G$ of order $m\in\ZZ$, that is $a\sim \sum_{k=0}^{+\infty} a_{m-k}$ with $a_{j}(x,\xi)$ homogeneous of order $j$ in $\xi$.
\end{itemize}
Then a pseudodifferential $G$-operator $P$ is obtained by the
formula:
\[ Pu(\gamma)=
\int_{\gamma'\in G_{s(\gamma)}\,,\, \xi\in \mathfrak{A}^{*}_{r(\gamma)}(G)}
e^{i\theta(\gamma'\gamma^{-1})\cdot \xi}a(r(\gamma),
\xi)\phi(\gamma'\gamma^{-1})u(\gamma')d\gamma'd\xi\]
where $u\in C^\infty_c(G)$ and we have fixed a scalar product on the Lie algebroid.
If $m>0$, then we  obtain an unbounded multiplier of $C^\infty_c(G)$;
if $m=0$, the operator $P$ is an element in  the multiplier algebra of the reduced groupoid C*-algebra;
finally, if $m<0$, then $P$ lies in the reduced groupoid C*-algebra.

\begin{examples}\label{expdo}
Let us recall some examples of 0-order pseudodifferential operators on Lie groupoids.
	\begin{enumerate}
		\item If $G= X\times X\rightrightarrows X$ is the pair groupoid, where $X$ is a compact smooth manifold, then a 0-order $G$-$\Psi$DO is simply a 0-order $\Psi$DO on $X$.
		
		\item Let $p: X\to Z$ a submersion, and 
		$G=X\times_{Z}X = \{(x,y)\in X\times X\, | \,p(x)=p(y)\}$ the
		associated subgroupoid of the pair groupoid $X\times X$.  Then a 0-order $G$-$\Psi$DO is given by families $(P_z)_{\{z\in Z\}}$ of  0-order $\Psi$DOs on $p^{-1}(z)$.

		\item 
		Let $G$ be the fundamental groupoid of a compact smooth
		manifold $M$ with fundamental group $\pi_1(M,x_0)=\Gamma$. Recall that if we denote
		by $\widetilde{M}$ a universal covering of $M$ and let $\Gamma$ act by covering transformations,
		then $G^{(0)} = \widetilde{M} /\Gamma = M$, $G =  \widetilde{M}\times_{\Gamma}\widetilde{M} $ and the source and the range maps are the two
		projections. Then a 0-order $G$-$\Psi$DO  is a  properly supported
		$\Gamma$-invariant 0-order classical $\Psi$DO on the universal covering $\widetilde{M}$ of $M$.
		
		\item Let $G=E\rightrightarrows X$ be the total space of a vector
		bundle $p\colon E\to X$ over a compact smooth manifold $X$, with $r=s=p$ and $(x,v)\cdot(x,w)=(x,v+w)$.  If $P$ is a pseudodifferential $G$-operator: 
		$$ 
		Pf(v) = \int_{w\in E_x} k_P(v-w)f(w)  
		$$
		Thus, for all $x\in X$, $P_x$ is a translation-invariant convolution operator on
		the linear space $E_x$ such that the underlying distribution $k_P$ identifies with the Fourier transform of a symbol on $E$. Consequently 
		we have that  a 0-order $G$-$\Psi$DO is given by a smooth function on $BE^*$, the unit ball bundle of $E^*$.
		
		\item If $G$ is the holonomy groupoid of a foliation $\mathcal{F}$ on a smooth manifold $X$, then a 0-order $G$-$\Psi$DO  is just  a leafwise 0-order $\Psi$DO on $(X,\mathcal{F})$.
		
	\end{enumerate} 
\end{examples}

\subsection{The $\varrho$ classes for Lie groupoids} \label{section-rho}

Before giving the precise definition $\varrho$-classes as elements of the K-theory of $C^*_r(G_{ad}^\circ)$, we are going to informally explain the idea of the construction. 
 We are going to think of $K_*(C^*_r(G_{ad}^\circ))$ as a group of relative K-theory in the following sense (see \cite{SkExt} for more details).
Let $A$ and $B$ be two separable C*-algebras and let $\partial$ be an element in $KK(A,B)$.
Using for instance the realization of KK-groups \`{a} la Cuntz (see \cite{C87}), one can always find a C*-algebra $A'$ and two *-homomorphisms $\varphi\colon A'\to A$ and $\psi\colon A'\to B$ such that 
\begin{itemize}
\item $\varphi$ induces a KK-equivalence between $A'$ and $A$;
\item we can decompose $\partial$ as $[\varphi]^{-1}\otimes_{A'}[\psi]$, the Kasparov product of the inverse of the KK-equivalence $[\varphi]$ and $[\psi]$. 
\end{itemize}

Hence one can see $\partial$, up to the KK-equivalence between $A$ and $A'$, as the boundary map for the long exact sequence in KK-theory associated to the following short exact sequence
\[
\xymatrix{0\ar[r] & B\otimes(0,1)\ar[r] & C_\psi(A',B)\ar[r]^(.65){pr_{A'}} & A'\ar[r]& 0}
\]
where $C_\psi(A',B)=\{a\oplus f\in A'\oplus B[0,1)\,|\, f(0)=\psi(a)\}$ is the mapping cone C*-algebra of $\psi$. We define the relative K-theory of $\partial$ as the  K-theory of this mapping cone.

So a $\varrho$ class will be defined as a class in this K-group. More precisely, if we identify $K_*(C_\psi(A',B))$ with the KK-group $KK^*(\CC,C_\psi(A',B))$, such an element can be given by the following data:
\begin{itemize}
\item a Kasparov $\CC$-$A'$ bimodule $(\mathcal{H},F)$;
\item a Kasparov $\CC$-$B[0,1)$ bimodule $(\mathcal{E}_t,G_t)$ such that $(\mathcal{E}_0,G_0)=(\mathcal{H}\otimes_\psi B,F\otimes_\psi 1)$ and $(\mathcal{E}_1,G_1)$ is degenerate. 
\end{itemize}
See \cite[Proof of Theorem 1.1]{CunSk}.

\begin{remark}
	Of course one can equivalently work in the unbounded setting, following \cite{BJ}.
\end{remark}
%
Let us quickly recall the construction of the adiabatic index of a 0-order 
  elliptic pseudodifferential $G$-operator $P$.
Its principal symbol $\sigma$ defines a class in the group $KK(\CC,C_0(\mathfrak{A}^*(G)))$ that is isomorphic to $KK(\CC,C^*_r(\mathfrak{A}(G)))$, by means of the Fourier transform after fixing a scalar product.

We know that $ev_0\colon C^*_r(G_{ad})\to C^*_r(\mathfrak{A}(G)) $ is a KK-equivalence. Let us give an explicit description of the inverse of the map induced in KK-theory: to the symbol $\sigma\in C_0(\mathfrak{A}^*(G))$, we can associate a symbol $\sigma_{ad}$ on $\mathfrak{A}^*(G)\times[0,1]$, the Lie algebroid of the adiabatic deformation, given by $\sigma_{ad}(\xi,t):=\sigma(\xi)$; hence we obtain the unbounded regular operator $P_{ad}$ (see \cite{vas}) on the $C^*_r(G_{ad})$-module $C^*_r(G_{ad})$ defined by

\begin{equation}\label{adoperator1}
(P_{ad} f )(\gamma,t)=\int_{\xi\in\mathfrak{A}^*(G)_{r(\gamma)}}\int_{\gamma'\in G_{s(\gamma)}}e^{\frac{i\langle\exp^{-1}(\gamma'\gamma^{-1})|\xi\rangle}{t}}\chi(\gamma'\gamma^{-1})\sigma(r(\gamma),\xi)f(\gamma',t)\frac{d\xi d\gamma'}{t}
\end{equation}
for $t\neq0$ and
\begin{equation}\label{adoperator2}
(P_{ad} f )(x,V,0)=\int_{\xi\in\mathfrak{A}^*(G)_{x}}\int_{V'\in\mathfrak{A}^(G)_{x}}e^{i\langle V'|\xi\rangle}\chi(\exp(V))\sigma(x,\xi)f(x,V-V',0)d\xi dV'
\end{equation}
for $t=0$, with $f\in C^\infty_c(G_{ad},\Omega^{\frac{1}{2}})$ (notice that at $t=1$ we obtain $P$ and that $P_{ad}$ is a deformation of $P$ to the Fourier transform of its symbol).

Here we have chosen an exponential map $\exp\colon U\to W$, from a neighbourhood of the zero section in the algebroid $\mathfrak{A}(G)$ to a tubular neighbourhood $W$ of $X$ in $G$, and a cut-off function $\chi$ with support in $W$. Furthermore we can also make this construction with coefficients in any vector bundle $E$ over $G$.
The operator $P_{ad}$ defines a class in $KK(\CC,C^*_r(G_{ad}))$ such that the evaluation at $0$ gives the class of $\sigma$ and the evaluation at $1$  gives the analytic $G$-index of $P$ $\mbox{Ind}_G(P)=[\sigma(P)]\otimes[ev_0]^{-1}\otimes[ev_1]$ in the group $KK(\CC, C^*_r(G))$.

Notice that $C^*_r(G_{ad}^\circ)$ is isomorphic to the mapping cone $C_{ev_1}(C_r^*(G_{ad}),C^*_r(G))$ of the evaluation at $1$.
Indeed if $(f_t)_{t\in[0,1]}$ is an element in $C^*_r(G_{ad}^\circ)$, then the pair 
$((f_{t/2})_{t\in[0,1]}, (f_{(t+1)/2})_{t\in[0,1]} )$ is the image of $(f_t)_{t\in[0,1]}$ in $C_{ev_1}(C_r^*(G_{ad}),C^*_r(G))$ through this isomorphism.

 Hence, if $P_t$ is a path of  $G$-operators such that $P_0=P$ and $P_1$ is invertible,
 through the above isomorphism, we obtain a class in $KK(\CC,C^*_r(G_{ad}^\circ))$.
This will be the home of the secondary invariants that we will study in the following sections. They are called secondary because they arise when the index, the primary invariant, vanishes.
\begin{definition}\label{rhogr}
	In the situation described above, let us denote the $\varrho$-invariant associated to $\sigma$ and $P_t$ as 
	\[
	\varrho(\sigma, P_t)\in KK(\CC,C^*_r(G_{ad}^\circ)).
	\]
\end{definition}

\begin{remark}\label{ubrho}
	
	If $P$ is not bounded and it is homotopic to an invertible operator through a path $P_t$, then we can construct a $\varrho$-class in the following way.
	Let $\sigma$ be the symbol of $P$, consider the symbol $\sigma_{ad}$ as above and construct the unbounded $G_{ad}$-operator $P_{ad}$.
	This operator and the path $P_t$ fit together and we obtain an unbounded 
$C^*_r(G_{ad})$-operator $P'_{ad}$.

	Let $\psi(s)=s\cdot(1+s^2)^{\frac{1}{2}}$ and observe that 
	$\psi_t(s):=\psi(\frac{s}{1-t})$ is a path of continuous functions, such that $\psi_0=\psi$ and $\psi_t(s)$ goes to $\mathrm{sign}(s)$ when $t$ goes to $1$.
	
	Now, since $P_1$ is invertible and there is a gap in its spectrum near zero, the concatenation of $\psi(P_{ad}')$ and $\psi_s(P_1)$, suitably parametrized,
	gives a bounded Fredholm $G_{ad}^\circ$-operator and we denote its class in $KK^*(\CC,C^*_r(G_{ad}^0))$ by
	$\varrho(\sigma,P_t)$.

\end{remark}

\subsection{Cobordism relations}

In this section we are going to investigate the relation between the $\varrho$-invariants associated to two cobordant Lie groupoids. 
Let $W$ be a smooth manifold with boundary $\partial W$ and let $G(W,\partial W)\rightrightarrows W$ be the $b$-groupoid of a Lie groupoid $G$ transverse to the boundary, as in Definition \ref{bcalculus}.
Let $P$ be an elliptic pseudodifferential $G(W,\partial W)$ operator and denote its restriction to the boundary by $P^{\partial}$ (this is a $G_{|\partial W}\times \RR$-operator).
\begin{definition}
Assume that there exists a homotopy $P^\partial_t$ from $P^\partial_0=P^{\partial}$ to an invertible operator $P^{\partial}_1$. 
Then we obtain the following classes:
\begin{itemize}
	\item a secondary invariant $\rho(P^{\partial}_t)\in K_*\left(C^*_r((G_{|\partial W})_{ad}^\circ\times \RR)\right)\simeq K_{*+1}\left(C^*_r((G_{|\partial W})_{ad}^\circ)\right)$;
	\item a class $[P_{ad}^{\mathcal{T}}]\in K_*\left(C^*_r(\mathcal{T}_{nc}G(W,\partial W))\right)$ , defined by the symbol of $P$ and the homotopy $P^\partial_t$;
	recall by Definition \ref{nctangent}
	that $\mathcal{T}_{nc}G(W,\partial W)$ is the union of $\mathfrak{A}(G_{|\mathring{X}})$, where we define the symbol, and $(G_{|\partial X})_{ad}^{\circ}\times \RR $, where the secondary invariant is defined; they glue together at $\mathfrak{A}(G_{\partial X})$ by definition;
	\item
	indeed, using the KK-equivalence in the Lemma \ref{bkk}, we can extend it to a class $[P_{ad}^{\mathcal{F}}]$ in $ K_*\left(C^*_r(G(W,\partial W)_{ad}^{\mathcal{F}})\right)$, whose restriction to the boundary is  $\rho(P^{\partial}_t)$; 
	\item  finally we get a class  $\mbox{Ind}(P,P^\partial_t)\in K_*(C^*_r(G_{|\mathring{W}}))$. This is the generalized Fredholm index of $P$ associated to the perturbation on the boundary $P^\partial_t$, obtained as the Kasparov product $[P_{ad}^{\mathcal{F}}]\otimes [ev_1]$, see for instance \cite{PZ} for a general treatment of Fredholm index of fully elliptic operators in the context of Lie groupoids.
\end{itemize}
\end{definition}
The following elementary result is useful to prove the main formula of this section.

\begin{lemma}\label{deloc}
Let
	\begin{itemize}
		\item[1)] $\xymatrix{0\ar[r]& J_B\ar[r]& A\ar[r]^\beta& B\ar[r]& 0}$,
		\item[2)] $\xymatrix{0\ar[r]& J_C\ar[r]& A\ar[r]^\gamma& C\ar[r]& 0}$,
	\end{itemize}
	be exact sequences of C*-algebras.
	Assume that $J_B+J_C=A$. 	
We have the following exact sequences:
	\begin{itemize} 
		\item[3)] $\xymatrix{0\ar[r]& J\ar[r]& J_C\ar[r]& B\ar[r]& 0}$,
		\item[4)] $\xymatrix{0\ar[r]& J\ar[r]& J_B\ar[r]& C\ar[r]& 0}$,
		\item[5)] $\xymatrix{0\ar[r]& J\ar[r]& A\ar[r]^(.4){\beta\oplus\gamma}& B\oplus C\ar[r]& 0}$,
	\end{itemize}
	
	where $J=J_B\cap J_C$.
	Let $\partial_B$ and $\partial_C$ be the boundary homomorphisms associated to the exact sequences 3) and 4) respectively. Then $\partial$, the boundary homomorphism associated to the exact sequences 5), is such that 
	\[
	\partial\colon x\oplus y\mapsto \partial_B(x)+\partial_C(y)
	\]
	where $x\in K_n(B)$, $y\in K_n(C)$ and $\partial_B(x)+\partial_C(y)\in K_{n+1}(J)$.

	Moreover if $\beta$ and $\gamma$ admit completely positive sections, we have that 
	$\partial= p_B^*\partial_B+p_C^*\partial_C\in KK^1(B\oplus C, J)$, where $p_B$ and $p_C$ are the projections from $B\oplus C$ to $B$ and $C$ respectively and $\partial_B$ and $\partial_C$ are elements of $KK^1(B,J)$ and $KK^1(C,J)$ respectively.
\end{lemma}
\begin{proof}

By the following commutative diagram
		\[
		\xymatrix{ 0\ar[r]& J\ar[r]\ar[d]& J_C\ar[r]\ar[d]& B\ar[r]\ar[d]& 0\\
			0\ar[r]& J\ar[r]& A\ar[r]^(.4){\beta\oplus\gamma}& B\oplus C\ar[r]& 0}
		\]
we deduce that $\partial(x\oplus0)=\partial_B(x)$.
By the following one
	\[
	\xymatrix{ 0\ar[r]&  J\ar[r]\ar[d]&  J_B\ar[r]\ar[d]&  C\ar[r]\ar[d]& 0\\
		0\ar[r]& J\ar[r]& A\ar[r]^(.4){\beta\oplus\gamma}& B\oplus C\ar[r]& 0}
	\]
we deduce that $\partial(0\oplus y)=\partial_C(x)$.
	Then we obtain that $\partial\colon x\oplus y\mapsto \partial_B(x)+\partial_C(y)$.
	The second part follows from the previous calculations.
\end{proof}

We want to apply this Lemma to the following situation.
We are interested in the following C*-algebras:
\[A=C^*_r\left(G(W,\partial W)_{ad}^{\mathcal{F}}\right)\,,\, B=C^*_r\left(G_{|\mathring{W}}\right)\,,\, 
C=C^*_r\left((G_{|\partial W})_{ad}^\circ\times \RR\right).
\]
We obtain the following diagram of exact sequences (see Lemma \ref{montexact} and the proof of Lemma \ref{bkk}):
\[
\xymatrix{ & 0\ar[d] & 0\ar[d] & & \\
	0\ar[r] &  C^*_r\left((G_{|\mathring{W}})_{ad}^\circ\right)\ar[r]\ar[d] & C^*_r\left((G_{|\mathring{W}})_{ad}\right)\ar[d]\ar@/^1pc/[dr]& & \\
	0\ar[r] &  C^*_r\left(G(W,\partial W)_{ad}^{0}\right)\ar[r]\ar@/_1pc/[dr] & C^*_r\left(G(W,\partial W)_{ad}^{\mathcal{F}}\right)\ar[d]\ar[r]& C^*_r\left(G_{|\mathring{W}}\right)\ar[r]\ar[d] & 0 \\
	& & C^*_r\left((G_{|\partial W})_{ad}^\circ\times\RR\right)\ar[d]\ar[r]& 0 & \\
	& & 0 & & 
}
\]
that fits with the hypothesis of Lemma \ref{deloc}. Here the quotient arrows are given by restriction and the other arrows are given by inclusions. 
Using the notations in the Lemma, we have that
\begin{itemize}
	\item $\partial_B\colon K_{*+1}\left(C^*_r\left(G_{|\mathring{W}}\right)\right)\to K_{*}\left(C^*_r\left((G_{|\mathring{W}})_{ad}^\circ\right)\right)$ is
	given by $b\mapsto b\otimes \mbox{Bott}\otimes[i]$, where $i\colon C^*_r\left(G_{|\mathring{W}}\times(0,1)\right)\to C^*_r\left((G_{|\mathring{W}})_{ad}^\circ\right)$ is the obvious inclusion;
	this is true because, up to Bott isomorphism, this exact sequence corresponds to the mapping cone exact sequence associated to the *-homomorphism $\mathrm{ev}_1\colon C^*_r((G_{|\mathring{W}})_{ad})\to C^*_r(G_{|\mathring{W}})$;

	\item $\partial_C\colon K_{*+1}\left(C^*_r\left((G_{|\partial W})_{ad}^\circ\times\RR\right)\right)\to K_{*}\left(C^*_r\left((G_{|\mathring{W}})_{ad}^\circ\right)\right)$ is as in Remark \ref{b-boundary}.
\end{itemize}

\begin{theorem}[Delocalized APS index theorem for Lie groupoids]\label{APSdelocgr}
Consider the following C*-algebras
	\[A=C^*_r\left(G(W,\partial W)_{ad}^{\mathcal{F}}\right)\,,\, B=C^*_r\left(G_{|\mathring{W}}\right)\,,\, 
	C=C^*_r\left((G_{|\partial W})_{ad}^\circ\times \RR\right).
	\]
	Using the notations of Lemma \ref{deloc}, we have the following equality
	\[
	\partial_C\left(\varrho(P^{\partial}_t)\right)=-\partial_B\left(\mathrm{Ind}(P,P^\partial_t)\right)\in K_{*+1}\left(C^*_r((G_{|\mathring{W}})_{ad}^\circ)\right).
	\]
\end{theorem}
\begin{proof}
	Since the pair $\mathrm{Ind}(P,P^\partial_t)\oplus \varrho(P_t^\partial)\in K_*(B\oplus C)$ is the image of $[P^\mathcal{F}_{ad}]\in K_*(A)$, by the exactness of the associated exact sequence $\partial(\mathrm{Ind}(P,P^\partial_t)\oplus \varrho(P_t^\partial))=0$.
Then the formula is an easy consequence of Lemma \ref{deloc}.
\end{proof}

\begin{remark}\label{cylinder}
	If $W=X\times[0,1]$ and $G(W,\partial W)=G\times\Gamma([0,1],\{0,1\})$, then the boundary map in KK-theory associated to the following exact sequence
	\begin{equation}\label{eqcyl}
	\xymatrix{0\ar[r]& C^*_r\left(G\times (0,1)\times(0,1 )\right)\ar[d]&  \\
		& C^*_r\left(G\times\Gamma([0,1],\{0,1\})\right)\ar[r] & C^*_r\left(G\times\RR\times\{0\}\right)\oplus C^*_r\left(G\times\RR\times\{1\}\right)  \ar[r] & 0}
	\end{equation}
	is given by 
	\begin{equation}\label{apsdeloccyl}
	\partial( x_0\oplus x_1)= x_0\otimes \mbox{Bott}_0^{-1}+x_1\otimes\mbox{Bott}_1^{-1},
	\end{equation}
	where $\mbox{Bott}_i$ is the Bott element for  $C^*_r\left(G\times\RR\times\{i\}\right)$, defined as the boundary map in Lemma \ref{cone}.
\end{remark}

\subsection{The Signature operator}

In \cite{HilSk} the authors give a direct proof of the fact that the K-theory classes of the higher signatures are homotopy invariant. They also prove it in the case of foliations, using a method that can be easily presented in a more abstract way for any Lie groupoid.

Let $G\rightrightarrows X$ be a Lie groupoid. Assume that  its Lie algebroid is oriented of rank $n$ and let  $\Lambda_\CC\mathfrak{A}^*(G)$ be the exterior algebra of $\mathfrak{A}^*(G)$. We can construct a right $C^*_r(G)$-module $\mathcal{E}(G)$ as the completion of $C_c^{\infty}(G,\Lambda_\CC\ker ds^*\otimes s^*\Omega^{\frac{1}{2}}(\mathfrak{A}(G)))$.
Furthermore we can define the following $C^*_r(G)$-valued quadratic form 
\begin{equation}\label{quadratic}Q(\xi,\zeta)(\gamma)= m_*\left(\overline{ p_1^*(i^*\xi)}\wedge p_2^*\zeta\right)(\gamma),\end{equation}
where  $m,p_k\colon G^{(2)}\to G$ are such that $m\colon (\gamma_1,\gamma_2)\mapsto \gamma_1\cdot\gamma_2$ and $p_k\colon (\gamma_1,\gamma_2)\mapsto \gamma_k$ (and $i$ is the inversion map of the groupoid). 
If $T\in\mathcal{L}(\mathcal{E}(G))$, let us denote $T'$ its adjoint with respect to the quadratic form $Q$ (i.e. $Q(T\xi,\zeta)=Q(\xi,T'\zeta)$ for any $\xi,\zeta\in \mathcal{E}(G)$).

The quadratic form $Q$ is regular in the sense of \cite[Definition 1.3]{HilSk}, by means of the operator $T$, given by
\[
T\alpha= i^{-\partial(\alpha)(n-\partial(\alpha))}*\alpha,
\]
where $*$ is the Hodge operator of $\ker ds^*$ associated to a smooth hermitian structure on it.

Consider the $s$-fiberwise exterior derivative operator  on $C_c^\infty\left(G,\Lambda_\CC\ker ds^*\otimes s^*\Omega^{\frac{1}{2}}(\mathfrak{A}(G))\right)$, that is closable. Let us still denote with $d_0$ its closure: it defines a regular operator on $\mathcal{E}(G)$.
We have that $\mbox{Im}d_0\subset\mbox{dom}d_0$ and then that $d_0^2=0$.
Now put $d\xi=i^{\partial\xi}d_0\xi$. 
The  regular operator $D_G=d+d^*$ is an elliptic and self-adjoint differential $G$-operator on $\mathcal{E}(G)$.
\begin{definition}
	The class $[\mathcal{E}(G),D_G]\in K_*(C^*_r(G))$, defined as in \cite{BJ}, is the analytic $G$-signature of $X$.
\end{definition}

A morphisms of Lie groupoids $\varphi\colon G\to H$ between Lie groupoids with oriented Lie algebroids is said to be oriented if its differential induces an orientation preserving map between the two algebroids.

\begin{definition}\label{homotopy}
	Let $G\rightrightarrows X$ and $H\rightrightarrows Y$ be two Lie groupoids.
	A morphism $\varphi\colon H\to G$ is a oriented homotopy equivalence  if there are an oriented morphism $\psi\colon G\to H$ and maps
	$T\colon X\times[0,1]\to G$ and $S\colon Y\times [0,1]\to H$, such that 
	\begin{itemize}
				\item $r\circ T_t(x)$ is equal to $x$ for all $t$ and $T_0= id_X$;
		\item for any $x\in X$ we have that $s(T_1(x))=\varphi\circ\psi(x)$,
		\item for any $\gamma\in G^x_{\varphi\circ\psi(x)}$ we have that $T_1(x)\cdot(\varphi\circ\psi(\gamma))=\gamma\cdot T_1(\varphi\circ\psi(x))$,

	\end{itemize}
	and similarly for $S$ and $\psi\circ\varphi$.
\end{definition}

\begin{remark}
	Indeed this is a strong equivalence of groupoids, see \cite[5.4]{MM}, with natural transformations homotopic to identities. In particular this implies that 
	$H=G^\varphi_\varphi$, see \cite[Proposition 5.11]{MM}.
\end{remark}

Of course a homotopy equivalence $\varphi$ between two groupoids gives a Morita equivalence, whose  imprimitivity bimodule is given by $\mu_\varphi$ as in Subsection \ref{pb}.

Now let us fix a oriented homotopy equivalence $\varphi$ between $H$ and $G$.
Consider the Lie groupoid $L:=G^{\varphi\cup\mathrm{id}_X}_{\varphi\cup\mathrm{id}_X}\rightrightarrows Y \cup X$
 and  the $C^*_r(L)$-module  $\mathcal{E}(L)$, that is the completion of $C_c^\infty(L, \Lambda_\CC\ker ds^*\otimes s^*\Omega^{\frac{1}{2}}(\mathfrak{A}(L)))$.
 We can see an element in $\mathcal{E}(L)$ as a $2\times2$ matrix in $\begin{pmatrix}\mathcal{E}(G)&\mathcal{E}(G_\varphi)\\\mathcal{E}(G^\varphi) &\mathcal{E}(G^\varphi_\varphi)\end{pmatrix}$,
 where the notation is self-explenatory.
 Then the $L$-operator $d_L$ given by the exterior derivative is a matrix $\begin{pmatrix}d_G & 0\\0 &-d_{G^\varphi_\varphi}\end{pmatrix}$.
 
 	Put $(\mathcal{E}_1,Q_1,D_1)=(\mathcal{E}(G_\varphi) \oplus\mathcal{E}(G),Q, d_G)$ and $(\mathcal{E}_2,Q_2,D_2)=(\mathcal{E}(G^\varphi_\varphi)\oplus\mathcal{E}(G^\varphi),Q, d_{G^\varphi_\varphi})$.
 	We want to construct an operator in $\mathcal{L}(\mathcal{E}_1,\mathcal{E}_2)$ that satisfies the hypotheses of \cite[Lemma 2.1]{HilSk}. The following material is from \cite{HilSk} and \cite[Section 2]{Wahl}.

 		Let $E$ be a vector bundle over $Y$ such that $E\oplus\varphi^*\mathfrak{A}G= Y\times \RR^k$ and let $\pi\colon Y\times \RR^k\to\varphi^*\mathfrak{A}G$ be the projection.
 			Let $I$ denote the open interval $(-1,1)\subset\RR$.
 		Set $p\colon Y\times I^k\to X$ as the map given by
 		\[
 		p\colon(y,\underline{t})\mapsto r(\exp_{\varphi(y)}(\pi(\underline{t}))).
 		\]
 		Notice that
 	since $\varphi$ is transverse, $p$ is a submersion, 
 		 moreover $p_{|Y\times\{0\}}=\varphi$.
 		
 		\begin{remark} 
 		
 			Then observe that
 			$G^{p}$ is isomorphic to $G^\varphi\times I^k$ through te map
 			\[
 			(y,\underline{t},\gamma)\mapsto (y,\exp_{\varphi(y)}(\pi(\underline{t}))^{-1}\cdot \gamma,\underline{t}),
 			\]  and $G^p_\varphi$ is isomorphic to $G^\varphi_\varphi\times I^k$ in a similar way ($G^p_\varphi$ is defined as $(Y\times I^k){}_p\times_r G{}_s\times_\varphi Y$).
 			
 		Moreover consider the following two natural maps
 		\begin{itemize}
 			\item $p\colon G^p_\varphi\cup G^{p}\to G_\varphi\cup G$;
 			\item $q\colon G^p_\varphi\cup G^{p}\to G_\varphi^\varphi\cup G^\varphi$;
 		\end{itemize}
 		they are restrictions of Lie groupoids morphisms which are submersive.
 	
 	Because of that, the pull-back of forms through $p$   extends to a  bounded and adjointable operator  from $\mathcal{E}(G)\oplus\mathcal{E}(G_\varphi)$ to $\mathcal{E}(G^p)\oplus\mathcal{E}(G^p_\varphi)$. Furthermore the push-forward of forms through $q$ also extend to a bounded and adjointable operator from $\mathcal{E}(G^p)\oplus\mathcal{E}(G^p_\varphi)$ to $\mathcal{E}(G^\varphi)\oplus\mathcal{E}(G^\varphi_\varphi)$ and  it corresponds to the integration of forms over $I^k$.
 		\end{remark}

 		\begin{definition}
 			Let $v$  be a smooth volume form on $I^k$ of volume 1. Then define 
 	$$\mathcal{T}_v(p)\colon \mathcal{E}(G)\oplus\mathcal{E}(G_\varphi)\to \mathcal{E}(G^\varphi)\oplus\mathcal{E}(G^\varphi_\varphi)$$
 	as $\xi\mapsto q_*(v\wedge p^*(\xi))=\int_{I^k}v\wedge p^*(\xi)$.
 		\end{definition}
 		
 		\begin{remark}\label{HSa}
 			Let us state some properties of $\mathcal{T}_v(p)$.
 			\begin{enumerate}
 				
 			\item Since $v$ is smooth we have that $\mathcal{T}_v(p)Dom (d_G)\subset Dom(d_{G^\varphi_\varphi})$.
 			\item  Moreover by the standard properties of the pull-back and the push-forward and the fact that $v$ is closed of volume 1, we have that
 			$\mathcal{T}_v(p)\circ d_G= d_{G^\varphi_\varphi}\circ\mathcal{T}_v(p)$. 
 			\item There is a stabilization process, namely $\mathcal{T}_v(p)=\mathcal{T}_{v\wedge w}(p\times 1)$, where $p\times 1\colon Y\times I^k\times I^l\to X$ is the obvious map and $w$ is a volume form on $I^l$ of volume 1.
 			\item Let $p_t\colon Y\times I^k\to X$ be a path of maps as above, that is $p_t$ is a submersion and  the restriction  of $p_t$ to  $Y\times\{0\}$ is equal to $\varphi$ for all $t\in [0,1]$. 
 			Let $s_p$ be the operator such that $\xi\mapsto\int_0^1\iota_{\partial_t}\alpha(p_t)^*p_t^*\xi dt$, where $\alpha_t$ is the isomorphism of Lie groupoids defined in Lemma \ref{isopull}. By standard arguments we have that 
 			$\alpha(p_t)^*p_1^*-p_0^*= d_{G_{p_0}^{p_0}}s_p+s_p d_G$.
 			Now taking the wedge with the volume form and integrating over $I^k$, we get the following formula
 			\begin{equation}
 			\mathcal{T}_v(p_1)(\xi)-\mathcal{T}_v(p_0)(\xi)= d_{G^\varphi_\varphi}\int_{I^k} v\wedge s_p(\xi)+\int_{I^k }v\wedge s_p(d_G\xi).
 			\end{equation}
 			Observe that we have  used the fact, by the construction of $\alpha(p_t)$, that
 			$$\int_{I^k}v\wedge \alpha(p_t)^*p_1^*(\xi)=\int_{I^k} \alpha(p_t)^*(v\wedge p_1^*(\xi))=\int_{I^k} v\wedge p_1^*(\xi).$$
 			
 			\end{enumerate}
 		\end{remark}
 		
 		\begin{lemma}\label{HSb}
 			If $\varphi$ is a homotopy equivalence, then 	$\mathcal{T}_v(p)$ induces an isomorphism from 
 			$\mathrm{Ker}\, d_G/\mathrm{Im}\, d_G$ to $\mathrm{Ker} \,d_{G^\varphi_\varphi}/ \mathrm{Im} \,d_{G^\varphi_\varphi}$.
 		\end{lemma}
 		\begin{proof}
Let $\psi$ be the homotopy inverse of $\varphi$ as in Definition \ref{homotopy}. 		Let $p'\colon X\times I^l\to Y$ be the submersion associated to $\psi$ as above with $p'_{|X\times\{0\}}=\psi$ and let $w$ be a volume form on $I^l$ of volume 1.
Then we have that
\begin{equation*}
\begin{split}
\mathcal{T}_w(p')\circ \mathcal{T}_v(p)(\xi)&= \int_{I^l}w\wedge (p')^*\left(\int_{I^k}v\wedge p^*(\xi))\right)=\\
&=\int_{I^{l}\times I^k}w\wedge v\wedge (p\circ (p'\times 1))^*(\xi)=\\
&=\mathcal{T}_{w\wedge v}(p\circ (p'\times 1)).
\end{split}
\end{equation*}
Since $p\circ (p'\times 1)(x,0,0)=\varphi\circ \psi(x)$, then $p\circ (p'\times 1)$ is homotopic to the identity through morphisms of Lie groupoids which are submersive in the $s$-direction (after possibly stabilizing).
Then, using point 3) and 4) of Remark \ref{HSa}, we obtain that $\mathcal{T}_w(p')\circ \mathcal{T}_v(p)$ is the identity up to boundaries. The same is true for $ \mathcal{T}_v(p)\circ\mathcal{T}_w(p')$.
 \end{proof}
 	
 	\begin{lemma}\label{HSc}
 		There exists a bounded operator $\mathcal{Y}$ of degree $-1$ on $\mathcal{E}(G)\oplus\mathcal{E}(G_\varphi)$ such that
 		\begin{equation}
 		1+\mathcal{T}_v(p)'\mathcal{T}_v(p)= d_G\mathcal{Y}+\mathcal{Y}d_G.
 	   \end{equation}
 	   Furthermore $\mathcal{Y}Dom(d_G)\subset Dom(d_G)$.
 	\end{lemma}
\begin{proof}
	Put $Z:=G^{\varphi}_\varphi\cup G^\varphi$
	For $i=1,2$ define $q_i\colon Z\times I^k\times I^k\to Z\times I^k$ as $(\gamma,\underline{t}_1,\underline{t}_2)\mapsto(\gamma,\underline{t}_i)$. For  $\alpha,\beta$ forms with compact supports on $I^k$ and $\xi\in \mathcal{E}(G^\varphi)\oplus\mathcal{E}(G^\varphi_\varphi)$ it holds, for instance,
	$$(q_2)_*(\xi\wedge\alpha\wedge\beta)=\xi\wedge \left(\int_{I^k}\alpha\right)\wedge\beta.$$
	
	Notice that the adjoint of the map $\mathcal{S}_v\colon \mathcal{E}(G^\varphi\times I^k)\oplus\mathcal{E}(G^\varphi_\varphi\times I^k)$, given by $\xi\mapsto\int_{I^k}v\wedge\xi$, equals $\mathcal{S}_v^*(\xi)=(*v)\wedge\xi$, so that $\mathcal{S}_v'(\xi)=v\wedge\xi$.
	Thus $$\mathcal{S}_v'\mathcal{S}_v(\xi)=v\wedge\left(\int_{I^k}v\right)\wedge \xi=(q_1)_*\left((q_1^*v)\wedge(q_2^*v )\wedge(q_2^*\xi)\right).$$
	
	Define on $ \mathcal{E}(G^\varphi\times I^k)\oplus\mathcal{E}(G^\varphi_\varphi\times I^k)$
	$$\mathcal{Q}_v\colon \xi\mapsto (q_1)_*\left((q_1^*v)\wedge(q_2^*v )\wedge(q_1^*\xi)\right)= v\wedge\xi.$$
	
  Now choose a homotopy of submersions $q\colon Z\times I^k\times I^k\to Z\times I^k$ from $q_1$ to $q_2$ such that $q(\{\gamma\}\times I^k\times I^k)\subset \{\gamma\}\times I^k$ for $\gamma\in Z$.
Then for $\xi\in\mathcal{E}(G^\varphi\times I^k)\oplus\mathcal{E}(G^\varphi_\varphi\times I^k)$
$$
(\mathcal{S}_v'\mathcal{S}_v-\mathcal{Q}_v)(\xi)=(q_1)_*\left((q_1^* v)\wedge(q_2^*v)\wedge(d_{Z\times I^k\times I^k}s_q(\xi)+s_q(d_{Z\times I^k}\xi)\right).
$$
On $\mathcal{E}(G^\varphi)\oplus\mathcal{E}(G^\varphi_\varphi)$ define $\mathcal{Q}_v(p)\colon \xi\mapsto p_*(v\wedge p^*\xi)$. Since $\mathcal{T}_v(p)=\mathcal{S}_v\circ p^*$, we obtain that 
\begin{equation*}
\begin{split}
&(\mathcal{T}_v(p)'\mathcal{T}_v(p)-\mathcal{Q}_v(p))(\xi)=p_*(\mathcal{S}_v'\mathcal{S}_v-\mathcal{Q}_v)(p^*\xi)=\\
&=d_{G^{\varphi}_\varphi}(p_*(q_1)_*((q_1^* v)\wedge(q_2^*v)\wedge s_q(p^*\xi))+p_*(q_1)_*((q_1^* v)\wedge(q_2^*v)\wedge s_q(p^*d_G\xi).
\end{split}
\end{equation*}
It remains to show that $-\mathcal{Q}_v(p)$ is the identity. For that aim, notice that $p\colon (G^\varphi_\varphi\cup G^\varphi)\times I^k\to G_\varphi\cup G$ is the composition of $\varphi\times id\colon (G^\varphi_\varphi\cup G^\varphi)\times I^k\to (G_\varphi\cup G)\times I^k$ and $\overline{q}\colon (G_\varphi\cup G)\times I^k\to (G_\varphi\cup G)$.
First observe that 
$$Q_{G^\varphi_\varphi\times I^k}((\phi\times id)^*\alpha,(\phi\times id)^*\beta)= Q_{G\times I^k}(\alpha,\beta)$$
where $Q_{G^\varphi_\varphi\times I^k}$ and $Q_{G\times I^k}$ are the $C^*_r(G^\varphi_\varphi\times I^k)$-valued and the $C^*_r(G\times I^k)$-valued quadratic forms defined as in \eqref{quadratic},respectively; $\alpha$ and $\beta$ belongs to $\mathcal{E}(G)\oplus\mathcal{E}(G_\varphi)$.
Then it follows that 
\begin{equation}
\begin{split}
Q_G(p_*(v\wedge p^*\alpha),\beta)&=-Q_{G^\varphi_\varphi\times I^k}(v\wedge (\varphi\times id)^* \overline{q}^* \alpha,(\varphi\times id)^* \overline{q}^* \beta)=\\
&=-Q_{G\times I^k}(v\wedge  \overline{q}^* \alpha, \overline{q}^* \beta)=\\
&=-Q_G((\overline{q})_*(v\wedge \overline{q}^*\alpha),\beta).
\end{split}
\end{equation}
So $-\mathcal{Q}_v(p)=\mathcal{Q}(\overline{q})$ which is the identity.
\end{proof} 		

\begin{theorem}[Hilsum-Skandalis]\label{HilSk}
	Let $H\rightrightarrows Y$ and $G\rightrightarrows X$ be two Lie groupoids, with $X$ and $Y$ compact manifolds, and let $\varphi\colon H \to G$ be a homotopy equivalence of groupoids. Let $L$ be the Lie groupoid $G^{\varphi\cup\mbox{id}_X}_{\varphi\cup\mbox{id}_X}$, then
	there exists a path $D_{L,t}^{HS}$ from $D_{L}$ to an invertible operator. Moreover the existence of this path implies that
	\begin{equation}\label{hsgr}
	[\mathcal{E}(H),D_H]\otimes\mu_\varphi=[\mathcal{E}(G),D_G]\in K_*(C^*_r(G)).
	\end{equation}
\end{theorem}

\begin{proof}
	
We want to apply \cite[Lemma 2.1]{HilSk}.	Put $(\mathcal{E}_1,Q_1,D_1)=(\mathcal{E}(G_\varphi) \oplus\mathcal{E}(G),-Q_G, d_G)$ and $(\mathcal{E}_2,Q_2,D_2)=(\mathcal{E}(G^\varphi_\varphi)\oplus\mathcal{E}(G^\varphi),Q_{G^\varphi_\varphi}, d_{G^\varphi_\varphi})$.
	
For this aim consider the operators $\mathcal{T}_v(p)$ and $\mathcal{Y}$ defined previously. Remark \ref{HSa},  Lemma \ref{HSb} and Lemma \ref{HSc} imply that these operators verify the hypotheses of \cite[Lemma 2.1]{HilSk}.
As in the proof of \cite[Lemma 2.1]{HilSk} we can construct an explicit path
$D_{L,t}^{HS}$ from $D_{L}$ to an invertible operator. In particular this means that $[\mathcal{E}(L),D_L]=0$ and the following equality
	\begin{equation*}
	[\mathcal{E}(L),D_L]\otimes\mu_{\varphi\cup\mbox{id}_X}=[\mathcal{E}(H),D_H]\otimes\mu_\varphi-[\mathcal{E}(G),D_G]
	\end{equation*}  
	implies \eqref{hsgr}.
\end{proof}

The following set is a generalisation of the \emph{Structure set} in Surgery Theory, in which we take in account  both the smooth and the groupoid structure of a Lie groupoid $G$.
\begin{definition}
	Let $G\rightrightarrows X$ be a Lie groupoid on a compact smooth manifold. We define the $G$-structure set $\mathcal{S}(G)$ of $X$ as the set 
	\[
	\{\varphi\colon H\to G\,|\, \varphi \mbox{ is a homotopy equivalence of Lie groupoids }\}/\sim,
	\]
	where $(H\rightrightarrows Y,\varphi)\sim(H'\rightrightarrows Y',\varphi')$ if there exist
	\begin{itemize}
		\item a cobordism $W$ with boundary $Y\cup Y'$,
		\item a Lie groupoid $K\rightrightarrows W$,
		transverse to the boundary 
		\item a morphism $\Phi\colon K\to G\times[0,1]\times[0,1]$ such that 
		$\Phi$ is a groupoid homotopy equivalence and, if we restrict it to the boundary, we have that
		$\Phi_{|Y}=\varphi\colon H\to G$ and  $\Phi_{|Y'}=\varphi'\colon H'\to G$.
	\end{itemize} 
	
\end{definition}

If $\varphi\colon H\to G$ is a homotopy equivalence of groupoids, we know that $H=G_\varphi^\varphi$. Let $L$ denote the Lie groupoid $G^{\varphi\cup\mbox{id}_X}_{\varphi\cup\mbox{id}_X}$.
The Signature operator of $L_{ad}$ and
 $D_{L,t}^{HS}$, the path from $D_L$ to an invertible operator given by Theorem \ref{HilSk}, produces an unbounded $\CC$-$C^*_r(L_{ad}^\circ)$-bimodule. Denote 
 by $[D_{L,t}^{HS}]$ the class $\varrho(\sigma(D_L),D_{L,t}^{HS})\in KK^*(\CC,C^*_r(G_{ad}^\circ))$ defined as in Remark \ref{ubrho}.

\begin{definition}\label{rhosignature}
	Let us define the secondary invariant $\varrho(\varphi)$ as the class
	\[
	[D_{L,t}^{HS}]\otimes (\varphi\cup \mbox{id}_X)_!^{ad}\in K_n\left(C^*_r(G_{ad}^\circ)\right),
	\]
	where $n$ is the rank of $\mathfrak{A}G$.
\end{definition}

\begin{proposition}\label{welldefinedrho}
	
	The map
	\[
	\varrho\colon \mathcal{S}_G(X)\to K_n\left(C^*_r(G_{ad}^\circ)\right)
	\] is well defined.
\end{proposition}

\begin{proof} 
	Let $\Phi\colon K\to G\times[0,1]\times[0,1]$ be a groupoid homotopy equivalence with
	$\Phi_{|Y}=\varphi\colon H\to G$ and  $\Phi_{|Y'}=\varphi'\colon H'\to G$.
		Since $\Phi$ is a homotopy equivalence of groupoids, $K(W,\partial W)= G(X\times[0,1],X\times\{0,1\})^\Phi_\Phi$.
		Let $\mathcal{L}\rightrightarrows W\cup X\times[0,1]$ be the pull-back of $G(X\times[0,1],X\times\{0,1\})$ through $\Phi\cup \mbox{id}_{X\times[0,1]}$, let $L$ and $L'$ be the restrictions of $\mathcal{L}$ to $(Y\cup X)\times\{0\}$ and $(Y'\cup X)\times\{1\}$ respectively.
	We have to show that 
	\[[D_{L,t}^{HS}]\otimes (\varphi\cup \mbox{id}_X)_!^{ad}=[D^{HS}_{L',t}]\otimes (\varphi'\cup \mbox{id}_X)_!^{ad}\in K_n\left(C^*_r(G_{ad}^\circ)\right).\]

	Thanks to Theorem \ref{HilSk}, we get a class $[D^{HS}_{\mathcal{L},t}]\in K_{n+1}\left(C^*_r(\mathcal{L}_{ad}^\circ)\right)$.
	The formula in Proposition \ref{APSdelocgr} and the fact that $\Phi$ is a homotopy equivalence
imply that
	\[
	\partial_b\left([D^{HS}_{\mathcal{L},t}]\otimes[ev_{\partial}]\right)=0\in K_{n+1}\left(C^*_r((\mathcal{L}_{\mathring{W}\cup X\times(0,1)})_{ad}^\circ)\right).
	\]
	where $\partial_b$ is defined as in Remark \ref{b-boundary}.
	
	Let $\pi_{X}\colon X\times[0,1]\to X$ the projection onto the first factor.
	If we prove that \[\partial_b\left([D^{HS}_{\mathcal{L},t}]\otimes[ev_{\partial}]\right)\otimes (\Phi_{|\partial W}\cup \mbox{id}_{X\times \{0,1\}})_!^{ad}\otimes (\pi_X)_!^{ad}
	\]
	and
	\[[D_{L,t}^{HS}]\otimes (\varphi\cup \mbox{id}_{X\times\{0\}})_!^{ad}-[D_{L',t}^{HS}]\otimes (\varphi'\cup \mbox{id}_{X\times\{1\}})_!^{ad}\]
	are the same class, we are done.
	By Proposition \ref{functoriality}, we get the following equality
	\[
	\partial_b\left([D_{\mathcal{L},t}^{HS}]\otimes[ev_{\partial}]\right)\otimes (\Phi\cup \mbox{id}_{X\times (0,1)})_!^{ad}=\partial'_b\left([D_{\mathcal{L},t}^{HS}]\otimes[ev_{\partial}]\otimes (\Phi_{|\partial W}\cup \mbox{id}_{X\times \{0,1\}})_!^{ad}\right)
	\]
	where $\partial'_b$ is the boundary map associated to the restriction from the cylinder $X\times[0,1]$ to the boundary $X\times\{0,1\}$.
	
Observe that $[D_{\mathcal{L},t}^{HS}]\otimes[ev_{\partial}]=\left(
	[D_{L,t}^{HS}]\otimes s\right)\oplus\left([D_{L',t}^{HS}]\otimes s\right) $ where $s$ is the generator of $K_1(C^*_r(\RR))$.

	We get that
		\begin{equation*}
		\begin{split}
		\partial_b&\left([D_{\mathcal{L},t}^{HS}]\otimes[ev_{\partial}]\right)\otimes (\Phi\cup \mbox{id}_{X\times (0,1)})_!^{ad}\otimes (\pi_X)_!^{ad}=\\
		&=\partial'_b\left([D_{\mathcal{L},t}^{HS}]\otimes[ev_{\partial}]\otimes (\Phi_{|\partial W}\cup \mbox{id}_{X\times \{0,1\}})_!^{ad}\right)\otimes (\pi_X)_!^{ad}\\
		&=\partial'_b\left(\left((
		[D_{L,t}^{HS}]\otimes\ s)\oplus([D_{L',t}^{HS}]\otimes s)\right)\otimes (\Phi_{|\partial W}\cup \mbox{id}_{X\times \{0,1\}})_!^{ad}\right)\otimes (\pi_X)_!^{ad}=\\
		&=\partial'_b\left(\left(\left([D_{L,t}^{HS}]\otimes (\varphi\cup \mbox{id}_{X\times\{0\}})_!^{ad}\right)\oplus\left([D_{L',t}^{HS}]\otimes (\varphi'\cup \mbox{id}_{X\times\{1\}})_!^{ad}\right) \right)\otimes s\right)\otimes (\pi_X)_!^{ad}=\\
		&=[D_{L,t}^{HS}]\otimes (\varphi\cup \mbox{id}_X)_!^{ad}-[D_{L',t}^{HS}]\otimes (\varphi'\cup \mbox{id}_X)_!^{ad},
		\end{split}
		\end{equation*}
		where we used Proposition \ref{functoriality} for the first equality, and the following calculation for the last one.
Let $\varrho_0$ and $\varrho_1$ denote $ [D_{L,t}^{HS}]\otimes (\varphi\cup \mbox{id}_X)_!^{ad}$ and $[D_{L',t}^{HS}]\otimes (\varphi'\cup \mbox{id}_X)_!^{ad}$ respectively.

\begin{equation*}
\begin{split}
((\varrho_0\oplus\varrho_1)\otimes s)\otimes[\mathrm{id}\otimes\mathrm{ev}_0]^{-1}\otimes(\mathrm{id}\otimes\beta_{ad})\otimes\Delta\otimes (\pi_X)_!^{ad}=\\
=((\varrho_0\oplus\varrho_1)\otimes s_{ad})\otimes(\mathrm{id}\otimes\beta_{ad})\otimes\Delta\otimes (\pi_X)_!^{ad}=\\
=((\varrho_0\oplus\varrho_1)\otimes Bott_{ad})\otimes\Delta\otimes (\pi_X)_!^{ad}
\end{split}
\end{equation*}
where $s_{ad}$ is the generator of $K_1(C^*_r(\RR_{ad}))$ and $Bott_{ad}$ is the generator of $K_0(C^*_r((0,1)\times(0,1)_{ad}))$.
Now, since $\pi_X=\pi\otimes c\colon (X\times\{0,1\})\times(0,1)\to X$, with $c\colon (0,1)\to pt$ the map to the point and $\pi$ the obvious map, it follows that $\Delta\otimes (\pi_X)_!^{ad}=\pi_!^{ad}\otimes c_!^{ad}$.
Finally, using the fact that $Bott_{ad}\otimes  c_!^{ad}=1$  and that $(\varrho_0\oplus\varrho_1)\pi_!^{ad}=\varrho_0-\varrho_1$, where the minus comes from the orientation, we obtain the desired result.
\end{proof}

\subsection{The Dirac operator}
Let $G\rightrightarrows X$ be a Lie groupoid over a compact manifold $X$, with Lie algebroid $\mathfrak{A}(G)\to X$. Let $g$ be a metric on $\mathfrak{A}(G)$, by means of it we can define a $G$-invariant metric on $\ker{ds}$ along the $s$-fibers of $G$. Let $\nabla$ be the fiberwise Levi-Civita connection associated to this metric.
\begin{definition}\label{diracgr}
	Let $\mathrm{Cliff}\left(\mathfrak{A}(G)\right)$ be the Clifford algebra bundle over $X$ associated to the metric $g$.
	Let $S$ be a bundle of Clifford modules over  $\mathrm{Cliff}\left(\mathfrak{A}(G)\right)$ and let $c(X)$ denote the Clifford multiplication by $X\in\mathrm{Cliff}\left(\mathfrak{A}(G)\right)$.
	Assume that $S$ is equipped with a metric $g_S$ and a compatible connection $\nabla^S$ such that:
	\begin{itemize}
		\item the Clifford multiplication is skew-symmetric, that is
		\[
		\langle c(X)s_1,s_2\rangle+\langle s_1,c(X)s_2\rangle=0
		\]
		for all $X\in C^\infty\left(X,\mathfrak{A}(G)\right)$ and $s_1,s_2\in C^\infty(X,S)$;
		\item $\nabla^S$ is compatible with the Levi-Civita connection $\nabla$, namely
		\[
		\nabla^S_X(c(Y)s)=c(\nabla_XY)s+c(Y)\nabla^S_X(s)
		\]
		for all $X,Y\in C^\infty\left(X,\mathfrak{A}(G)\right)$ and $s\in  C^\infty(X,S)$.
	\end{itemize}
	The Dirac operator associated to this is defined as
	\[
	\slashed{D}_S\colon s\mapsto \sum_{\alpha}c(e_\alpha)\nabla^S_\alpha(s)
	\]
	for $s\in C^\infty(X,S)$ and $\{e_\alpha\}_{\alpha\in A}$ a local orthonormal frame.
	
\end{definition}
With this local expression one can easily prove the analogue of the Lichnerowitz-Weitzenb\"{o}ck formula:
\begin{equation}\label{Gdirac}
\slashed{D}_S^2=(\nabla^S)^*\nabla^S+\sum_{\alpha<\beta}c(e_\alpha)c(e_\beta)R(\nabla^S)_{\alpha\beta},
\end{equation}
where $R(\nabla^S)_{\alpha\beta}$ denotes the terms of the curvature of $\nabla^S$.
Assume that the Lie algebroid $\mathfrak{A}(G)$ is $\mathrm{Spin}$, namely it is orientable and its structure group $SO(n)$ can be lifted to the double cover $\mathrm{Spin}(n)$. Moreover we can consider the \emph{spinor} bundle $\slashed{S}$ and denote the associated Dirac operator just by $\slashed{D}$. In this case the second term in \eqref{Gdirac} is equal to $\frac{1}{4}$ of the scalar curvature of $\nabla^{\slashed{S}}$, see \cite[Section 3.3]{PPT}.

\begin{remark}
	The above discussion implies that, if the scalar curvature of $\nabla^{\slashed{S}}$ is uniformly positive everywhere, then the Dirac operator $\slashed{D}$ is invertible. Hence the operator
	$\slashed{D}_{ad}$, defined as in \eqref{adoperator1} and \eqref{adoperator2}, is an unbounded multiplier of $G_{ad}$ that is invertible at $1$.

	Remember that for the Signature operator we need to perturb the operator to an invertible one, whereas in the case of the Dirac operator  we already have the invertiblity condition at $1$ in the adiabatic deformation, thanks to the positivity of the scalar curvature.
\end{remark}

\emph{From now on we will assume that $BG$, the classifying space of $G$, is a manifold and $\mathcal{BG}\rightrightarrows BG$ is the Lie groupoid associated to the universal 1-cocycle $\xi$ (see Appendix \ref{app3} for definitions)}.

We want to define a groupoid version of the Stolz sequence 
\[\xymatrix{\Omega^{\mathrm{spin}}_{n+1}(\mathcal{BG})\ar[r] & \mathrm{R}^{\mathrm{spin}}_{n+1}(\mathcal{BG})\ar[r] & \mathrm{Pos}^{\mathrm{spin}}_{n}(\mathcal{BG})\ar[r] & \Omega^{\mathrm{spin}}_{n}(\mathcal{BG})}\]
(see for instance \cite{Stolz} for the  definition in the case where $G$ is a group).

\begin{definition}\label{defstolzgr}
	Let $G\rightrightarrows X$ be a Lie groupoid. 
	\begin{itemize}
		\item Let $\mathrm{Pos}^{spin}_n(\mathcal{BG})$ be the set of bordism classes of triples $\left(M,f\colon M\to BG, g\right)$. Here $f\colon M\to BG$ is a smooth  map from a smooth closed manifold $M$ such that: $f$ is transverse with respect to $\mathcal{BG}$, $\mathfrak{A}(\mathcal{BG}_f^f)$ is spin of rank $n$ and it is equipped with a metric $g$ with positive scalar curvature.
		
		A bordism between $\left(M,f\colon M\to BG, g\right)$ and $\left(M',f'\colon M'\to BG, g'\right)$ is a triple \[(W,F\colon W\to BG,h),\] where $W$ is a compact smooth manifold with boundary $\partial W=M\sqcup-M'$, a reference map $F$  that restricts to $f$ and $f'$ on the boundary and such that $\mathfrak{A}(\mathcal{BG}_F^F)$ is spin equipped with a metric $h$ with positive scalar curvature, which has a product structure near the boundary and restricts to $g$ and $g'$ on the boundary.
		
		\item Let $R_{n+1}^{spin}(\mathcal{BG})$ be the set of bordism classes $(W,f\colon W\to BG, g)$. Here $W$ is a compact  smooth manifold, possibly with boundary;
		$f\colon W\to BG$ is a smooth map  that is transverse with respect to $\mathcal{BG}$ and such that $\mathfrak{A}(\mathcal{BG}_f^f)$ is spin, of rank $n$ and equipped with a metric $g$ which has product structure near the boundary; the metric $g$ has positive scalar curvature on the boundary.
		
		Two triples $(W,f,g)$ and $(W',f',g')$  are bordant if there exists a bordism \[(N,\varphi\colon N\to BG, h)\] between $(\partial W,f_\partial, g_\partial)$ and $(\partial W',f'_\partial,g'_\partial)$ such that $(\mathfrak{A}(\mathcal{BG}_\varphi^\varphi),h)$ is spin with positive scalar curvature and such that
		\[Y:=W\sqcup_{\partial W}N\sqcup_{\partial W'} W'\]
		is the boundary of a manifold $Z$
		such that the reference map $ F= f\sqcup\varphi\sqcup f'$  extends to a reference map $F'\colon Z\to BG$  and the Lie algebroid of the associated Lie groupoid is spin.
		
		\item Let $\Omega_n^{spin}(\mathcal{BG})$ be the set of bordim classes $(M,f\colon M\to BG)$. Here $M$ is a closed  smooth manifold; $f\colon M \to BG$ is a smooth map that is transverse with respect to $\mathcal{BG}$ and such that $\mathfrak{A}(\mathcal{BG}_f^f)$ is spin of rank $n$. The bordism equivalence between triples is as for $\mathrm{Pos}^{spin}_n(\mathcal{BG})$, without conditions about the metric.
	\end{itemize}

\end{definition}

Thus we obtain a groupoid version of  the Stolz sequence, 
as in the classical case, and we want to build a diagram
\begin{equation}\label{stolzgr}\xymatrix{\Omega^{\mathrm{spin}}_{n+1}(\mathcal{BG})\ar[r]\ar[d]^\beta & \mathrm{R}^{\mathrm{spin}}_{n+1}(\mathcal{BG})\ar[r]\ar[d]^{\mathrm{Ind}_{\mathcal{BG}}} & \mathrm{Pos}^{\mathrm{spin}}_{n}(\mathcal{BG})\ar[r]\ar[r]\ar[d]^\varrho & \Omega^{\mathrm{spin}}_{n}(\mathcal{BG}) \ar[d]^\beta\\
	K_{n+1}\left(C^*_r(\mathfrak{A}(\mathcal{BG}))\right)\ar[r] & K_{n}\left(C^*_r(\mathcal{BG}\times(0,1))\right)\ar[r]^\iota & K_{n}\left(C^*_r(\mathcal{BG}_{ad}^\circ)\right)\ar[r] & K_{n}\left(C^*_r(\mathfrak{A}(\mathcal{BG}))\right)}\end{equation}
such that all the squares commute.

\begin{remark}
	It is left to the reader to check that the first row of the previous diagram is exact. This follows immediately by the definitions: if for instance a cycle $\left(M,f\colon M\to BG, g\right)$ for  $\mathrm{Pos}^{spin}_n(\mathcal{BG})$ becomes trivial in $\Omega^{\mathrm{spin}}_{n}(\mathcal{BG})$, this means that there is a manifold $W$ with boundary $M$ and a map $F\colon W\to BG$ which restricts to $f$ on $M$. Extending  $g$ to any metric on $W$ which has product structure near the boundary we obtain a lift of $\left(M,f\colon M\to BG, g\right)$ in $\mathrm{R}^{\mathrm{spin}}_{n+1}(\mathcal{BG})$.
	Exactness at the other groups of the sequence is proven similarly.
\end{remark}

Let us give the definition of the vertical homomorphisms in \eqref{stolzgr}.

\subsubsection*{Definition of $\beta\colon\Omega^{\mathrm{spin}}_{n}(\mathcal{BG})\to K_{n}\left(\mathfrak{A}(C^*_r(\mathcal{BG}))\right)$}
Let $(M,f\colon M\to BG)$ an element of $\Omega^{spin}_n(\mathcal{BG})$.
Then the Lie algebroid $\mathcal{BG}_f^f$ is spin and, as in Definition \ref{diracgr}, we can define a Dirac operator associated to this spin structure. We will denote it by $\slashed{D}_{f}$ and its symbol $\sigma(\slashed{D}_{f})\in \mathcal{M}(C_0(\mathfrak{A}^*(\mathcal{BG}_f^f)))$ defines a class in $K_n\left(C^*_r(\mathfrak{A}(\mathcal{BG}_f^f))\right)$, by Fourier transform. 
Then $\beta$ is defined as follows
\[
\beta(M,f):= [\hat{\sigma}(\slashed{D}_{f})]\otimes df_!\in K_n\left(C^*_r(\mathfrak{A}(\mathcal{BG}))\right).
\]

It is easy to prove that $\beta$ is well defined.
Indeed if $(M,f)$ and $(M',f')$ are bordant through $(W,F)$, let $E$ denote the dual of the Lie algebroid $\mathfrak{A}(\mathcal{BG}_F^F)$ over $W$, let $\partial E= \mathfrak{A}(\mathcal{BG}_f^f)\times\RR\sqcup \mathfrak{A}(\mathcal{BG}_{f'}^{f'})\times\RR$ be its restriction to the boundary of $W$ and let $\mathring{E}$ be its restriction to the interior of $W$. Then the symbol of the Dirac operator $\slashed{D}_F$ defines a  class $[\sigma(\slashed{D}_F)]$ in the group $K_{n+1}(C_0(E^*))$.
Consider the following commutative diagram:
\[
	\xymatrix{
		K_{*}\left(C_0(E^*)\right)\ar[r]^(.5){\mathrm{ev}_\partial}& K_{*}(C_0(\partial E^*))\ar[d]_(.5){d(f\sqcup f ')_!\otimes Bott}\ar[r]^(.5){\partial} & K_{*}(C_0(\mathring{E}^*))\ar[dl]^{d\mathring{F}_!}\\
		 & K_{*+1} (C_0(\mathfrak{A}^*\mathcal{BG} ))  &  }
\]
where the boundary morphism $\partial$ is  $di_!\otimes Bott$, with $i\colon \partial W\hookrightarrow W$ the inclusion.
Then
  \[[\sigma(\slashed{D}_F)]\otimes \mathrm{ev}_\partial\otimes d(f\sqcup f ')_!\otimes Bott=\beta(M,f)-\beta(M',f').\]
But 
 \[[\sigma(\slashed{D}_F)]\otimes \mathrm{ev}_\partial\otimes d(f\sqcup f ')_!=[\sigma(\slashed{D}_F)]\otimes \mathrm{ev}_\partial\otimes\partial\otimes d\mathring{F}_!=0\]
 by exactness of the top row exact sequence.
 This proves that 
		\[
		\beta(M,F)=\beta(M',f').
		\]

	\subsubsection*{Definition of $\mathrm{Ind}_{\mathcal{BG}}\colon R^{spin}_{n+1}(\mathcal{BG})\to K_{n}\left(C^*_r(\mathcal{BG}\times(0,1))\right) $}
	Let us consider an element 
	$(W,f\colon W\to BG,g) \in R^{spin}_{n+1}(\mathcal{BG})$ and
	the Dirac operator $\slashed{D}_f$. Since we have positive scalar curvature on the boundary, using \eqref{adoperator1}, \eqref{adoperator2}  for $\sigma(\slashed{D}_{\partial f})$ we obtain a class $[\sigma_{nc}(\slashed{D}_f)]\in K_{n+1}(C^*_r\left(\mathcal{T}_{nc}\mathcal{BG}^f_f(W,\partial W)\right))$ that is given by the symbol of $\slashed{D}_f)$ on the interior and by $(\slashed{D}_{\partial f})_{ad}$ on the adiabatic deformation of the boundary. Using the KK-equivalence \ref{bkk}, we obtain a class  
	\begin{equation}\label{y}
	y:=[\hat{\sigma}_{nc}(\slashed{D}_f)]\otimes[\mathrm{ev}_{W_\partial}]^{-1}
	\in K_{n+1}(C^*_r(\mathcal{BG}_f^f(W,\partial W)_{ad}^{\mathcal{F}}))\end{equation}
	(see Definition \ref{nctangent}).
	Hence we can define the map $\mathrm{Ind}_{\mathcal{BG}}$ in the following way
	\[(W,f\colon W\to BG,g)\mapsto y\otimes \mathrm{ev}_{1}\otimes \mu_f\otimes s \in K_{n}\left(C^*_r(\mathcal{BG}\times(0,1))\right),\]
	where
	\begin{itemize}
		\item $\mathrm{ev}_1\colon\mathcal{BG}_f^f(W,\partial W)_{ad}^{\mathcal{F}}\to \mathcal{BG}_f^f(\mathring{W}) $ is the evaluation at $t=1$ in the adiabatic deformation;
		\item $\mu_f$ is the Morita equivalence associated to the pull-back construction;
		\item $s$ is the generator of $K_1(C^*_r(\RR))$.
	\end{itemize} 
	
	This map is well-defined on bordism classes:
	let $(W,f,g)$ and $(W',f',g')$ be two triples in $R^{spin}_{n+1}(\mathcal{BG})$; let $N$,
	$(\mathcal{BG}_\varphi^\varphi,h)$, $(Y,F)$ and $(Z,F')$ be as in Definition \ref{defstolzgr}.
	
	Let  $z\in K_{n+1}\left(C^*_r(\mathcal{BG}_F^F)\right)$  be the $\mathcal{BG}$-index class  of $\slashed{D}_F$. Since $Y$ is a boundary, by Remark \ref{boundaryloc} and \cite[Theorem 4.3]{HilBoundary}, the class $z$ is the zero class.
Let  $i$ be the  inclusion of $\mathring{W}\sqcup\mathring{W}'$ in $Y$.
	As the scalar curvature on $N$ is positive, we have that $\mathrm{ev}_{N}(z)=0$ and then $z$ is an element of $K_{n+1}\left(C^*_r(\mathcal{BG}_f^f(\mathring{W})\sqcup\mathcal{BG}_{f'}^{f'}(\mathring{W}'))\right)$, that is just $K_{n+1}\left(C^*_r(\mathcal{BG}_f^f(\mathring{W}))\right)\oplus K_{n+1}\left(C^*_r(\mathcal{BG}_{f'}^{f'}(\mathring{W'}))\right)$. This element is the direct sum of $\mathrm{ev_1}(y)$ and $-\mathrm{ev_1}(y')$ (the sign $-$ is given by the orientation in the pasting process), namely the indices of $\slashed{D}_f$ and $\slashed{D}_{f'}$ respectively.
	By the definition of $F$, it follows that $\mu_i\otimes\mu_F =\mu_{f\sqcup f'}$. Hence
	\begin{equation*}
	\begin{split}
	\mathrm{Ind}_{\mathcal{BG}}(W,f,g)-&\mathrm{Ind}_{\mathcal{BG}}(W',f',g') = y\otimes \mathrm{ev}_{1}\otimes \mu_f\otimes s -y'\otimes \mathrm{ev}_{1}\otimes \mu_{f'}\otimes s =\\
	&=
(y\otimes \mathrm{ev}_{1}\oplus -y'\otimes \mathrm{ev}_{1})\otimes \mu_{f\sqcup f'}\otimes s=\\
	&
	=
	(y\otimes \mathrm{ev}_{1}\oplus -y'\otimes \mathrm{ev}_{1})\otimes \mu_i\otimes  \mu_{F}\otimes s =
	\\
	& 	=
	z\otimes \mu_{F}\otimes s = 0.
	\end{split}
	\end{equation*}
	
	\subsubsection*{Definition of $\varrho\colon \mathrm{Pos}^{spin}_n(\mathcal{BG})\to K_n\left(C^*_r(\mathcal{BG}^\circ_{ad})\right)$}
	Let $(M,f,g)$ be a triple in $\mathrm{Pos}^{spin}_n(\mathcal{BG})$.
	In this case, since  the algebroid is spin and the scalar curvature is positive, the Dirac operator $\slashed{D}_f$ defines directly a class $\rho(\sigma(\slashed{D}_f),\slashed{D}_f)\in K_n\left(C^*_r((\mathcal{BG}_f^f)^\circ_{ad})\right)$
	associated to the path constantly equal to $\slashed{D}_f$, as in the Remark \ref{ubrho}. Then we define the $\rho$-class as follows:
	\begin{equation}\label{rhodiracgr}
	\varrho(M,f,g):=\varrho(\sigma(\slashed{D}_f),\slashed{D}_f)\otimes f_!^{ad}\in K_n(C^*_r(\mathcal{BG}_{ad}^\circ)).
	\end{equation}
	
	We should check that this map is well-defined, but the proof of this fact is completely analogous to the one of Proposition \ref{welldefinedrho}.
	\begin{proposition}
		Diagram \eqref{stolzgr} is commutative.
\end{proposition}
	\begin{proof}
	The commutativity of the first square is clear.
	The commutativity of the third square is easily obtained using Theorem \ref{th-comm}.
	The only square whose commutativity is not obvious is the second one.
	
	 Consider a cycle $(W,f\colon W\to BG,g)\in \mathrm{R}^{spin}_{n+1}(\mathcal{BG})$.
	 Its image in $\mathrm{Pos}^{spin}_n(\mathcal{BG})$ is given by $(\partial W, \partial f\colon \partial W\to BG,g_{|\partial W})$. 
	 Now $\varrho(\partial W, \partial f,\partial g)$ is given by $\varrho(\sigma(\slashed{D}_{\partial_f}),\slashed{D}_{\partial f})\otimes \partial f_!^{ad}\in K_n(C^*_r(\mathcal{BG}_{ad}^\circ))$.
	 We have to prove that $\iota_*\mathrm{Ind}_{\mathcal{BG}}(W,f,g)=\varrho(\partial W, \partial f,\partial g)$.
	
	 Let $y\in K_{n+1}(C^*_r(\mathcal{BG}_f^f(W,\partial W)_{ad}^{\mathcal{F}}))$ be the element in \eqref{y} used to define $\mathrm{Ind}_{\mathcal{BG}}(W,f,g)$ and notice that its restriction to the boundary, give by $y\otimes[\mathrm{ev}_\partial]$, is equal to the suspension of the $\varrho$-class $\varrho(\sigma(\slashed{D}_{\partial_f}),\slashed{D}_{\partial f})\otimes s$, where $s$ is the generator of $K_1(C^*_r(\RR))$.
	 
	Then we have the following equalities
	\begin{equation*}
	\begin{split}
   & \iota_* \mathrm{Ind}(W,f,g)=\\
    &y\otimes \mathrm{ev}_{1}\otimes \mu_f\otimes s\otimes[\iota]=\\
    &y\otimes \mathrm{ev}_{1}\otimes  s\otimes[\iota]\otimes f_!^{ad}=\\
    &y\otimes[\mathrm{ev}_\partial]\otimes\partial_b\otimes f_!^{ad}=\\
   & (\varrho(\sigma(\slashed{D}_{\partial_f}),\slashed{D}_{\partial f})\otimes s)\otimes\partial_b\otimes f_!^{ad}=\\
   &  (\varrho(\sigma(\slashed{D}_{\partial_f}),\slashed{D}_{\partial f})\otimes s)\otimes [\mathrm{id}\otimes\mathrm{ev_0}]^{-1}\otimes(\mathrm{id}\otimes \beta_{ad})\otimes[\Delta]\otimes [\iota_b]\otimes f_!^{ad}=\\
     & (\varrho(\sigma(\slashed{D}_{\partial_f}),\slashed{D}_{\partial f})\otimes s_{ad})\otimes(\mathrm{id}\otimes \beta_{ad})\otimes[\Delta]\otimes [\iota_b]\otimes f_!^{ad}=\\
          &  (\varrho(\sigma(\slashed{D}_{\partial_f}),\slashed{D}_{\partial f})\otimes Bott_{ad}) \otimes[\Delta]\otimes [\iota_b]\otimes f_!^{ad}=\\
              &    \varrho(\sigma(\slashed{D}_{\partial_f}),\slashed{D}_{\partial f})\otimes Bott_{ad} \otimes[\Delta]\otimes  (\partial f\times \mathrm{id}_{(0,1)})_!^{ad}=\\
                &    \varrho(\sigma(\slashed{D}_{\partial_f}),\slashed{D}_{\partial f})\otimes Bott_{ad} \otimes  (\partial f_!^{ad}\otimes c_!^{ad})=\\
                 &    \varrho(\sigma(\slashed{D}_{\partial_f}),\slashed{D}_{\partial f}) \otimes  \partial f_!^{ad}
  	\end{split}
	\end{equation*}
	 where $\partial_b$ is as in Remark \ref{b-boundary}, $s_{ad}$ and $Bott_{ad} $ are as in the proof of Proposition \ref{welldefinedrho},
	 $c$ is the map from $(0,1)$ to the point. Moreover we used the fact that the inclusion $[\iota_b]$ is equal to the shriek map associated to the inclusion of the open collar neighbourhood of $\partial W$ into $\mathring{W}$, whose composition with $f$ is given by $\partial f\times \mathrm{id}_{(0,1)}$.

	\end{proof}
	
\begin{remark}  
	We can generalise the above construction when $\mathcal{BG}\rightrightarrows BG$ is a limit of smooth manifolds. Assume that the Lie groupoid is transitive, then we can forget about transversality of the maps. Using the Milnor's join construction, we have tat the classifying space $BG$ is the colimit of finite dimensional CW-complexes. Since any finite CW-complex is homotopic to a smooth manifold, it follows that there is an inductive system $\{M_i, f_{ij}, f_i\}$ which defines $BG$ and then one can build the relevant K-groups using inductive limits through lower-shriek maps associated to the maps $f_{ij}$.
The more general case where the groupoid is not transitive presents some more issues about transversality, we will treat it in a different work.

\end{remark}

\subsection{Products}

Let $G\rightrightarrows X$ and $H\rightrightarrows Y$ be two Lie groupoids. In this section we will define an external product between the K-theory of the C*-algebra of the adiabatic deformation of $G$ and the K-theory of C*-algebra of the Lie algebroid of $H$, valued in the C*-algebra of the adiabatic deformation of $G\times H$.
Moreover all along this section we will use minimal tensor product of C*-algebras.

Let us build a KK-class $\alpha\in KK\left(C^*_r(G_{ad}^\circ)\otimes C^*_r(\mathfrak{A}(H)),C^*_r((G\times H)_{ad}^\circ)\right)$ in the following way:
notice that 
\[\mathrm{id}\otimes \mathrm{ev}_0\colon C^*_r(G_{ad}^\circ)\otimes C^*_r(H_{ad})\to C^*_r(G_{ad}^\circ)\otimes C^*_r(\mathfrak{A}(H)) \]
induces a KK-equivalence;
moreover, since $ C^*_r(G_{ad}^\circ)\otimes C^*_r(H_{ad})= C^*_r(G_{ad}^\circ\times H_{ad})$, we have a $C_0([0,1)\times[0,1])$-algebra and the restriction to the diagonal of $[0,1)\times[0,1]$ induces a KK-element 
$[\Delta]\in KK\left(C^*_r(G_{ad}^\circ\times H_{ad}),C^*_r((G\times H)_{ad}^\circ)\right)$.

Thus we can define  the class $\alpha$ as the Kasparov product 
\[[\mathrm{id}\otimes \mathrm{ev}_0]^{-1}\otimes_{C^*_r(G_{ad}^\circ\times H_{ad})}[\Delta]\in KK\left(C^*_r(G_{ad}^\circ)\otimes C^*_r(\mathfrak{A}(H)),C^*_r((G\times H)_{ad}^\circ)\right).\]

\begin{definition}\label{productgr}
	The external product
	\[\boxtimes\colon KK_i\left(\CC,C^*_r(G_{ad}^\circ)\right)\times KK_j\left(\CC,C^*_r(\mathfrak{A}(H))\right)\to KK_{i+j}\left(\CC,C^*_r((G\times H)_{ad}^\circ)\right)\]
	is defined as the map 
	\[x\times y\to(x\otimes_\CC y)\otimes_{D}\alpha,\]
	where $D=C^*_r(G_{ad}^\circ)\otimes C^*_r(\mathfrak{A}(H))$.
\end{definition}
Now let us fix an element $y\in K_n(C^*_r(\mathfrak{A}(H)))$.
We want to investigate the injectivity of the following map
\[
K_*(C^*_r(G_{ad}^\circ))\to K_{*+n}(C^*_r((G\times H)_{ad}^\circ))
\]
 given by the external product with $y$.
 
To do it, let us construct an element \[\beta\in KK\left(C^*_r((G\times H)_{ad}^\circ), C^*_r(G_{ad}^\circ)\otimes C^*_r(H)\right).\]

Let $T$ be the restriction of $G_{ad}^\circ\times H_{ad}$ to
$X\times Y\times \mathfrak{T} $, where $\mathfrak{T}$ is the triangle \[\mathfrak{T}:=\{s\geq t\,|\,(t,s)\in [0,1)\times[0,1] \}.\]

\begin{lemma}
	Denote the restriction map of $T$ to the diagonal side of $\mathfrak{T}$ by $\Delta'$. Then it induces a KK-equivalence
	\[[\Delta']\in KK(C^*_r(T),C^*_r((G\times H)_{ad}^\circ)).\]
\end{lemma}
\begin{proof}
	Observe that $C^*_r(T)$ is a $C_0(\mathfrak{T})$-algebra and that, for this reason, the restriction to the diagonal gives an exact sequence of reduced C*-algebras.
	Indeed the kernel of the restiction morphism turns out to be isomorphic to the C*-algebra $C^*_r(G)\otimes C^*_r(H_{ad}^\circ)\otimes C[0,1)$, that is K-contractible.
	So the only thing to prove is that the restriction admits a completely positive section.
	
	Let $\mathfrak{Q}$ be the set $\{(t,s)\in[0,1]^2\setminus \{(1,1)\}\}$ 
	and let $\lambda\colon X\times Y\times \mathfrak{Q}\to X\times Y\times\mathfrak{T}$
	be the map that sends $(x,y,(t,s))$ to $(x,y,(ts,s))$.
Define $Q$  to be the pull-back groupoid $T^\lambda_\lambda$  over $X\times Y\times\mathfrak{Q}$: notice that  the restriction of $Q$  to $X\times Y\times\{(t,s)\}$ is equal to the restriction of $T$ to 
	$X\times Y\times\{(ts,s)\}$. 
	
	Since  $\lambda$ is proper, we have that $\lambda\colon C_c^\infty(T)\to C_c^\infty (Q)$ is well defined
	and it is an isometry with respect to the reduced C*-norm: this follows by the construction of $Q$ as a pull-back groupoid and by the definition of the reduced norm.  In particular it extends to a *-homomorphism between the reduced C*-algebras. The image of $\lambda^*$ is contained in the *-subalgebra of $C_c^\infty (Q)$ whose elements are functions constant on $\{s=0\}$; let us call $A$ its closure in $C^*_r(Q)$.
	We have the following diagram with exact rows:
	\[
\xymatrix{ 0\ar[r]&C^*_r(T_{\{s\neq0\}})\ar[d]^{\lambda^*}\ar[r]&C^*_r(T)\ar[d]^{\lambda^*}\ar[r]&C^*_r(T_{\{s=0\}})\ar[d]^{\lambda^*}\ar[r]&0\\
	0\ar[r]& A_{\{s\neq0\}}\ar[r]&A\ar[r]&A_{\{s=0\}}\ar[r]&0}.
	\]
Observe that, since the restriction of $\lambda$  to $X\times Y\times(\mathfrak{Q}\setminus\{s=0\})$ is a diffeomorphism, the first vertical arrow is an isomorphism of C*-algebras. By definition of $A$ also the third one is an isomorphism. By the Five Lemma  we have that $C^*_r(T)\cong A$.
	
	Moreover the restriction of $C^*_r(T)$ to the diagonal, corresponds to the restriction of $A$ to the union of the bottom side $\mathfrak{B}$ and the right side $\mathfrak{R}:=\{t=1\}$ of $\mathfrak{Q}$.
	The restriction to $\mathfrak{B}$ of $Q$ is amenable  since it is the product of vector bundles $\mathfrak{A}H\times\mathfrak{A}G$, then the *-homomorphism induced on the reduced C*-algebras admits a completely positive lifting.
	The same is true for the restriction to $\mathfrak{R}$: 
	since $C^*_r(Q)$ contains as ideal $ C^*_r((G\times H)_{ad}^\circ)\otimes C_0(0,1]$ and $\mathfrak{R}$ is $(G\times H)_{ad}^\circ\times\{1\}$, the map $\xi\mapsto t\cdot\xi$ is a completely positive section of the restriction to the right side.
	
	Using Lemma \ref{cplift}, we obtain a completely positive section for the restriction to $\mathfrak{B}\cup\mathfrak{R}$ and then a completely positive section for $\Delta'$.
\end{proof}
Then let us define $\beta$ as the Kasparov product
\[
[\Delta']^{-1}\otimes_{C^*_r(T)}[\mathrm{ev}_{\{s=1\}}]\in KK\left(C^*_r((G\times H)_{ad}^\circ), C^*_r(G_{ad}^\circ)\otimes C^*_r(H)\right).
\]
Let us calculate $\alpha\otimes_{C^*_r\left((G\times H)_{ad}^\circ\right)}\beta\in KK\left(C^*_r(G_{ad}^\circ)\otimes C^*_r(\mathfrak{A}(H)) ,C^*_r(G_{ad}^\circ)\otimes C^*_r(H)\right)$:

\begin{equation*}
\begin{split}
\alpha\otimes\beta&=[\mathrm{id}\otimes\mathrm{ev}_0]^{-1}\otimes[\Delta]\otimes [\Delta']^{-1}[\mathrm{ev}_{\{s=1\}}]=\\
&=[\mathrm{id}\otimes\mathrm{ev}_0]^{-1}\otimes[\mathrm{ev}_{\mathfrak{T}}]\otimes[\mathrm{ev}_{\{s=1\}}]=\\
&=[\mathrm{id}\otimes\mathrm{ev}_0]^{-1}\otimes[\mathrm{id}\otimes\mathrm{ev}_1]=\\
&=\mathrm{id}_{C^*_r(G_{ad}^\circ)}\otimes \mathrm{Ind}_H
\end{split}
\end{equation*}
where $ \mathrm{ev}_{\mathfrak{T}}$ is the restriction from $C^*_r(G_{ad}^\circ)\otimes C^*_r(H_{ad})$ to $C^*_r(T)$ and we used that $\mathrm{ev}_{\mathfrak{T}}\circ\Delta'=\Delta$ and that $\mathrm{ev}_{\{s=1\}}\circ \mathrm{ev}_{\mathfrak{T}}=\mathrm{id}\otimes\mathrm{ev}_1$, moreover
$\mathrm{Ind}_H\in KK\left(C^*_r(\mathfrak{A}(H),C^*_r(H))\right)$ is the index KK-class as in the Remark \ref{adiabaticindex}.

\begin{lemma}\label{injprod}
	Let $y$ be a class in  $K_i\left(\mathfrak{A}(H)\right)$. Assume that there extists a K-homology class $\eta\in KK\left(C^*_r(H),\CC\right)$ such that
	\[y\otimes_{C^*_r(\mathfrak{A}(H))}\mathrm{Ind}_H\otimes_{C^*_r(H)}\eta=n\in \ZZ,\] with $n\neq0$, then
	the map $K_j\left(C^*_r(G_{ad}^\circ)\right)\to K_{i+j}\left(C^*_r((G\times H)_{ad}^\circ)\right)$ given by 
	\[ x\mapsto x\boxtimes y\] 
	is rationally injective.
	If $n=1$, then the map is honestly injective. 
\end{lemma}
\begin{proof}
	From the previous discussion we have that
	
	\begin{equation}
	\begin{split}
	&(x\boxtimes y)\otimes_{C^*_r((G\times H)_{ad}^\circ)}\beta\otimes_{C^*_r(H)}\eta=\\
	&x\otimes_\CC y\otimes _D\alpha\otimes_{C^*_r((G\times H)_{ad}^\circ)}\beta\otimes_{C^*_r(H)}\eta=\\
	&x\otimes_\CC y\otimes_{C^*_r(\mathfrak{A}(H))} \mathrm{Ind}_H\otimes_{C^*_r(H)}\eta=n\cdot x.
	\end{split}
	\end{equation}
	
	So, if we $n$, we have that the exterior product with $y$ is rationally injective and that if $n=1$ it is injective.
\end{proof}

\begin{remark}
If $H$ is the Lie groupoid $\widetilde{Y}\times_\Lambda\widetilde{Y}$ associated to a Galois $\Lambda$-covering $\widetilde{Y}$ of $Y$ and $G=\widetilde{X}\times_\Gamma\widetilde{X}$, we recover \cite[Proposition 5.19]{zenobi}.
Another related result in this context is given by \cite[Corollary 5.8]{zeidler} which is in some sense complementary in the assumptions:  apart from the fact that Zeidler deals with partial secondary invariants, the difference between the two results is that in this paper we do not assume anything on the nature of the manifold $Y$, but we have some assumptions on the group $\Gamma$ (such as for instance K-amenabilty, see \cite[Section 5]{zenobi}); whereas Zeidler assume that the manifold $Y$ is hypereuclidean and that the group $\Lambda$ is any group.
\end{remark}
%
	\subsubsection{Product formulas for secondary invariants}

	Now we would like to apply the product in Definition \ref{productgr} to the $\varrho$-invariant of Definition \ref{rhosignature} and \eqref{rhodiracgr}.
	
	\begin{proposition}
		Let $G\rightrightarrows X$ and $H\rightrightarrows Y$ be two Lie groupoids with oriented Lie algebroid, homotopy equivalent by means of the oriented groupoid morphism $\varphi\colon H \to G$. Let $J\rightrightarrows Z$ be another Lie groupoid with oriented Lie algebroid.
		Consider the secondary invariant $\varrho(\varphi)\in K_i(C^*_r(G_{ad}^\circ))$ and the symbol class of the $J$-signature operator on $Z$, given by $[\sigma_J]\in K_j(C^*_r(\mathfrak{A}(J)))$.
		
		Then we have the following product formula
		\[
		\varrho(\varphi)\boxtimes [\sigma_J]=\varrho(\varphi\times\mathrm{id}_J)\in K_{i+j}\left((G\times J)_{ad}^\circ\right),
		\]
		where $\varphi\times \mathrm{id}_J$ is a homotopy equivalence between $H\times J$ and $G\times J$.
	\end{proposition}
	
	\begin{proof}
		If $L= G_{\varphi\cup\mathrm{id}_X}^{\varphi\cup\mathrm{id}_X}$, then $\varrho(\varphi)=[D_{L,t}^{HS}]\otimes(\varphi\cup\mathrm{id}_X)_!^{ad}$.
		Consequently, following the notations of Definition \ref{productgr}, one has that
		$\varrho(\varphi)\boxtimes [\sigma_J]$ is equal to
		\[
		\Big(\left([D_{L,t}^{HS}]\otimes(\varphi\cup\mathrm{id}_X)_!^{ad}\right)\otimes_\CC[\sigma_J]\Big)\otimes_{D}[\mathrm{id}\otimes\mathrm{ev}_0]^{-1}\otimes_{D'}\Delta,
		\] 
		where $D=C^*_r(G_{ad}^\circ)\otimes C^*_r(\mathfrak{A}(J))$ and
		$D'=C^*_r(G_{ad}^\circ)\otimes C^*_r(J_{ad})$.
		
		That is equal to 
		\[
		\Big(\left([D_{L,t}^{HS}]\otimes_\CC[\sigma_J]\right)\otimes\left((\varphi\cup\mathrm{id}_X)_!^{ad}\otimes(\mathrm{id}_Z)_!\right)\Big)\otimes_{D}[\mathrm{id}\otimes\mathrm{ev}_0]^{-1}\otimes_{D'}\Delta,
		\] 
		Notice that  the following equalities holds:
		
		\begin{itemize}
			\item
			by Remark \ref{rmkshriek} we have that
			\[
			\left((\varphi\cup\mathrm{id}_X)_!^{ad}\otimes(\mathrm{id}_Z)_!\right)\otimes_{D}[\mathrm{id}\otimes\mathrm{ev}_0]^{-1}=[\mathrm{id}\otimes\mathrm{ev}_0]^{-1}\otimes\left((\varphi\cup\mathrm{id}_X)_!^{ad}\otimes(\mathrm{id}_Z)^{ad}_!\right);
			\]
			\item 
			moreover  it is easy to verify that
			\[
			\left((\varphi\cup\mathrm{id}_X)_!^{ad}\otimes(\mathrm{id}_Z)^{ad}_!\right)\otimes \Delta=\Delta\otimes(\varphi\times\mathrm{id}_Z\cup\mathrm{id}_{X\times Z})_!^{ad}.
			\]
		\end{itemize}
		
		Then it turns out that
		\[
		\varrho(\varphi)\boxtimes [\sigma_J]=[D_{L,t}^{HS}]\boxtimes[\sigma_J]\otimes(\varphi\times\mathrm{id}_Z\cup\mathrm{id}_{X\times Z})_!^{ad}.
		\]
		So it only remains to notice that $[D_{L,t}^{HS}]\boxtimes[\sigma_J]= [D_{L\times J}^{HS}]\in K_{i+j}\big(C^*_r((L\times J)_{ad}^\circ)\big)$. This follows from the fact that if $\mathcal{T}$ is the Hilsum-Skandalis perturbation associated to $\varphi$ defined in the proof of Theorem 	\ref{HilSk}, then $\mathcal{T}\otimes\mathrm{id}$ is the Hilsum-Skandalis perturbation associated to $\varphi\times\mathrm{id}_J$.	
	\end{proof}

	One can similarly prove the analogous result for Dirac operators.
	
	\begin{proposition}
		Let $G\rightrightarrows X$ and $H\rightrightarrows Y$ be two Lie groupoids such that both $BG$ and $BH$ are smooth manifolds.
		Let $(M,f,g)$ be a triple in $\mathrm{Pos}^{spin}_n(BG)$ and let $(N,f')$ be an element in $\Omega^{spin}_m(BH)$.
		Then we have that 
		\[
		\varrho(M,f,g)\boxtimes\beta(N,f')=\varrho(M\times N,f\times f',g\oplus h)\in K_{n+m}\left(C^*_r((\mathcal{BG}\times\mathcal{BH})_{ad}^\circ)\right)
		\]
		where $h$ is any metric on $\mathfrak{A}(\mathcal{BH}_{f'}^{f'})$ such that $g\oplus h$ on $\mathfrak{A}(\mathcal{BG}_{f}^{f})\oplus\mathfrak{A}(\mathcal{BH}_{f'}^{f'})$ has positive scalar curvature.
	\end{proposition}

	\section{Foliated bundles} \label{foliated-bundles}
	A general reference for this Section is given by \cite{benameur-piazza}.
	Let $M$ be a smooth compact manifold with fundamental group $\Gamma$ and universal covering $\widetilde{M}$. Let $T$ be a compact manifold such that $\Gamma$ acts by diffeomorphisms on it. 
	Hence $\Gamma$ acts freely and properly on $\widetilde{M}\times T$ by the formula $(\widetilde{m},\theta)\cdot\gamma=(\widetilde{m}\cdot \gamma, \gamma^{-1}\cdot \theta)$. Denote by
	$V$ the quotient.
	
	If $p\colon \widetilde{M}\times T\to V  $ is the natural projection then the leaves of a foliation $\mathcal{F}$ on V are given by the
projections $L_\theta=p(\widetilde{M}\times\{\theta\})$, where $\theta$ runs through the compact manifold $T$.
If $\Gamma(\theta)$ is the isotropy group of $\theta\in T$ then $L_\theta$ is diffeomorphic
to the quotient manifold $\widetilde{M}\times\Gamma(\theta)$.
It is not difficult to see that the monodromy groupoid associated to this foliation is given by $Mon(V,\mathcal{F})=\widetilde{M}\times\widetilde{M}\times T/\Gamma\rightrightarrows V$.
Moreover notice that $C^*(Mon(V,\mathcal{F}))$ is Morita equivalent to $C^*(T\rtimes\Gamma)$, the C*-algebra of the crossed product groupoid $T\rtimes\Gamma\rightrightarrows T$.

\begin{definition}
	Let $F$ be a fundamental domain for the action of $\Gamma$ on $\widetilde{M}$.
	We have the following traces on $C_c^\infty(Mon(V,\mathcal{F}))$ 
	\begin{itemize}
		\item $\tau_{reg}\colon f\mapsto \int_{F\times T} f([\widetilde{m},\widetilde{m}, \theta]) $;
		\item $\tau_{av}\colon f\mapsto \int_{F\times T} \sum_{\gamma\in \Gamma(\theta)}f([\widetilde{m},\widetilde{m}\cdot \gamma, \theta]) $;
	\end{itemize}
	for $f\in C_c^\infty(Mon(V,\mathcal{F}))$.
We also have the following traces on $C_c^\infty(T\rtimes\Gamma)$:
	\begin{itemize}
		\item $\tau_{reg}\colon f\mapsto \int_{T} f(\theta,e) $;
		\item $\tau_{av}\colon f\mapsto \int_{T} \sum_{\gamma\in \Gamma(\theta)}f(\theta,\gamma)$;
	\end{itemize}
	for  $f\in C_c^\infty(T\rtimes\Gamma)$ .
\end{definition}
One can prove that  $\tau_{reg}$ extends to the reduced groupoid C*-algebra and that $\tau_{av}$ extends to the maximal groupoid C*-algebra and that we have the following commutative triangles
\begin{equation}\label{morita-traces}
\xymatrix{K_*(C^*_r(Mon(V,\mathcal{F})))\ar[r]^(.7){\tau_{reg}}\ar[d]_{\mathcal{M}_r}& \RR\\
	            K_*(C^*_r(T\rtimes\Gamma))\ar[ru]_{\tau_{reg}}} \quad\quad \xymatrix{K_*(C^*(Mon(V,\mathcal{F})))\ar[r]^(.7){\tau_{av}}\ar[d]_{\mathcal{M}}& \RR\\
	            K_*(C^*(T\rtimes\Gamma))\ar[ru]_{\tau_{av}}} 
\end{equation}
where $\mathcal{M}_r$ and $\mathcal{M}$ are the Morita equivalences. See \cite[Proposition 2.10]{benameur-piazza}.

\begin{proposition}
	Let $G\rightrightarrows X$ be a Lie groupoid. Then every element in the image of $\mathrm{Ind}_G\colon K_0(C^*(\mathfrak{A}G))\to K_0(C^*(G))$ can be represented by an element in $C^*_c(G)$, with support arbitrarily closed to $X$.
\end{proposition}
\begin{proof}
Recall from \cite[Section 2]{AS1} that $K_0(C^*(\mathfrak{A}G))\simeq K_0(C_0(\mathfrak{A}^*G))$	is generated by 0-order elliptic symbols. Let $\sigma$ be an elliptic symbol on $\mathfrak{A}^*G$, then the operator $P_{ad}$, constructed as in \eqref{adoperator1} and \eqref{adoperator2}
Notice that $P_{ad}$ is in $\Psi_c^0(G_{ad})$ and it is elliptic. By \cite[Proposition 19]{vas} we have that there exists $Q_{ad}\in \Psi_c^0(G_{ad})$ such that $I-P_{ad}Q_{ad}= R_{ad}$ and $I-Q_{ad}P_{ad}=S_{ad}$ with $R_{ad},S_{ad}\in C_c^{\infty}(G_{ad})$. 

Recall that the odd operator $\tiny{\begin{pmatrix}0 & P_{ad}\\ Q_{ad}& 0\end{pmatrix}}$ represents the image of $[\hat{\sigma}]$
through the KK-equivalence $[\mathrm{ev}_0]^{-1}\colon KK(\CC,C^*(\mathfrak{A}G))\to KK(\CC,C^*(G_{ad}))$.
Moreover $\mathrm{ev}_t\colon C^*(G_{ad})\to C^*(G)$ induces the same homomorphism in K-theory for $t\in (0,1]$. Because of the topology of $G_{ad}$, a compact set in $G_{ad}$ concentrates near $X$ when $t$ goes to 0. Then the support of $R_{ad}$ and $S_{ad}$ concentrates near the diagonal for $t$ small. 

Now observe that the index class in $K_0(C^*(G))$ is given by $[p_t]-[q]$, where
\[
p_t=\begin{pmatrix}R_t^2 & R_t(1+R_t)Q_t\\ S_tP_t& 1-S_t^2\end{pmatrix} \quad\quad q=\begin{pmatrix}0 & 0\\ 0& 1\end{pmatrix}
\] 
and $T_t$ is the evaluation at $t$ of $T_{ad}$, for $T= P,Q,R,S$.
Taking $t\neq0$ arbitrarily small, we obtain the desired result.
\end{proof}

\begin{corollary}
	Let $G$ be the monodromy groupoid $Mon(V,\mathcal{F})$.
	If $x\in K_0(C^*(G))$ is in the image of $\mathrm{Ind}_G$, then 
	\[
	\tau_{av}(x)=\tau_{reg}(x).
	\]
\end{corollary}

\begin{proof}
	By the previous proposition, since we can chose as a representative of $x$ an element $f\in C_c^\infty(Mon(V,\mathcal{F}))$ with support arbitrarily closed to the diagonal, the only contribution in the sum in $\tau_{av}$ is given by the identity element of 
	$\Gamma$. Then the corollary follows.
\end{proof}

\begin{remark}
	This corollay is the analogous of \cite[Theorem A.1]{WY}, but applied the monodromy groupoid of foliated bundles and what follows is directly inspired by the methods used in that article.
\end{remark}

Now assume that $\gamma\in \Gamma$ is torsion of order $n$ and that $\gamma\in \Gamma(\theta)$ for all $\theta\in T$. The element 
\[p_\gamma:=\frac{1}{n}\sum_{i=1}^n\delta_{\gamma^{i}}\]
is a projection in $C^*(T\rtimes\Gamma)$ such that $\tau_{reg}(p_\gamma)= \frac{1}{n}$ and $\tau_{av}(p_\gamma)=1$.  By \eqref{morita-traces}, it follows that $\mathcal{M}^{-1}[p_\gamma]\in K_0(C^*(Mon(V,\mathcal{F})))$ is not in the image of $\mathrm{Ind}_{Mon(V,\mathcal{F})}$.

Let   $\iota\colon C^*(Mon(V,\mathcal{F}))\otimes C_0(0,1)\to C^*(Mon(V,\mathcal{F})_{ad}^\circ)$ be the obvious inclusion, then the previous discussion implies that
$\iota_*(Bott^{-1}(\mathcal{M}^{-1}[p_\gamma]))$ is not the zero class, because \eqref{AES} is exact.

Let us consider the two following geometric situation.
\begin{itemize}
	\item Let $(W, f\colon W\to V, g)$ be an element in $R^{spin}_{2n}(Mon(V,\mathcal{F}))$ such that 
 $\partial W\neq\emptyset$ and	$\mathrm{Ind}_{Mon(V,\mathcal{F})}(W, f\colon W\to V, g)=Bott^{-1}(\mathcal{M}^{-1}[p_\gamma])$. Then, by the commutativity of \eqref{stolzgr},
 $\varrho(\partial W, f_{|\partial W}\colon \partial W\to V, g_{\partial W})$ is non zero.
 If in particular $W=V\times[0,1]$, $(V,\mathcal{F})$ is a spin foliation and $f$ is the projection on $V$, this would imply that $\varrho(g_0)\neq\varrho(g_1)$ and that $Pos^{spin}_{n}(Mon(V,\mathcal{F}))$ has infinite elements.
 
 \item Assume that $(V,\mathcal{F})$ is an oriented foliation such that the dimension of the leaves is odd. Let $W$ be a cobordism between $V$ and a manifold $V'$ and let $\mathcal{H}$ an oriented foliation on $W$, transverse to the boundary, such that $\mathcal{H}_{|V}=\mathcal{F}$. Assume that there is a function $F\colon W\to V\times[0,1]$ that is transverse to the foliation $(V\times[0,1], \mathcal{F}\times[0,1])$, that is the pull-back foliation of $\mathcal{F}$ along the projection  $V\times[0,1]\to V$. Moreover suppose that $F_{|V'}$ is a foliated homotopy equivalence and that $F_{|V}$ is the identity on $V$. 
 
Consider the signature $G$-operator $D$ where $G$ is the groupoid $Mon(V\times[0,1],\mathcal{F}\times[0,1])^{F\times\mathrm{id}}_{F\times\mathrm{id}}\rightrightarrows W\cup V\times[0,1]$. Because on the boundary we have that $F$ is a homotopy equivalence, we can perturb the signature operator and consider the push-forward of its $G$-index in $K_0(C^*(Mon(V,\mathcal{F}))\times C_0(0,1)$ is $Bott^{-1}(\mathcal{M}^{-1}[p_\gamma])$. 
Using Theorem \ref{APSdelocgr} we obtain that $\varrho(F_{|V'})-\varrho(\mathrm{id}_V)= \iota_*(Bott^{-1}(\mathcal{M}^{-1}[p_\gamma]))\neq 0$. 

\end{itemize}

The construction of $(W, f\colon W\to V, g)\in R^{spin}_{2n}(Mon(V,\mathcal{F}))$ or $F\colon W\to V\times[0,1]$ such that the index of the associated Dirac and signature operators have index equal to $Bott^{-1}(\mathcal{M}^{-1}[p_\gamma])$ and trying to do it for more general foliations constitute an open problem that the author will tackle in a future work.

We discussed here this question in order to show one of the many geometric situations that could be investigated with the methods presented in the this work.

\begin{appendices}

\section{ Classifying spaces and 1-cocycles}
\label{app3}
In this appendix we are going to recall some basic construction from \cite{Haefliger,HilSk2}.
Let $G\rightrightarrows X$ be a topological groupoid. 

\begin{definition}
	Let $Y$ a topological space and $\{U_i\}_{i\in I}$ an open cover of $Y$.

		A 1-cocycle with values in $G$, defined on the cover $\{U_i\}_{i\in I}$ is the data  of a continuous application \[\lambda_{ij}\colon U_i\cap U_j\to G \] for any pair $(i,j)$,
		in such a way that, if $y\in U_i\cap U_j\cap U_k$, then
		$\lambda_{ij}(y)$ is composable with $\lambda_{jk}(y)$ and
		\[\lambda_{ik}(y)=\lambda_{ij}(y)\lambda_{jk}(y).\]
	We will say that two 1-cocycles $(\{U_i\}_{i\in I},\lambda_{ij})$ and $(\{U'_k\}_{k\in K},\lambda_{kj})$ are cohomologous if, for each $i\in I$ and $k\in K$, there are continuous maps 
	$\mu_{ik}\colon U_i\cap U'_k\to G$ such that
	\[
	\mu_{ik}(x)\lambda'_{kl}(x)=\mu_{il}(x)\quad \mbox{and}\quad \lambda_{ji}(y)\mu_{ik}(y)=\mu_{il}(y)
	\]
	for $x\in U_i\cap U'_k\cap U'_l$ and $y\in U_i\cap U_k\cap U'_l$.
	Let $H^1(Y,G)$ denote the set of the cohomology classes of $G$-valued 1-cocycles.
\end{definition}

If $\varphi\colon Y'\to Y$ is a continuous map, then we have a natural map $\varphi^*\colon H^1(Y,G)\to H^1(Y',G)$ that associates to a 1-cocycle $(\lambda_{ij},\{U_i\}_{i\in I})$
the pull-back $(\lambda_{ij}\circ \varphi,\{\varphi^{-1}(U_i)\}_{i\in I})$.

\begin{remark}
	If $(\lambda_{ij},\{U_i\}_{i\in I})$
	is a $G$-valued 1-cocycle on $Y$, then $\lambda_{ii}$ takes values in $G^{(0)}=X $ for any $i\in I$ and $\lambda_{ij}(x)=\lambda_{ji}^{-1}(x)$ for any $i,j\in I$ and $x\in U_i\cap U_j$.
\end{remark}

Let $(\lambda_{ij},\{U_i\}_{i\in I})$
be a $G$-valued 1-cocycle on $Y$, then one can canonically construct a groupoid $G_\lambda^\lambda$ over $Y$ in the following way:
\begin{itemize}
	\item take the disjoint union
	$\bigsqcup_i U_i$ of all the open sets of the cover;
	\item consider the map $\Lambda\colon \bigsqcup_i U_i\to X$ given by $\lambda_{ii}$ on each $U_i$;
	\item build the pull-back groupoid $G^{\Lambda}_{\Lambda}=\bigsqcup_{i,j} U_i\times_X G\times_X U_j $;
	\item finally define $G_\lambda^\lambda$ as the quotient of $G^\Lambda_\Lambda$ by the following equivalence relation:
	$$(y_i,\gamma, y_j)\sim (y_k,\gamma',y_h)$$ with $(y_i,\gamma, y_j)\in U_i\times_X G\times_X U_j$ and $(y_k,\gamma', y_h)\in U_k\times_X G\times_X U_h$, if $y_i=y_k\in U_i\cap U_k$, $y_j=y_h\in U_j\cap U_h$ and 
	$\gamma'=\lambda_{ki}(y_i)\gamma\lambda_{jh}(y_j)$.
	
\end{itemize}
The isomorphism class of the groupoid $G^\lambda_{\lambda}\rightrightarrows Y$ depends only on the cohomology class of $\lambda$.
Notice that the groupoid $G\rightrightarrows X$ itself is associated to the cocycle $\lambda\in H^1(X,G)$ given by the identity on $X$.

In the literature there are many equivalent definition of the classifying space $BG$ of $G$.
Here we will take as definition the following proposition.

\begin{proposition}\cite[3.1.1]{Haefliger}
	There exists a unique space $BG$ up to homotopy, equipped with a universal 1-cocycle $\xi\in H^1(BG,G)$ such that for any 1-cocycle $\lambda\in H^1(Y,G)$ on a paracompact topological space $Y$, there exists a unique function $f\colon Y\to G$ up to  homotopy such that
	\[
	\lambda=f^*\xi\in H^1(Y,G).
	\]
\end{proposition}
Let $\mathcal{BG}\rightrightarrows BG$ be the groupoid associated to the 1-cocycle $\xi\in H^1(BG,G)$.
One can easily check, by functoriality of the construction, that for any $f\colon Y\to BG$ the groupoid 
$G^{f^*\xi}_{f^*\xi}\rightrightarrows Y$ is isomorphic to $\mathcal{BG}_f^f\rightrightarrows Y$, the pull-back of $\mathcal{BG}$ along $f$.

\begin{remark}
	Let $\lambda\in H^1(Y,G)$ be represented by  $(\lambda_{ij},\{U_i\}_{i\in I})$ and let $\{\alpha_i\}_{i\in I}$ a partition of the unity associated to a locally finite cover $\{U_i\}_{i\in I}$. 
	Then the function $f\colon Y\to BG$ such that $G_\lambda^\lambda\cong \mathcal{BG}_f^f$ is given by 
	\[
	\sum_{i\in I}\alpha_i\lambda_{ii}\colon Y\to BG.
	\]
\end{remark}

Now let us assume that $G$ is a Lie groupoid.
A smooth 1-cocycle $(\lambda_{ij},\{U_i\}_{i\in I})$ is defined to be transverse if $\lambda_{ii}$ is transverse for any $i\in I$.
It is clear that if $\lambda$ is a smooth and transverse 1-cocycle, then 
$G^\lambda_\lambda$ is a Lie groupoid.
In particular if $BG$ is a smooth manifold, then the function $f\colon Y\to BG $ can be chosen to be a smooth transition function tranverse with respect to $\mathcal{BG}$.

\end{appendices}
\addcontentsline{toc}{section}{References}
\bibliographystyle{plain}
\nocite{*}
\bibliography{biblip}

\begin{thebibliography}{10}

\bibitem{AS1}
M.~F. Atiyah and I.~M. Singer.
\newblock The index of elliptic operators. {I}.
\newblock {\em Ann. of Math. (2)}, 87:484--530, 1968.

\bibitem{BJ}
Saad Baaj and Pierre Julg.
\newblock Th\'eorie bivariante de {K}asparov et op\'erateurs non born\'es dans
  les {$C^{\ast} $}-modules hilbertiens.
\newblock {\em C. R. Acad. Sci. Paris S\'er. I Math.}, 296(21):875--878, 1983.

\bibitem{benameur-roy-ell2}
M.~Benameur and I.~Roy.
\newblock The {H}igson-{R}oe exact sequence and $\ell^2$ eta invariants.
\newblock {\em J. Funct. Anal.}, 4(268):974--1031, 2015.

\bibitem{benameur-piazza}
Moulay-Tahar Benameur and Paolo Piazza.
\newblock Index, eta and rho invariants on foliated bundles.
\newblock {\em Ast\'erisque}, (327):201--287 (2010), 2009.

\bibitem{BR}
Moulay-Tahar Benameur and Indrava Roy.
\newblock Leafwise homotopies and {H}ilbert-{P}oincar\'e complexes {I}.
  {R}egular {HP} complexes and leafwise pull-back maps.
\newblock {\em J. Noncommut. Geom.}, 8(3):789--836, 2014.

\bibitem{CS}
A.~Connes and G.~Skandalis.
\newblock The longitudinal index theorem for foliations.
\newblock {\em Publ. Res. Inst. Math. Sci.}, 20(6):1139--1183, 1984.

\bibitem{CunSk}
J.~Cuntz and G.~Skandalis.
\newblock Mapping cones and exact sequences in {$KK$}-theory.
\newblock {\em J. Operator Theory}, 15(1):163--180, 1986.

\bibitem{C87}
Joachim Cuntz.
\newblock A new look at {$KK$}-theory.
\newblock {\em $K$-Theory}, 1(1):31--51, 1987.

\bibitem{DL}
Claire Debord and Jean-Marie Lescure.
\newblock Index theory and groupoids.
\newblock In {\em Geometric and topological methods for quantum field theory},
  pages 86--158. Cambridge Univ. Press, Cambridge, 2010.

\bibitem{DS}
Claire Debord and Georges Skandalis.
\newblock Adiabatic groupoid, crossed product by {$\Bbb{R}_+^\ast$} and
  pseudodifferential calculus.
\newblock {\em Adv. Math.}, 257:66--91, 2014.

\bibitem{D-Sk-stability}
Claire Debord and Georges Skandalis.
\newblock Stability of {L}ie groupoid {$C^*$}-algebras.
\newblock {\em J. Geom. Phys.}, 105:66--74, 2016.

\bibitem{DG2}
Robin~J. Deeley and Magnus Goffeng.
\newblock Realizing the analytic surgery group of {H}igson and {R}oe
  geometrically part {II}: relative {$\eta$}-invariants.
\newblock {\em Math. Ann.}, 366(3-4):1319--1363, 2016.

\bibitem{DG3}
Robin~J. Deeley and Magnus Goffeng.
\newblock Realizing the analytic surgery group of {H}igson and {R}oe
  geometrically part {III}: higher invariants.
\newblock {\em Math. Ann.}, 366(3-4):1513--1559, 2016.

\bibitem{DG1}
Robin~J. Deeley and Magnus Goffeng.
\newblock Realizing the analytic surgery group of {H}igson and {R}oe
  geometrically, part {I}: the geometric model.
\newblock {\em J. Homotopy Relat. Struct.}, 12(1):109--142, 2017.

\bibitem{Haefliger}
Andr{\'e} Haefliger.
\newblock Groupo\"\i des d'holonomie et classifiants.
\newblock {\em Ast\'erisque}, (116):70--97, 1984.
\newblock Transversal structure of foliations (Toulouse, 1982).

\bibitem{HLS}
N.~Higson, V.~Lafforgue, and G.~Skandalis.
\newblock Counterexamples to the {B}aum-{C}onnes conjecture.
\newblock {\em Geom. Funct. Anal.}, 12(2):330--354, 2002.

\bibitem{HRk}
Nigel Higson and John Roe.
\newblock {\em Analytic {$K$}-homology}.
\newblock Oxford Mathematical Monographs. Oxford University Press, Oxford,
  2000.
\newblock Oxford Science Publications.

\bibitem{HigRoeI}
Nigel Higson and John Roe.
\newblock Mapping surgery to analysis. {I}. {A}nalytic signatures.
\newblock {\em $K$-Theory}, 33(4):277--299, 2005.

\bibitem{HigRoeII}
Nigel Higson and John Roe.
\newblock Mapping surgery to analysis. {II}. {G}eometric signatures.
\newblock {\em $K$-Theory}, 33(4):301--324, 2005.

\bibitem{HigRoeIII}
Nigel Higson and John Roe.
\newblock Mapping surgery to analysis. {III}. {E}xact sequences.
\newblock {\em $K$-Theory}, 33(4):325--346, 2005.

\bibitem{HilBoundary}
Michel Hilsum.
\newblock Hilbert modules of foliated manifolds with boundary.
\newblock In {\em Foliations: geometry and dynamics ({W}arsaw, 2000)}, pages
  315--332. World Sci. Publ., River Edge, NJ, 2002.

\bibitem{HilSk2}
Michel Hilsum and Georges Skandalis.
\newblock Morphismes {$K$}-orient\'es d'espaces de feuilles et fonctorialit\'e
  en th\'eorie de {K}asparov (d'apr\`es une conjecture d'{A}. {C}onnes).
\newblock {\em Ann. Sci. \'Ecole Norm. Sup. (4)}, 20(3):325--390, 1987.

\bibitem{HilSk}
Michel Hilsum and Georges Skandalis.
\newblock Invariance par homotopie de la signature \`a coefficients dans un
  fibr\'e presque plat.
\newblock {\em J. Reine Angew. Math.}, 423:73--99, 1992.

\bibitem{JT}
Kjeld~Knudsen Jensen and Klaus Thomsen.
\newblock {\em Elements of {$KK$}-theory}.
\newblock Mathematics: Theory \& Applications. Birkh\"auser Boston Inc.,
  Boston, MA, 1991.

\bibitem{kasp}
G.~G. Kasparov.
\newblock Equivariant {$KK$}-theory and the {N}ovikov conjecture.
\newblock {\em Invent. Math.}, 91(1):147--201, 1988.

\bibitem{Lance}
E.~C. Lance.
\newblock {\em Hilbert {$C^*$}-modules}, volume 210 of {\em London Mathematical
  Society Lecture Note Series}.
\newblock Cambridge University Press, Cambridge, 1995.
\newblock A toolkit for operator algebraists.

\bibitem{McK}
K.~Mackenzie.
\newblock {\em Lie groupoids and {L}ie algebroids in differential geometry},
  volume 124 of {\em London Mathematical Society Lecture Note Series}.
\newblock Cambridge University Press, Cambridge, 1987.

\bibitem{mack}
Kirill C.~H. Mackenzie.
\newblock {\em General theory of {L}ie groupoids and {L}ie algebroids}, volume
  213 of {\em London Mathematical Society Lecture Note Series}.
\newblock Cambridge University Press, Cambridge, 2005.

\bibitem{MM}
I.~Moerdijk and J.~Mr{\v{c}}un.
\newblock {\em Introduction to foliations and {L}ie groupoids}, volume~91 of
  {\em Cambridge Studies in Advanced Mathematics}.
\newblock Cambridge University Press, Cambridge, 2003.

\bibitem{these-month}
Bertrand Monthubert.
\newblock {T}h\`ese de doctorat : {G}roupo\"ides et calcul
  pseudo-diff\'erentiel sur les vari\'et\'es \`acoins.
\newblock
  https://www.math.univ-toulouse.fr/{\raisebox{-.6ex}{\symbol{"7E}}}monthube/recherche.php.

\bibitem{Mont}
Bertrand Monthubert.
\newblock Groupoids and pseudodifferential calculus on manifolds with corners.
\newblock {\em J. Funct. Anal.}, 199(1):243--286, 2003.

\bibitem{MP}
Bertrand Monthubert and Fran{\c{c}}ois Pierrot.
\newblock Indice analytique et groupo\"\i des de {L}ie.
\newblock {\em C. R. Acad. Sci. Paris S\'er. I Math.}, 325(2):193--198, 1997.

\bibitem{MRW}
Paul~S. Muhly, Jean~N. Renault, and Dana~P. Williams.
\newblock Equivalence and isomorphism for groupoid {$C^\ast$}-algebras.
\newblock {\em J. Operator Theory}, 17(1):3--22, 1987.

\bibitem{NWX}
Victor Nistor, Alan Weinstein, and Ping Xu.
\newblock Pseudodifferential operators on differential groupoids.
\newblock {\em Pacific J. Math.}, 189(1):117--152, 1999.

\bibitem{PPT}
M.~J. Pflaum, H.~Posthuma, and X.~Tang.
\newblock The index of geometric operators on {L}ie groupoids.
\newblock {\em Indag. Math. (N.S.)}, 25(5):1135--1153, 2014.

\bibitem{PS}
Paolo Piazza and Thomas Schick.
\newblock Rho-classes, index theory and {S}tolz' positive scalar curvature
  sequence.
\newblock {\em J. Topol.}, 7(4):965--1004, 2014.

\bibitem{PS2}
Paolo Piazza and Thomas Schick.
\newblock The surgery exact sequence, {K}-theory and the signature operator.
\newblock {\em Ann. K-Theory}, 1(2):109--154, 2016.

\bibitem{PZ}
Paolo Piazza and Vito~Felice Zenobi.
\newblock Singular spaces, goupoids and metrics of positive scalar curvature.
\newblock arXiv:1803.02697.

\bibitem{Renault}
Jean Renault.
\newblock {\em A groupoid approach to {$C^{\ast} $}-algebras}, volume 793 of
  {\em Lecture Notes in Mathematics}.
\newblock Springer, Berlin, 1980.

\bibitem{CLM}
P.~Carrillo Rouse, J.~M. Lescure, and B.~Monthubert.
\newblock A cohomological formula for the {A}tiyah--{P}atodi--{S}inger index on
  manifolds with boundary.
\newblock {\em J. Topol. Anal.}, 6(1):27--74, 2014.

\bibitem{Schick}
Thomas Schick.
\newblock The topology of positive scalar curvature.
\newblock In {\em Proceedings of the {I}nternational {C}ongress of
  {M}athematicians---{S}eoul 2014. {V}ol. {II}}, pages 1285--1307. Kyung Moon
  Sa, Seoul, 2014.

\bibitem{SkExt}
Georges Skandalis.
\newblock On the strong {E}xt bifunctor.
\newblock
  http://webusers.imj-prg.fr/{\raisebox{-.6ex}{\symbol{"7E}}}georges.skandalis/Publications/StrongExt.pdf.

\bibitem{SkExact}
Georges Skandalis.
\newblock Exact sequences for the {K}asparov groups of graded algebras.
\newblock {\em Canad. J. Math.}, 37(2):193--216, 1985.

\bibitem{SkPoly}
Georges Skandalis.
\newblock C*-alg\`ebres, {A}lg\'ebres de von {N}eumann, {E}xemples.
\newblock {\em
  https://webusers.imj-prg.fr/{\raisebox{-.6ex}{\symbol{"7E}}}georges.skandalis/poly2015.pdf},
  2015.

\bibitem{Stolz}
Stephan Stolz.
\newblock Positive scalar curvature metrics---existence and classification
  questions.
\newblock In {\em Proceedings of the {I}nternational {C}ongress of
  {M}athematicians, {V}ol.\ 1, 2 ({Z}\"urich, 1994)}, pages 625--636.
  Birkh\"auser, Basel, 1995.

\bibitem{JLT1}
Jean~Louis Tu.
\newblock La conjecture de {N}ovikov pour les feuilletages hyperboliques.
\newblock {\em $K$-Theory}, 16(2):129--184, 1999.

\bibitem{JLT2}
Jean-Louis Tu.
\newblock The gamma element for groups which admit a uniform embedding into
  {H}ilbert space.
\newblock In {\em Recent advances in operator theory, operator algebras, and
  their applications}, volume 153 of {\em Oper. Theory Adv. Appl.}, pages
  271--286. Birkh\"auser, Basel, 2005.

\bibitem{vas}
St{\'e}phane Vassout.
\newblock Unbounded pseudodifferential calculus on {L}ie groupoids.
\newblock {\em J. Funct. Anal.}, 236(1):161--200, 2006.

\bibitem{Wprod}
Charlotte Wahl.
\newblock Product formula for {A}tiyah-{P}atodi-{S}inger index classes and
  higher signatures.
\newblock {\em J. K-Theory}, 6(2):285--337, 2010.

\bibitem{Wahl}
Charlotte Wahl.
\newblock Higher {$\rho$}-invariants and the surgery structure set.
\newblock {\em J. Topol.}, 6(1):154--192, 2013.

\bibitem{WXY}
Shmuel Weinberger, Zhizhang Xie, and Guoliang Yu.
\newblock Additivity of higher rho invariants and nonrigidity of topological
  manifolds.
\newblock Preprint 2016. arXiv 1608.03661.

\bibitem{WY}
Shmuel Weinberger and Guoliang Yu.
\newblock Finite part of operator {$K$}-theory for groups finitely embeddable
  into {H}ilbert space and the degree of nonrigidity of manifolds.
\newblock {\em Geom. Topol.}, 19(5):2767--2799, 2015.

\bibitem{XY}
Zhizhang Xie and Guoliang Yu.
\newblock Positive scalar curvature, higher rho invariants and localization
  algebras.
\newblock {\em Adv. Math.}, 262:823--866, 2014.

\bibitem{XY2}
Zhizhang Xie and Guoliang Yu.
\newblock Higher rho invariants and the moduli space of positive scalar
  curvature metrics.
\newblock {\em Adv. Math.}, 307:1046--1069, 2017.

\bibitem{zeidler}
Rudolf Zeidler.
\newblock Positive scalar curvature and product formulas for secondary index
  invariants.
\newblock {\em J. Topol.}, 9(3):687--724, 2016.

\bibitem{zenobi}
Vito~Felice Zenobi.
\newblock Mapping the surgery exact sequence for topological manifolds to
  analysis.
\newblock {\em J. Topol. Anal.}, 9(2):329--361, 2017.

\end{thebibliography}

\end{document}